\documentclass{amsart}

 \usepackage{mathtools}
 \usepackage{amsthm,amssymb,mathrsfs,enumitem,wasysym,hyperref}

 \usepackage{esint,enumitem}

 \usepackage[numbers,sort&compress]{natbib}
 \usepackage{mathscinet,todonotes} 

\usepackage{tikz,mathrsfs}
\usetikzlibrary{arrows,shapes}
\usetikzlibrary{decorations.pathreplacing}

 \numberwithin{equation}{section}
 \numberwithin{figure}{section}
 \theoremstyle{plain}
 \newtheorem{thm}{Theorem}[section]
   \theoremstyle{definition}
   \newtheorem{defn}[thm]{Definition}
   \newtheorem*{defn*}{Definition}
   \theoremstyle{remark}
   \newtheorem*{rem*}{Remark}
   \theoremstyle{plain}
   \newtheorem{lem}[thm]{Lemma}
   \theoremstyle{plain}
   \newtheorem{prop}[thm]{Proposition}
   
   \theoremstyle{definition}

   \newcounter{step}

 \makeatletter
 \newcommand{\norm}{\@ifstar{\@normb}{\@normi}}
 \newcommand{\@normb}[2]{\left\Vert{#1}\right\Vert_{#2}}
 \newcommand{\@normi}[2]{\Vert{#1}\Vert_{#2}}
 \makeatother

 \global\long\def\tSob#1{{H}^{#1}}
 \global\long\def\thSob#1{\dot{H}^{#1}}

 \global\long\def\Leb#1{L_{#1}}

 \newcommand{\action}[1]{\left<#1 \right>}

 \RequirePackage{bm}

 \newcommand{\boldb}{\bm{b}}
 \newcommand{\boldc}{\bm{c}}
 \newcommand{\boldu}{\bm{u}}
 \newcommand{\boldw}{\bm{w}}
 
 \newcommand{\boldF}{\mathbf{F}}
 \newcommand{\boldG}{\mathbf{G}}
 \newcommand{\boldv}{\bm{v}}

 \DeclareMathOperator{\Div}{div}
 \newcommand{\relphantom}[1]{\mathrel{\phantom{#1}}}
 
 \newcommand{\myd}[1]{\,d{#1}}

 \newcommand{\Lor}[2]{L_{#1,#2}}

 \DeclareMathOperator{\supp}{supp}
 
\renewcommand{\Re}{\operatorname{Re}}

\makeatletter
\def\@tocline#1#2#3#4#5#6#7{\relax
  \ifnum #1>\c@tocdepth 
  \else
    \par \addpenalty\@secpenalty\addvspace{#2}%
    \begingroup \hyphenpenalty\@M
    \@ifempty{#4}{%
      \@tempdima\csname r@tocindent\number#1\endcsname\relax
    }{%
      \@tempdima#4\relax
    }%
    \parindent\z@ \leftskip#3\relax \advance\leftskip\@tempdima\relax
    \rightskip\@pnumwidth plus4em \parfillskip-\@pnumwidth
    #5\leavevmode\hskip-\@tempdima
      \ifcase #1
       \or\or \hskip 1em \or \hskip 2em \else \hskip 3em \fi%
      #6\nobreak\relax
    \hfill\hbox to\@pnumwidth{\@tocpagenum{#7}}\par
    \nobreak
    \endgroup
  \fi}
\makeatother

\newcommand{\Proj}{\mathbb{P}}
\newcommand{\inner}[1]{\left<#1\right>}
\newcommand{\NLor}[2]{{#1,#2}}
\newcommand{\NLeb}[1]{{#1}}

\DeclareFontFamily{U}{mathx}{}
\DeclareFontShape{U}{mathx}{m}{n}{<-> mathx10}{}
\DeclareSymbolFont{mathx}{U}{mathx}{m}{n}
\DeclareMathAccent{\widehat}{0}{mathx}{"70}
\DeclareMathAccent{\widecheck}{0}{mathx}{"71}

\newcommand{\dsp}{\hspace{0.1em}}

 \title[Stokes-Magneto system with fractional diffusions]{Global existence and uniqueness of weak solutions of a Stokes-Magneto system with fractional diffusions}
 \author{Hyunseok Kim}
 \address{Department of Mathematics, Sogang University, Seoul, 04107, Republic of Korea}
 \email[Corresponding Author]{kimh@sogang.ac.kr}
 \thanks{The authors were supported by Basic Science
Research Program through the National Research Foundation of Korea (NRF) funded by
the Ministry of Education (No. NRF-2016R1D1A1B02015245).}

 \author{Hyunwoo Kwon}
 \address{Department of Mathematics, Sogang University, Seoul, 04107, Republic of Korea}
\curraddr{Division of Applied Mathematics, Brown University, 182 George Street, Providence, RI 02912, USA}
\email{hyunwoo\_kwon@brown.edu}

\subjclass{35Q35;76W05}
\keywords{Global existence; Uniqueness; Weak solutions; Fractional diffusions}

 \allowdisplaybreaks
\begin{document}

\begin{abstract}
We consider a Stokes-Magneto system  in $\mathbb{R}^d$ ($d\geq 2$) with fractional diffusions   $\Lambda^{2\alpha}\boldsymbol{u}$ and $\Lambda^{2\beta}\boldsymbol{b}$ for the velocity $\boldsymbol{u}$ and the magnetic field $\boldsymbol{b}$, respectively. Here $\alpha,\beta$ are positive constants and $\Lambda^s = (-\Delta)^{s/2}$ is the fractional Laplacian of order $s$. We establish global existence of weak solutions of the  Stokes-Magneto system   for any initial data in $\Leb{2}$ when $\alpha$, $\beta$ satisfy $1/2<\alpha<(d+1)/2$, $\beta >0$,
and $\min\{\alpha+\beta,2\alpha+\beta-1\}>d/2$. It is also shown that weak solutions are   unique  if $\beta \geq 1$ and $\min \{\alpha+\beta,2\alpha+\beta-1\}\geq d/2+1$, in addition.
\end{abstract}
\maketitle

\section{Introduction}

For several decades, many authors have studied linear and nonlinear equations with fractional diffusions from a mathematical point of view and applied them  
to diverse fields such as biology \cite{BRR13}, financial mathematics \cite{A04}, and probability \cite{SV03}. Particularly important models from fluid mechanics are generalized Navier-Stokes equations (see e.g \cite{L67,L69,CDD18,T09}) and surface quasi-geographic equations (see e.g. \cite{CCCGW12,CV10}).

In this paper, we consider the following initial value  problem for a Stokes-Magneto system with fractional diffusions in $\mathbb{R}^d$, $d\geq 2$:
\begin{equation}\label{eq:MHD-relaxed}
\left\{
\begin{alignedat}{2}
\nu \Lambda^{2\alpha} \boldu +\nabla {p_*} &=(\boldb\cdot \nabla)\boldb&&\quad \text{in }  \mathbb{R}^d\times (0,\infty),\\
\partial_t \boldb+\eta \Lambda^{2\beta} \boldb +(\boldu\cdot \nabla)\boldb &=(\boldb\cdot \nabla)\boldu &&\quad \text{in }\mathbb{R}^d \times(0,\infty),\\
\Div \boldu =\Div \boldb &=0 &&\quad \text{in } \mathbb{R}^d\times (0,\infty),\\
\boldb(\cdot,0)&=\boldb_0 &&\quad \text{on } \mathbb{R}^d.
\end{alignedat}
\right.
\end{equation}
Here $\boldu:\mathbb{R}^d\times [0,\infty)\rightarrow \mathbb{R}^d$ is the velocity field, $\boldb: \mathbb{R}^d\times [0,\infty)\rightarrow \mathbb{R}^d$ is the magnetic field, and ${p_*}:\mathbb{R}^d\times [0,\infty)\rightarrow \mathbb{R}$ denotes the total pressure: ${p_*} = p  + \frac{1}{2}|\boldb|^2$  with $p$ being the pressure on the fluid. The positive  constants $\nu$ and $\eta$ stand for the viscosity constant and the magnetic diffusivity, respectively.  For $s\in \mathbb{R}$, $\Lambda^s = (-\Delta)^{s/2}$ denotes the fractional Laplacian
 of order $s$ (see Section 2 for more details).

Our model \eqref{eq:MHD-relaxed} was motivated by Moffatt \cite{M85,M21} who proposed the magnetic relaxation method to construct magnetohydrodynamics (MHD) equilibria from a given magnetic field $\boldb_0$. More precisely, Moffatt \cite{M85} first suggested the method of magnetic relaxation to construct a magnetic equilibrium from a solution of ideal MHD equations by letting $t\rightarrow \infty$. Later, in 2009, Moffatt suggested another possible way to construct an MHD equilibrium by introducing the system \eqref{eq:MHD-relaxed} with $\alpha=1$ and $\eta=0$. See Beekie-Friedlander-Vicol \cite{BFV21}, Constantin-Pasqualotto \cite{CP22}, and references therein for further motivation of magnetic relaxation.

In comparison to \eqref{eq:MHD-relaxed}, we recall the  generalized MHD system:
\begin{equation}\label{eq:MHD}
\left\{
\begin{alignedat}{2}
\partial_t\boldu+\nu \Lambda^{2\alpha} \boldu +(\boldu\cdot\nabla)\boldu+\nabla {p_*} &=(\boldb\cdot \nabla)\boldb&&\quad \text{in } \mathbb{R}^d\times (0,\infty),\\
\partial_t \boldb+\eta \Lambda^{2\beta} \boldb +(\boldu\cdot \nabla)\boldb &=(\boldb\cdot \nabla)\boldu &&\quad \text{in }\mathbb{R}^d\times(0,\infty),\\
\Div \boldu=\Div \boldb&=0 &&\quad \text{in } \mathbb{R}^d\times (0,\infty),\\
\boldu(\cdot,0)=\boldu_0,\quad \boldb(\cdot,0)&=\boldb_0 &&\quad \text{on } \mathbb{R}^d.
\end{alignedat}
\right.
\end{equation}
When $\nu,\eta>0$ and $\alpha=\beta=1$, Duvaut-Lions \cite{DL72} proved global existence of weak solutions of \eqref{eq:MHD} with $\Leb{2}$-initial data.   Later, Wu \cite{W03} extended this result to arbitrary $\alpha,\beta>0$ and also proved existence of global classical solutions  for smooth initial data when $\alpha,\beta$ satisfy $\alpha, \beta \geq d/4+1/2$. The existence results in \cite{W03} were later improved by the same author \cite{W11} when $\alpha\geq d/4+1/2$, $\beta>0$, and $\alpha + \beta \geq d/2+1$. In addition to these results, many authors proved global regularity of classical solutions of \eqref{eq:MHD} in two-dimensional spaces under various hypotheses on $\alpha$ and $\beta$; see Yamazaki \cite{Y19} and references therein. On the other hand, when $\nu=0$ and $\eta>0$, Kozono \cite{K89} proved local existence and uniqueness of classical solutions of \eqref{eq:MHD} when $\beta=1$. This result was extended to any $\beta>0$ by Wu  \cite{W11} for sufficiently smooth initial data. Recently, Fefferman et.al \cite{FMRR14,FMRR17} established local existence of strong solutions of \eqref{eq:MHD} when  $\nu>0$, $\eta=0$, and $\alpha=1$, which was extended by Kim-Zhou \cite{KZ20} to general $\alpha$.

In contrast to the generalized MHD system, there are few results for the Stokes-Magneto system with or without fractional diffusions. When $\nu,\eta>0$, $\alpha=\beta=1$, and $d=2,3$, McCormick-Robinson-Rodrigo \cite{MRR14} proved global existence of weak solutions of \eqref{eq:MHD-relaxed}. For the two-dimensional case, they also proved uniqueness and regularity of weak solutions. Recently, Ji-Tan \cite{JT21} proved global existence of strong solutions of \eqref{eq:MHD-relaxed}  when $d=3$, $\alpha=1$, and $\beta\geq 3/2$.  On the other hand, when $\nu>0$, $\eta=0$, and $\alpha=0$, Brenier \cite{B14} proved  global existence of dissipative weak solutions on the two-dimensional torus $\mathbb{T}^2$. When $\nu>0$, $\eta=0$, and $\alpha=1$, Fefferman et.al. \cite{FMRR14} proved local existence and uniqueness of strong solutions of \eqref{eq:MHD-relaxed} on $\mathbb{R}^d$, $d=2,3$.  Recently, when $\nu>0$, $\eta=0$, and $\alpha>d/2+1$, Beekie-Friedlander-Vicol \cite{BFV21}  established global existence of strong solutions of \eqref{eq:MHD-relaxed}  on the torus $\mathbb{T}^d$, $d=2,3$. Moreover, they also investigated 2D stability and 3D instability as well as the long-time behavior of solutions.

In this paper, we establish  global existence and uniqueness of weak solutions of the Stokes-Magneto system \eqref{eq:MHD-relaxed} with fractional diffusions. To elaborate a motivation of our notion of weak solutions, we multiply the first equation  in \eqref{eq:MHD-relaxed} by $\boldu$ and the second equation by  $\boldb$, and then integrate the resulting equations on $\mathbb{R}^d$. Using the divergence-free condition  on $\boldu$ and $\boldb$, we have
\[	\nu\int_{\mathbb{R}^d} |\Lambda^\alpha \boldu|^2 \myd{x}=-\int_{\mathbb{R}^d} (\boldb \otimes \boldb):\nabla \boldu \myd{x}	\]
and
\[	\frac{1}{2}\frac{d}{dt}\int_{\mathbb{R}^d} |\boldb|^2 \myd{x} + \eta\int_{\mathbb{R}^d} |\Lambda^\beta \boldb|^2 \myd{x}=\int_{\mathbb{R}^d} (\boldb \cdot \nabla)\boldu \cdot \boldb \myd{x}.\]
Adding these two identities and integrating it in time, we derive the energy identity
\begin{equation}\label{eq:energy-equality}
\int_{\mathbb{R}^d}|\boldb(t)|^2\myd{x} + 2 \int_0^t\int_{\mathbb{R}^d} \left( \nu|\Lambda^\alpha \boldu(s)|^2+ \eta |{\Lambda^\beta \boldb(s)}|^2 \right) \myd{x}\myd{s}=\int_{\mathbb{R}^d}|\boldb_0|^2\myd{x}
\end{equation}
for all  $t \ge 0$. Furthermore, it will be shown in Subsection \ref{subsec:sol-formula} that if $1/2<\alpha<(d+1)/2$, then   $\boldu$ is given  by
\begin{equation}\label{sol-formula-int}
	\boldu=\mathbf{U}_{\alpha}*(\boldb\otimes \boldb)
\end{equation}
where  $\mathbf{U}_\alpha$ is a matrix-valued function satisfying
\[	|\mathbf{U}_{\alpha}(x)|\leq \frac{C}{|x|^{d+1-2\alpha}} \]
for some constant $C$ depending only on $d$ and $\alpha$. By Young's convolution inequality in weak spaces (Lemma \ref{lem:Young}), we deduce that
\begin{equation}\label{eq:weak-spaces-estimate}
\norm{\boldu(t)}{\Lor{d/(d+1-2\alpha)}{\infty}}\leq C \norm{\boldb(t)\otimes \boldb(t)}{\Leb{1}}\leq C\norm{\boldb(t)}{\Leb{2}}^2
\end{equation}
for some constant $C=C(d,\alpha)$. Here $\Leb{p}$  and $\Lor{p}{q}$ denote the Lebesgue spaces and Lorentz spaces  on $\mathbb{R}^d$, respectively, where $1\leq p<\infty$ and $1\leq q\leq \infty$. For $s\in \mathbb{R}$, let $\thSob{s}$ be the homogeneous Sobolev spaces on $\mathbb{R}^d$ (see Section \ref{sec:prelim} for details). Then motivated by   \eqref{eq:energy-equality} and \eqref{eq:weak-spaces-estimate}, we define weak solutions of \eqref{eq:MHD-relaxed} as follows.


\begin{defn}
Let  $1/2< \alpha < (d+1)/2$ and  $0< \beta < \infty$.  Suppose that $\boldb_0 \in \Leb{2}$, ${\rm div}\, \boldb_0 =0$ in $\mathbb{R}^d$, and   $0<T \le \infty$. Then a pair $(\boldu,\boldb )$ of vector fields  satisfying
\begin{align*}
 \boldu &\in \Leb{\infty}(0,T;\Lor{{d}/{(d+1-2\alpha)}}{\infty})\cap \Leb{2}(0,T;\thSob{\alpha}), \\
 \boldb &\in \Leb{\infty}(0,T;\Leb{2}) \cap \Leb{2}(0,T;\thSob{\beta})
\end{align*}
is called a \emph{weak solution} of \eqref{eq:MHD-relaxed} in $[0, T)$ if the following are satisfied:
\begin{enumerate}[label=\textnormal{(\arabic*)}]
\item  For all $\Phi \in C_{c}^\infty(\mathbb{R}^d;\mathbb{R}^d )$ with $\Div \Phi=0$,
\begin{equation}\label{eq:velocity}
 \int_{\mathbb{R}^d} \nu \Lambda^{\alpha} \boldu (t) \cdot \Lambda^\alpha \Phi \myd{x}=-\int_{\mathbb{R}^d} \boldb (t)\otimes \boldb (t): \nabla \Phi \myd{x}
\end{equation}
for almost all $t\in (0,T)$.
\item For all $\Phi \in C_c^\infty(\mathbb{R}^d\times[0,T);\mathbb{R}^d)$,
\begin{equation}\label{eq:magnetic}
\begin{aligned}
 \relphantom{=}\int_0^T \int_{\mathbb{R}^d} \boldb \cdot (-\partial_t \Phi  + \eta \Lambda^{2\beta}\Phi)\myd{x}dt & -\int_0^T\int_{\mathbb{R}^d} (\boldu\otimes \boldb-\boldb\otimes \boldu):\nabla \Phi \myd{x}dt\\
&= \int_{\mathbb{R}^d} \boldb_0 \cdot \Phi( 0) \,dx.
\end{aligned}
\end{equation}
\item  For all $\psi \in C_{c}^\infty(\mathbb{R}^d )$,
\[  \int_{\mathbb{R}^d}   \boldu (t) \cdot \nabla \psi  \myd{x}= \int_{\mathbb{R}^d}   \boldb (t) \cdot \nabla \psi  \myd{x} =0  \]
for almost all $t\in (0,T)$.  
\end{enumerate}
A weak solution of  \eqref{eq:MHD-relaxed} in $[0, \infty)$ will be called a \emph{global weak solution}.
\end{defn}

Now we are ready to present the main results of the paper. We  first establish existence of global weak solutions of \eqref{eq:MHD-relaxed} under a suitable assumption on $\alpha$ and $\beta$.
\begin{thm}\label{thm:GWP}
Let    $\alpha$ and $\beta$ satisfy
\[
   \frac{1}{2}<\alpha<\frac{d+1}{2},\quad \beta>0 ,\quad\mbox{and}\quad \min\{ \alpha+\beta,2\alpha+\beta-1\}>\frac{d}{2}.
\]
Then for any $\bm{b}_0\in \Leb{2}$ with ${\rm div}\, \bm{b}_0 =0$, there exists at least one  global  weak solution of   \eqref{eq:MHD-relaxed}.
\end{thm}

Next, we   prove the time-continuity, energy identity, and uniqueness of  weak solutions of \eqref{eq:MHD-relaxed} for larger values of  $\alpha$ and $\beta$.

\begin{thm}\label{thm:uniqueness}
Let $\alpha$ and $\beta$ satisfy
\[	\frac{1}{2}<\alpha<\frac{d+1}{2},\quad \beta \geq 1,\quad \text{and}\quad \min\{\alpha+\beta,2\alpha+\beta-1\}\geq \frac{d}{2}+1.	\]
Then for any $\bm{b}_0\in \Leb{2}$ with ${\rm div}\, \bm{b}_0 =0$, there exists a unique  global  weak solution $(\boldu,\boldb )$ of  \eqref{eq:MHD-relaxed}. Moreover, the   solution $(\boldu,\boldb )$ is  continuous in $\Lor{{d}/{(d+1-2\alpha)}}{\infty}   \times \Leb{2}$ and satisfies the  energy identity \eqref{eq:energy-equality} for any $t\ge 0$.
\end{thm}
 
\begin{rem*}

(a) Assume that $\alpha =1$ and $\beta >d/2 -1$. Then it follows from Theorem \ref{thm:GWP} that   for any $\bm{b}_0\in \Leb{2}$ with ${\rm div}\, \bm{b}_0 =0$, there exists a  global  weak solution $(\boldu,\boldb )$ of   \eqref{eq:MHD-relaxed}.
This result   extends the existence   result  of  McCormick-Robinson-Rodrigo \cite{MRR14} who considered the case when $d =2, 3$ and $\beta =1$.
Moreover, by Theorem \ref{thm:uniqueness}, weak solutions of   \eqref{eq:MHD-relaxed} are unique if $\beta \ge d/2 $. Hence it follows from the   existence   theorem  for strong solutions due to  Ji-Tan \cite{JT21} that  if $d=3$ and $\bm{b}_0\in H^1$, then $(\boldu,\boldb)$   satisfies the following strong regularity properties:
$$
\boldu \in C ( [0,T]; H^{3/2}  )\cap L_2  ( 0,T;H^2  ),
\quad \boldb \in C ([0,T]; H^{1}  )\cap L_2  ( 0,T;H^{1+\beta}  )
$$
for any $T<\infty$.

(b) Assume that $\beta =1$ and $\max \{ 1/2, d/2 -1 , d/4 \} < \alpha <(d+1)/2$. Then it follows from Theorem \ref{thm:GWP} that   for any $\bm{b}_0\in \Leb{2}$ with ${\rm div}\, \bm{b}_0 =0$, the problem \eqref{eq:MHD-relaxed} has  global  weak solutions. Moreover, by Theorem  \ref{thm:uniqueness}, weak solutions of  \eqref{eq:MHD-relaxed} are     unique if $ \alpha  \ge d/2 $ in addition.

(c) Assume that $2 \le d \le 4$ and  $(d+2)/6 < \alpha  =\beta <(d+1)/2$. Then it follows from Theorems \ref{thm:GWP} and \ref{thm:uniqueness} that   for any $\bm{b}_0\in \Leb{2}$ with ${\rm div}\, \bm{b}_0 =0$, the problem \eqref{eq:MHD-relaxed} has  global  weak solutions, which    are     unique if $(d+2)/4 \le  \alpha  =\beta <(d+1)/2$. 
\end{rem*}

One of our main tools to prove Theorems \ref{thm:GWP} and \ref{thm:uniqueness} is the following product estimate (see Lemma  \ref{lem:product-estimate}): if $1/2<\alpha<(d+1)/2$,  $\beta>0$, and  $\alpha+\beta>d/2$, then there exist constants   $0<\theta_1,\theta_2<1$ such that
\begin{equation}\label{eq:prod-est}
	\norm{fg}{\NLeb{2}}\apprle \norm{f}{\NLor{d/(d+1-2\alpha)}{\infty}}^{1-\theta_1}\norm{f}{\thSob{\alpha}}^{\theta_1}\norm{g}{\NLeb{2}}^{1-\theta_2}\norm{g}{\thSob{\beta}}^{\theta_2}
\end{equation}
for all $f\in\Lor{d/(d+1-2\alpha)}{\infty}\cap \thSob{\alpha}$ and $g\in \tSob{\beta}$. Moreover, if $\alpha+(1-\mu)\beta \geq d/2+1$ for some $0\leq \mu \le 1$, we can choose $\theta_1$ and $\theta_2$ so that $(1+\mu)\theta_1+\theta_2+1/\beta \leq 2$. This additional condition enables us to establish  uniqueness of weak solutions of \eqref{eq:MHD-relaxed}.  The estimate \eqref{eq:prod-est} will be derived from generalized Gagliardo-Nirenberg interpolation inequalities that were obtained recently by Byeon-Kim-Oh \cite{BKO21} and Wang-Wei-Ye \cite{WWY21}.

To prove   existence of weak solutions (Theorem \ref{thm:GWP}) of \eqref{eq:MHD-relaxed}, we first construct a family of solutions $\{(\boldu_R,\boldb_R)\}$ of the $R$-truncated problems \eqref{eq:MHD-relaxed-approx}.  Moreover, it is shown that these solutions   satisfy uniform bounds. Then by a standard Aubin-Lions compactness argument combined with the  estimate  \eqref{eq:prod-est}, we  construct a sequence of functions $\{\boldb_{R_k}\}$ that has   strong convergence in $\Leb{2}(0,T;\Leb{2}(K))$ for every finite $T>0$ and compact   $K \subset \mathbb{R}^d$. To show strong convergence of $\boldu_{R_k}$, we utilize  the solution formula (\ref{sol-formula-int}) for the fractional Stokes system that will be derived  in Section 3 (Theorem \ref{theorem:Frational Stokes-Lpdata}). 
Very weak solutions of the fractional Stokes system are   introduced to apply  the solution formula (\ref{sol-formula-int}) for $L^1$-data.
Finally, making crucial use of the assumption $2 \alpha +\beta >d/2+1$, we   show that the nonlinear terms $\boldu_{R_k} \otimes \boldb_{R_k}$ converge  strongly in $L_{1,\mathrm{loc}} (\mathbb{R}^d \times (0,\infty ))$. The  strong  convergence properties of $ \boldu_{R_k}$ and $\boldb_{R_k}$ enable us to complete the proof of    existence of weak solutions of \eqref{eq:MHD-relaxed} by passing to the limit $k\to\infty$.    To prove  uniqueness of weak solutions (Theorem \ref{thm:uniqueness}), we first study weak solutions of  perturbed fractional heat equations. Then applying  a standard Gronwall-type argument combined with product estimates (Lemmas \ref{lem:product-estimate} and \ref{lem:commutator}), we show that weak solutions are unique.

This paper proceeds in six sections and three appendix sections. In Section \ref{sec:prelim}, we introduce our notations, the fractional Laplacian, homogeneous Sobolev spaces, and Lorentz spaces. We also provide proofs of product estimates that play crucial roles in proving existence and uniqueness of weak solutions.  In Section \ref{sec:Stokes}, we prove the  solution formula \eqref{sol-formula-int} for the fractional Stokes system. The notion of very weak solutions is also introduced. In Section \ref{sec:perturbed-heat}, we prove existence and regularity properties of weak solutions of  perturbed fractional heat equations that will be used in the proof of uniqueness of weak solutions of \eqref{eq:MHD-relaxed}.  The main theorems are proved in Sections \ref{sec:proofs} and \ref{sec:uniqueness}. Finally we  give somewhat detailed proofs of some technical results  in Appendices.

\section{Preliminaries}\label{sec:prelim}

\subsection{Notations}
Let $\mathbb{N}_0 =\{0,1,2,\ldots\}$ be the set of nonnegative integers.   For $\gamma =(\gamma_1,\dots,\gamma_d)\in \mathbb{N}^d_0$ and  a  function $u: \mathbb{R}^d\rightarrow \mathbb{R}$, we write
\[   D^\gamma u = D_1^{\gamma_1}\cdots D_d^{\gamma_d}u \quad\mbox{and}\quad |\gamma|=\gamma_1+\cdots+\gamma_d. \]
We write $\nabla u = (D_1 u,\dots,D_d u)$ for the gradient of $u$.

We use bolditalic to denote vector fields, e.g., $\boldu:\mathbb{R}^d\times (0,\infty)\rightarrow \mathbb{R}^d$. Similarly we use boldroman to denote 2-tensors, e.g., $\boldF :\mathbb{R}^d\times (0,\infty)\rightarrow \mathbb{R}^{d\times d}$. For two vector fields $\boldu=(u^1,\dots,u^d)$ and  $\boldv=(v^1,\dots,v^d)$, let $\boldu\otimes \boldv$ be the $2$-tensor whose $(i,j)$ component is   $u^i v^j$ for $1\leq i,j\leq d$. Also, $(\boldu \cdot \nabla) \boldv$ is the vector field defined by
\[    [(\boldu \cdot \nabla) \boldv]^i = \sum_{j=1}^d u^j \frac{\partial v^i}{\partial x_j} \quad (1\leq i\leq d).
\]
For two 2-tensors $\boldF=[F^{ij}]_{i,j=1}^d$ and $\boldG=[G^{ij}]_{i,j=1}^d$,  their inner product is defined  by
\[   \boldF:\boldG = \sum_{i,j=1}^d F^{ij} G^{ij} \]

For $\Omega\subset \mathbb{R}^d$ and  $1\leq p\leq \infty$, $\Leb{p}(\Omega)$ denotes the $\Leb{p}$-space over $\Omega$   with the norm $\norm{\cdot}{\NLeb{p};\Omega}$. For $1 \le p<\infty$ and $1\leq q\leq \infty$, we denote by $\Lor{p}{q}(\Omega)$ the Lorentz space  over $\Omega$ with the quasi-norm $\norm{\cdot}{\NLor{p}{q};\Omega}$.  Let $C^\infty_c(\Omega)$ denote  the space of all smooth functions (or vector fields) that have compact supports in $\Omega$. We denote  by $C_{c,\sigma}^\infty (\Omega)$ the set of all vector fields in $C_c^\infty(\Omega)$ that are  divergence-free.  For the sake of convenience, we drop $\Omega$ in   function spaces and their (quasi-)norms if $\Omega=\mathbb{R}^d$, e.g., $L_p =L_p (\mathbb{R}^d )$ and $\|\cdot \|_p = \|\cdot\|_{p; \mathbb{R}^d}$. For $x\in \mathbb{R}^d$ and $r>0$, let $B_r (x)$ denote the open ball   of radius $r$ centered at $x$. We write $B_r$ for open balls centered at the origin.

For two nonnegative quantities $A$ and $B$, we write $A\apprle_{\alpha,\beta,\ldots} B$ if $A\leq CB$ for some positive constant $C$ depending on the parameters $\alpha$, $\beta$, \ldots. If the dependence is evident, we usually omit the subscripts and simply write $A\apprle B$.  We write $A\approx B$ if $A\apprle B$ and $B\apprle A$.
For two topological vector spaces $X$ and $Y$ with $X \subset Y$, we write $X\hookrightarrow Y$ if  $X$ is continuously embedded into $Y$, that is, the embedding $x \mapsto x$ is continuous.  For a topological vector   space $X$, we denote by $X^*$ the (topological) dual space of $X$. The dual pairing between $X^*$ and $X$ is denoted by $\action{\cdot,\cdot}$.

Let $\mathscr{S}$ and $\mathscr{S}'$ be the Schwartz class and the space of tempered distributions on $\mathbb{R}^d$, respectively.
For an integrable  function  $f$ on $\mathbb{R}^d$, we denote by $\hat{f}$ and $f^{\vee}$  the  Fourier transform and the inverse Fourier transform, respectively, of $f$:
\[	\hat{f}(\xi) = \int_{\mathbb{R}^d} e^{-2\pi i x\cdot \xi} f(x) \myd{x}\quad \text{and}\quad f^{\vee}(x)=\int_{\mathbb{R}^d} e^{2\pi i x\cdot \xi} f (\xi) \myd{\xi}.\]
If $f$ is a tempered distribution, then $\hat{f}$ and $f^{\vee}$ are defined via duality by  
\[
\langle{\hat{f},\phi}\rangle = \langle{f, \hat{\phi}}\rangle \quad\mbox{and}\quad \action{f^{\vee},\phi} =  \action{f,\phi^\vee} \quad\mbox{for all}\,\,\phi \in \mathscr{S}.
\]

\subsection{Sobolev spaces}
 
For $s\in \mathbb{R}$ and $1<p<\infty$, the \emph{inhomogeneous Sobolev space} $\tSob{s}_p$ is defined by
\[	\tSob{s}_p = \{ f \in \mathscr{S}' : J^s f \in \Leb{p} \},\]
where  $J^s = (I-\Delta)^{s/2} : \mathscr{S}'\rightarrow \mathscr{S}'$ is defined via the Fourier transform by
\[	\widehat{J^s f} = \left( 1+4\pi^2 |\cdot|^2 \right)^{s/2} \hat{f}\quad \text{for } f\in \mathscr{S}'.\]
For  $p=2$, we write $\tSob{s}  =\tSob{s}_2$. Then the Sobolev space $\tSob{s}$ is indeed a Hilbert space equipped with the inner product
\[  (u,v)_{\tSob{s}} = \left( J^s u, J^s v \right) \]
with $(\cdot,\cdot)$ denoting the inner product on $\Leb{2}$. By the Plancherel theorem, we   have
\[     (u,v)_{\tSob{s}} = \int_{\mathbb{R}^d} \left( 1+4\pi^2 |\xi|^2 \right)^s \hat{u}(\xi)\overline{\hat{v}(\xi)} \,\myd{\xi}.\]

To define   homogeneous Sobolev spaces, we need to   introduce the following subspace of $\mathscr{S}$:
\[	\mathscr{S}_0 =\{ \phi \in \mathscr{S} : D^\gamma \hat{\phi}(0)=0\quad \text{for all } \gamma \in \mathbb{N}^d_0 \}.\]
Let $\mathscr{S}_0'$ be the dual space of $\mathscr{S}_0$. Then since  $\mathscr{S}_0 \hookrightarrow \mathscr{S}  \hookrightarrow \Leb{1} \cap \Leb{\infty}$, we have
\[
\Leb{1}+\Leb{\infty} \hookrightarrow\mathscr{S} ' \hookrightarrow \mathscr{S}_0 ' .
\]
Moreover, it follows from the Hahn-Banach theorem that every $f \in \mathscr{S}_0 '$ can be extended to a tempered distribution $F$. Such an $F$ is unique up to additive polynomials, thanks to the  following simple  lemma (see e.g. \cite[Proposition 1.1.3]{G14-1}).
\begin{lem}\label{lem:polynomial-isomorphism}  Let $f $ be a tempered distribution. Then $f$ is a polynomial  if and only if $\action{f,\psi}=0$ for all $\psi \in \mathscr{S}_0$.
\end{lem}

The fractional Laplacian $\Lambda^s : \mathscr{S}_0\rightarrow \mathscr{S}_0$ of order $s\in \mathbb{R}$ is defined by
\[	\Lambda^s \phi = \left((2\pi |\cdot|)^s \hat{\phi} \right)^{\vee}\quad \text{for all } \phi \in \mathscr{S}_0.\]
Then since $\Lambda^{s+t}=\Lambda^s \Lambda^t$ for any $s,t \in \mathbb{R}$, it follows that  $\Lambda^s : \mathscr{S}_0\rightarrow \mathscr{S}_0$ is an isomorphism with the inverse $\Lambda^{-s}$. Hence for each $u\in \mathscr{S}_0'$, we   define $\Lambda^s u\in \mathscr{S}_0'$ via duality by
\[	\action{\Lambda^s u,\phi}=\action{u,\Lambda^s \phi}\quad \text{for all } \phi \in \mathscr{S}_0.   \]
Note also   that $\Lambda^s : \mathscr{S}_0'\rightarrow \mathscr{S}_0'$ is an isomorphism.

For $s\in \mathbb{R}$ and $1<p<\infty$, the \emph{homogeneous Sobolev space} $\thSob{s}_p$ is defined by
\[	\thSob{s}_p = \{ f\in \mathscr{S}_0' : \Lambda^s f\in \Leb{p} \}.\]
For $p=2$, we write $\thSob{s}=\thSob{s}_p$. The space $\thSob{s}$ is a Hilbert space equipped with the inner product
\[	(u,v)_{\thSob{s}} = \left(\Lambda^s u,\Lambda^s v \right).\]
It should be remarked  that $\Lambda^t : \thSob{s+t} \rightarrow \thSob{s}$ is an isometric  isomorphism for any $s,t \in \mathbb{R}$. It is standard that $\tSob{s}=\Leb{2}\cap \thSob{s}$ for any $s\geq 0$ (see e.g. \cite[Theorem 6.3.2]{BL76}).

The following density results will be used several times in the paper.

\begin{lem}\label{lem:density-approx-S0}\leavevmode
\begin{enumerate}[label=\textnormal{(\arabic*)}]
\item $C_c^\infty$ is dense in $\mathscr{S}$.
\item For each $\phi\in \mathscr{S}$, there exists  a sequence $\{\phi_k\}$ in $\mathscr{S}_0$ such that $\phi_k  \rightarrow \phi$ in $\tSob{s}$ for any $s \ge 0$
\item$ \mathscr{S}_0$ is dense in both $\tSob{s}$ and  $\thSob{s}$ for   $s \in \mathbb{R}$.
\item   $\mathscr{S}_0$ is dense in $\Leb{p}\cap \thSob{s}$ for  $1<p<\infty$ and $s\in \mathbb{R}$.
\end{enumerate}
\end{lem}
\begin{proof}
To prove Parts (1) and  (2), we choose $\psi \in C^\infty$ such  that $0\leq \psi\leq 1$, $\psi=0$ in $B_1$, and $\psi=1$ outside $B_2$. Given $\phi\in \mathscr{S}$, we  define $\phi_k (x) = \left[1-\psi (x/k )\right]\phi (x)$ for $k\in \mathbb{N}$. Then $\phi_k \in C_c^\infty$ and $\phi_k  \rightarrow \phi$ in $\mathscr{S}$. Moreover, if    $\psi_k  = \left(\psi (k \cdot ) \hat{\phi} \right)^{\vee}$ for $k\in \mathbb{N}$,  then  $\psi_k \in \mathscr{S}_0$ and
$$
\norm{\psi_k  -\phi}{\tSob{s}}^2 = \int_{\mathbb{R}^d} \left( 1+4\pi^2 |\xi|^2 \right)^s \left| \psi  (k \xi )- 1 \right|^2 |\hat{\phi}(\xi)|^2 \,\myd{\xi} \rightarrow 0
$$
for any $s \ge 0$, by the dominated convergence theorem.

Suppose that $f \in \thSob{s}$ and $u= \Lambda^s f$. Then since $\mathscr{S}$ is dense in $L_2$, it follows from Part (2) that  there exists a sequence $\{u_k\}$ in $\mathscr{S}_0$ such that $u_k \to u$ in $L_2$. If $f_k = \Lambda^{-s} u_k$, then $f_k \in  \mathscr{S}_0$ and  $\norm{f_k-f}{\thSob{s}} = \norm{u_k-u}{2} \to 0$. This proves that $\mathscr{S}_0$ is dense in $\thSob{s}$. Similarly, using $J^s$ instead of $\Lambda^s$, one   easily shows that $\mathscr{S}_0$ is dense in $\tSob{s}$. This proves Part (3). A proof of (4) will be given later in Appendix \ref{sec:app-b}.
\end{proof}

An immediate consequence of Lemma \ref{lem:density-approx-S0} (3) is that every $f \in \thSob{-s}$ can be extended to a bounded linear functional on $\thSob{s}$ in such a way that the inequality
\begin{equation}\label{dual-Hspace}
\action{f,g } \le \|f\|_{\thSob{-s}}\|g\|_{\thSob{s}}
\end{equation}
holds for all $g \in \thSob{s}$. Moreover, since $L^2$ is self-dual and $\Lambda^{t} : \dot{H}^{t} \to L^2$ is an isomorphism for any $t \in \mathbb{R}$, it is easily shown that $\dot{H}^{-s}$ is  the dual space of $\dot{H}^s$.

Assume now that $s > 0$. Then it can be shown (see e.g. \cite[Lemma 1]{GO14}) that   $\Lambda^s$ is well-defined on $\mathscr{S}$ and maps $\mathscr{S}$ to $ \mathscr{S}_s$, where 
\[ \mathscr{S}_s =\left\{ f\in C^\infty : \left( 1+|\cdot|^{d+s} \right)D^\gamma \psi \in \Leb{\infty} \,\,\text{ for all } \gamma \in \mathbb{N}_0^d \right\} .
\]
In particular,  $\Lambda^s$ is bounded from $\mathscr{S}$ to $\Leb{1}\cap \Leb{\infty}$. Hence for each $u\in \Leb{1}+\Leb{\infty}$,  we can extend   $\Lambda^s u$    to  a tempered distribution by defining
\[	\action{\Lambda^s u ,\phi}=\int_{\mathbb{R}^d} u \Lambda^s \phi \myd{x}\quad \mbox{for all}\,\, \phi \in \mathscr{S}  .\]

\subsection{Lorentz spaces}

Let $\Omega$ be any domain in $\mathbb{R}^d$. For a measurable function $f:\Omega \rightarrow \mathbb{C}$, the \emph{decreasing rearrangement} $f^*$ of $f$ is defined by
\[ f^*(t) = \inf \{ s>0 : d_f(s)\leq t\}\quad (t>0), \]
where $d_f(s)=|\{ x\in \Omega : |f(x)|>s\}|$.   For $1\leq p<\infty$ and $1\leq q\leq \infty$, we define
\[  \norm{f}{\NLor{p}{q};\Omega} =\begin{dcases*}
 \left( \int_0^\infty [t^{1/p}f^*(t)]^q\frac{dt}{t} \right)^{1/q} & if $q<\infty$,\\
 \sup_{t>0} t^{1/p}f^*(t)& if $q=\infty$.
\end{dcases*}   \]
The set of all $f$ satisfying $\norm{f}{\NLor{p}{q};\Omega}<\infty$ is denoted by $\Lor{p}{q}(\Omega)$ and   called the \emph{Lorentz space} over $\Omega$ with indices $p$ and $q$. It is well-known that  $\Lor{p}{q}(\Omega)$ is a complete quasi-normed space and is normable if $1<p<\infty$  (see e.g. \cite[Theorem 3.22]{SW71}).  Moreover, for $1 \le p<\infty$, we have
\[
   \Lor{p}{1}(\Omega)\hookrightarrow \Lor{p}{p}(\Omega)=\Leb{p}(\Omega)\hookrightarrow \Lor{p}{\infty}(\Omega).
\]

Using Lorentz spaces, we can refine several well-known inequalities in analysis. The first one is  H\"older's inequality in Lorentz spaces which was proved by O'Neil \cite[Theorems 3.4, 3.5]{O63}.

\begin{lem}\label{lem:Holder}
  Let $1 \le  p,p_1,p_2 <\infty$ and $1 \le q,q_1,q_2\leq \infty$ satisfy
  \[  \frac{1}{p}=\frac{1}{p_1}+\frac{1}{p_2}\quad \text{and}\quad \frac{1}{q}\leq \frac{1}{q_1}+\frac{1}{q_2}.  \]
  Then  for all $f\in \Lor{p_1}{q_1}(\Omega)$ and $g\in \Lor{p_2}{q_2}(\Omega)$,  \[   \norm{fg}{\NLor{p}{q};\Omega}\apprle \norm{f}{\NLor{p_1}{q_1};\Omega}\norm{g}{\NLor{p_2}{q_2};\Omega}.   \]
\end{lem}

The second one is the following refinement of Sobolev's inequality (see e.g. \cite[Theorem 7.34]{AF03} and \cite[Theorem 6.5.1]{BL76}).
\begin{lem}\label{lem:embedding}
Let $0<s<d/2$. Then for all $u\in \thSob{s}$, we have
\[   \norm{u}{\NLor{{2d}/{(d-2s)}}{2}}\apprle \norm{u}{\thSob{s}}.\]
\end{lem}

The third one is the following generalization of  Young's convolution inequality due to O'Neil \cite[Theorem 2.6]{O63} (see also  \cite[Theorems 1.2.13 and 1.4.24]{G14-2}).
\begin{lem}\label{lem:Young}
Let $1 < p,p_1 < \infty$ and $1 \le p_2 < \infty$  satisfy
$$
\frac{1}{p}+1=\frac{1}{p_1}+\frac{1}{p_2} .
$$
Then for all $f \in L^{p_1, \infty}  $ and $g \in L^{p_2} $, we have
\begin{align*}
 \|f*g\|_{p}& \apprle  \|f\|_{p_1, \infty} \|g\|_{p_2} \quad\mbox{if}\,\, 1< p_2 < \infty;\\
\|f*g\|_{p,\infty}& \apprle  \|f\|_{p_1, \infty} \|g\|_{1} \quad\mbox{if}\,\,  p_2 =1 .
\end{align*}
\end{lem}

Recall the following classical Gagliardo-Nirenberg interpolation inequality   (see e.g. \cite[Theorem 5.10]{BKO21}).
\begin{prop}\label{prop:Gagliardo}
Let $0\leq s_0< s$,  $1<p,p_1<\infty$, and $0<\theta<1$ satisfy
\[	\frac{1}{p}-\frac{s_0}{d}=\frac{1-\theta}{p_1}+\theta \left(\frac{1}{2}-\frac{s}{d}\right)\quad\text{and}\quad \theta\geq \frac{s_0}{s}.\]
Then for all $f\in \Leb{p_1}\cap \thSob{s}$, we have
\[	\norm{\Lambda^{s_0} f}{\NLeb{p}}\apprle \norm{f}{\NLeb{p_1}}^{1-\theta}\norm{\Lambda^{s}f}{\NLeb{2}}^{\theta}.\]	
\end{prop}
Particularly, if $s_0 =0$ and $p_1 =2$ in Proposition \ref{prop:Gagliardo}, we deduce the well-known Sobolev embedding
$H^s \hookrightarrow L^p $ for $s>0$ and $2 \le p<\infty$ satisfying $1/p \ge 1/2 -s/d$.

The Gagliardo-Nirenberg interpolation inequality can be refined by using Lorentz spaces.  The following is a special case of the Gagliardo-Nirenberg interpolation inequalities in Sobolev-Lorentz spaces (see  \cite[Theorem 1.1]{WWY21} and \cite[Theorem 5.17]{BKO21}).
\begin{prop}\label{prop:upper-critical}
Let $  s >0$, $1<p , p_1 <  \infty$, and $0<\theta<1$ satisfy
\[ \frac{1}{p}=\frac{1-\theta}{p_1}+\theta \left(\frac{1}{2}-\frac{s}{d} \right)\quad \text{and}\quad \frac{1}{p_1}\neq \frac{1}{2}-\frac{s}{d}.  \]
Then for all   $f\in \Lor{p_1}{\infty}\cap \thSob{s}$, we have
\[   \norm{f }{\NLor{p}{1}}\apprle \norm{f}{\NLor{p_1}{\infty}}^{1-\theta}\norm{\Lambda^s f}{\NLeb{2}}^{\theta}.\]
\end{prop}

\subsection{Product estimates}

In this subsection, we prove   product estimates that will be used in the paper. The first product estimate plays a crucial role in the proof of existence and uniqueness of weak solutions of \eqref{eq:MHD-relaxed}.
\begin{lem}\label{lem:product-estimate}
  Let $\alpha , \beta>0$ satisfy $1/2<\alpha<(d+1)/2$ and  $\alpha+\beta >d/2$. Then there exist $0<\theta_1,\theta_2<1$ such that
  \[ \norm{fg}{\NLeb{2}}\apprle \norm{f}{\NLor{d/(d+1-2\alpha)}{\infty}}^{1-\theta_1}\norm{f}{\thSob{\alpha}}^{\theta_1}\norm{g}{\NLeb{2}}^{1-\theta_2}\norm{g}{\thSob{\beta}}^{\theta_2}   \]
  for all $f\in \Lor{d/(d+1-2\alpha)}{\infty}\cap \thSob{\alpha}$ and $g\in \tSob{\beta}$. Moreover, if $\alpha+(1-\mu)\beta \geq d/2+1$ for some $0\leq \mu \le 1$, then $\theta_1$ and $\theta_2$ can be chosen so that $(1+\mu)\theta_1+\theta_2+1/\beta\leq 2$.
\end{lem}
\begin{proof}   We will show that there exists $2< p< \infty$ satisfying
  \begin{equation}\label{eq:GN-pair}
      \frac{d}{2}-\frac{d}{p}<\alpha<\frac{d+1}{2}-\frac{d}{2p}\quad \text{and}\quad \beta >\frac{d}{p}.
  \end{equation}
  For such   $p$, we define
  \[ \theta_1=\frac{d+1-2\alpha-d/p}{{d}/{2}+1-\alpha}\quad \text{and }\quad \theta_2 =\frac{d}{p\beta}. \]
  Then  since
  \begin{equation}\label{eq:p-theta-relation}
\begin{aligned}
\frac{1}{p} &= (1-\theta_1 ) \frac{d+1-2\alpha}{d} +\theta_1 \left(\frac{1}{2}-\frac{\alpha}{d}\right),\\
  \frac{p-2}{2p} &= \frac{1-\theta_2}{2} +\theta_2 \left(\frac{1}{2}-\frac{\beta}{d}\right),
\quad\mbox{and}\quad
 0< \theta_1 , \theta_2 < 1 ,
\end{aligned}
\end{equation}
it follows from  H\"older's inequality and  Proposition \ref{prop:upper-critical} that
\begin{align*}
  \norm{fg}{\NLeb{2}}\leq \norm{f}{\NLeb{p}}\norm{g}{\NLeb{2p/(p-2)}}\apprle \norm{f}{\NLor{d/(d+1-2\alpha)}{\infty}}^{1-\theta_1}\norm{f}{\thSob{\alpha}}^{\theta_1}\norm{g}{\NLeb{2}}^{1-\theta_2}\norm{g}{\thSob{\beta}}^{\theta_2}.
\end{align*}
Moreover, since
\begin{equation*}
{\beta \theta_2} = \frac{d}{p}= (1-\theta_1)\left({d+1-2\alpha}\right)+\theta_1 \left(\frac{d}{2}-{\alpha}\right),
\end{equation*}
we deduce that  if $\alpha+(1-\mu)\beta \geq d/2+1$ for some $0\leq \mu \le   1$, then
\begin{align*}
\beta\left[(1+\mu)\theta_1+\theta_2-2\right]+1&=
\left[\alpha+(1+\mu)\beta-\frac{d}{2}-1\right] \theta_1 -2\alpha -2\beta +d+2\\
&\leq -\alpha-(1-\mu)\beta+\frac{d}{2}+1\leq 0.
\end{align*}
Therefore it remains to prove existence of  $2< p<\infty$ satisfying \eqref{eq:GN-pair}.

 First, if ${d}/{2}\leq \alpha<{(d+1)}/{2}$ and $\beta >0$, then choosing a sufficiently large $2<p<\infty$, we obviously have
\[
\frac{d}{2}-\frac{d}{p}< \frac{d}{2} \le \alpha<\frac{d+1}{2}-\frac{d}{2p}\quad \text{and}\quad \beta >\frac{d}{p}.
\]
Suppose next that ${1}/{2}<\alpha< {d}/{2}$  and $ \beta >d/2 -\alpha$.
Then
\[
\max \left\{ 2 , \frac{d}{\beta}, \frac{d}{d+1-2\alpha} \right\} < \frac{d}{{d}/{2}-\alpha} < \infty  .
\]
Choosing any $p$ such that
\[
\max \left\{ 2 , \frac{d}{\beta}, \frac{d}{d+1-2\alpha} \right\} < p< \frac{d}{{d}/{2}-\alpha} ,
\]
we have
\[
2<p<\infty , \quad \frac{d}{2}-\frac{d}{p}<   \alpha<\frac{d+1}{2}-\frac{d}{2p} , \quad \text{and}\quad \beta >\frac{d}{p}.
\]
This completes the proof of Lemma \ref{lem:product-estimate}.
  \end{proof}

To prove uniqueness of weak solutions of \eqref{eq:MHD-relaxed} particularly when $1/2<\alpha <1$, we also need the following product estimate. 
\begin{lem}\label{lem:commutator}
Let  $\gamma>d/2$ and $0<s\leq \gamma$. Then for all $f\in \tSob{s}$ and $g\in \tSob{\gamma}$,
\[	\norm{\Lambda^s(fg)}{\NLeb{2}}\leq C \norm{f}{\tSob{s}}\norm{g}{\tSob{\gamma}}, \]
where $C=C(d,\gamma,s)>0$.
\end{lem}
\begin{proof}
A proof is given in \cite[Lemma 3.2]{KZ20}. However,  we provide a detailed proof of the lemma for the sake of convenience.

Suppose that $0<s < \gamma$. Then since $\gamma >d/2$, there exists $2 < p <  \infty$ such that
\[	\max \left\{ \frac{1}{2}-\frac{s}{d},0\right\} < \frac{1}{p} < \min \left\{\frac{1}{2},\frac{\gamma-s}{d}\right\}.\]
By the Sobolev embedding theorem, we have
\[	\norm{f}{\NLeb{p}}\apprle \norm{f}{\tSob{s}}\quad \text{and}\quad\norm{J^s g}{\NLeb{{2p}/{(p-2)}}}\apprle \norm{J^s g}{\tSob{\gamma-s}}.\]
Therefore, it follows from the fractional Leibniz rule (see e.g. \cite[Theorem 1]{GO14}) and Sobolev embedding theorem that
\begin{align*}
\norm{\Lambda^s(fg)}{\NLeb{2}}&\apprle \norm{f}{\NLeb{p}}\norm{\Lambda^s g}{\NLeb{{2p}/{(p-2)}}}+\norm{\Lambda^s f}{\NLeb{2}}\norm{g}{\NLeb{\infty}}\\
&\apprle \norm{f}{\tSob{s}}\norm{J^s g}{\tSob{\gamma-s}}+\norm{\Lambda^s f}{\NLeb{2}}\norm{g}{\tSob{\gamma}}\\
&\apprle \norm{f}{\tSob{s}}\norm{g}{\tSob{\gamma}}.
\end{align*}

Similarly, if $s =\gamma$, then
\begin{align*}
\norm{\Lambda^s(fg)}{\NLeb{2}}&\apprle \norm{f}{\NLeb{\infty}}\norm{\Lambda^s g}{\NLeb{2}}+\norm{\Lambda^s f}{\NLeb{2}}\norm{g}{\NLeb{\infty}} \\
&\apprle \norm{f}{\tSob{s}}\norm{g}{\tSob{\gamma}}.
\end{align*}
This completes the proof of Lemma \ref{lem:commutator}.
\end{proof}

\begin{lem}\label{lem:product2}
Let    $\alpha$ and $\beta$ satisfy either one of the following   conditions:
\begin{enumerate}[label=\textnormal{(\arabic*)}]
\item  $\alpha <1$, $\beta >d/2$, and $\alpha +\beta \ge 1$.
\item $1 \le \alpha <d/2 +1$  and $\alpha +2 \beta \ge d/2 +1$.
\end{enumerate}
Then for all $f , g\in \tSob{\beta}$,
\[	\norm{fg}{\thSob{1-\alpha}}\leq C \norm{f}{\tSob{\beta}}\norm{g}{\tSob{\beta}}, \]
where $C=C(d,\alpha , \beta )>0$.
\end{lem}
\begin{proof}
Suppose that  $\alpha <1$, $\beta >d/2$, and $\alpha +\beta \ge 1$. Then since  $0< 1-\alpha \le \beta$, it immediately follows from  Lemma \ref{lem:commutator} that
\[
\norm{fg}{\thSob{1-\alpha}}\leq C \norm{f}{\tSob{1-\alpha}}\norm{g}{\tSob{\beta}}\leq C \norm{f}{\tSob{\beta}}\norm{g}{\tSob{\beta}}.
\]

Suppose next that $1 \le \alpha < d/2 +1$ and $\alpha +2 \beta \ge d/2 +1$. Recall then that
$$	
\thSob{1-\alpha}=\left(\thSob{\alpha-1}\right)^* .
$$
Let $2 \le r<\infty$ be  defined by
$$
\frac{1}{r}= \frac{1}{2}-\frac{\alpha-1}{d}.
$$
Note that
\[
\beta>0 \quad\mbox{and}\quad \frac{1}{2 r'} = \frac{1}{2}-\frac{1}{2}\left(\frac{1}{2}-\frac{\alpha-1}{d} \right) \ge \frac{1}{2} -\frac{\beta}{d}  .
\]
Hence by  the Sobolev embedding theorem (see Lemma \ref{lem:embedding} and Proposition \ref{prop:Gagliardo} with $p_1 =2$ e.g.),  we have
\[
 \int_{\mathbb{R}^d} (fg) \phi \myd{x}
   \leq  \norm{f }{\NLeb{2r'}} \norm{g}{\NLeb{2r'}}\norm{\phi}{\NLeb{r}}
     \apprle  \norm{f}{\tSob{\beta}}  \norm{g}{\tSob{\beta}}  \norm{\phi}{\thSob{\alpha-1}}
\]
for any $\phi \in \thSob{\alpha-1}$. This completes the proof of Lemma \ref{lem:product2}.
\end{proof}

\section{Fractional Stokes equations}\label{sec:Stokes}
In this section, we consider the fractional Stokes equations in $\mathbb{R}^d , d \ge 2$:
\begin{equation}\label{frac-Stokes}
  \Lambda^{2\alpha}\boldu+\nabla {p_*} = \Div\boldF, \quad \Div\boldu=0 \quad\mbox{in}\,\, \mathbb{R}^d ,
\end{equation}
where $\alpha$ is a real number. Here $\boldF = [F^{jk}]_{1 \le j, k \le d}$ is a given matrix-valued function and   $\Div \boldF$ is the vector field whose $j$-th component is given by
\[
 (\Div \boldF)^j = \sum_{k=1}^d D_{k} F^{jk};
\]
hence
\[
\Div \Div \boldF = \sum_{j,k=1}^d D_{jk} F^{jk}
\]
and
\[
\widehat{\Div \Div \boldF}(\xi ) = (2 \pi i \xi  ) \otimes (2\pi i \xi ) : \widehat{\boldF}(\xi )  =  - 4 \pi^2 \sum_{j,k=1}^d  \xi_j  \xi_k  \widehat{F}^{jk}(\xi ) .
\]

This section consists of three subsections. In Subsection \ref{subsec:sol-Sob}, we prove   solvability results in Sobolev spaces for \eqref{frac-Stokes}. In Subsection \ref{subsec:sol-formula}, we derive a solution formula for the problem \eqref{frac-Stokes} that will be used in several places in the rest of the paper. Finally,    very weak solutions of \eqref{frac-Stokes} are studied  in Subsection \ref{subsec:very-weak}.

\subsection{Solutions in Sobolev spaces}\label{subsec:sol-Sob} The problem \eqref{frac-Stokes} is first solved in the   class $\mathscr{S}_0$ by decoupling  the fractional Stokes system    into two elliptic equations.

\begin{lem}\label{lemma:frac-Stokes-S and S'} Let $\alpha$ be any real number. Then for each $\boldF \in \mathscr{S}_0$, there exists a unique solution  $(\boldu,{p_*}) \in \mathscr{S}_0 \times \mathscr{S}_0$ of  \eqref{frac-Stokes}.
\end{lem}

\begin{proof} Suppose that  $\boldF \in \mathscr{S}_0 $. Then since $\Div \boldF \in  \mathscr{S}_0$ and $\Delta = - \Lambda^2 $ is an isomorphism from $\mathscr{S}_0$ onto itself,   there exists a unique  ${p_*} \in \mathscr{S}_0$ such that
\begin{equation}\label{Lapl-equation}
     \Delta {p_*} =\Div \Div \boldF \quad\mbox{in}\,\, \mathbb{R}^d .
\end{equation}
Let $\mathbf{I}$ denote the identity matrix. Then since $\Lambda^{2\alpha}: \mathscr{S}_0 \rightarrow \mathscr{S}_0$ is an isomorphism,  there exists a unique $\boldu \in \mathscr{S}_0$ such that
\begin{equation}\label{frac-Stokes-u only}
  \Lambda^{2\alpha}\boldu  = \Div (\boldF - {p_*} \mathbf{I} )  \quad\mbox{in}\,\, \mathbb{R}^d .
\end{equation}
Since $\Div ({p_*} \mathbf{I} ) =\nabla {p_*}$, we have
$$
\Lambda^{2\alpha}\boldu+\nabla {p_*} = \Div\boldF \quad\mbox{in}\,\, \mathbb{R}^d.
$$
 Moreover,
taking the divergence in (\ref{frac-Stokes-u only}), we have
\[ 
 \Lambda^{2\alpha} (\Div \boldu )   = \Div \Div \boldF -\Delta  {p_*}  =0   \quad\mbox{in}\,\, \mathbb{R}^d .
\] 
Since $\Div \boldu \in \mathscr{S}_0$ and $\Lambda^{2\alpha}: \mathscr{S}_0 \rightarrow \mathscr{S}_0$ is an isomorphism, it follows   that $\Div \boldu =0$ in $\mathbb{R}^d$. This proves existence of a solution in  $ \mathscr{S}_0 \times \mathscr{S}_0$ of  \eqref{frac-Stokes}.

To prove the uniqueness assertion, let    $(\boldu,{p_*}) \in \mathscr{S}_0 \times \mathscr{S}_0$ be  a solution of  \eqref{frac-Stokes} with the trivial data $\boldF= {\mathbf 0}$. Then since $\Div \boldu=0$ in $\mathbb{R}^d$ and $\Lambda^{2\alpha}\Div =\Div \Lambda^{2\alpha}$ in $\mathscr{S}_0$, we take  the divergence in (\ref{frac-Stokes}) to obtain
\[
   \Delta {p_*} = 0\quad\mbox{in}\,\, \mathbb{R}^d .
\]
Since $\Delta : \mathscr{S}_0 \rightarrow \mathscr{S}_0$ is an isomorphism, it follows that   ${p_*} =0$ in $\mathbb R^d$. Then since  $\Lambda^{2\alpha}\boldu  = \mathbf{0}$ in $\mathbb{R}^d$, we conclude that $\boldu  =\mathbf{0}$ in $\mathbb R^d$. This completes the proof of Lemma \ref{lemma:frac-Stokes-S and S'}.
\end{proof}

Using the decomposition of \eqref{frac-Stokes} into  \eqref{Lapl-equation} and \eqref{frac-Stokes-u only}, we   prove unique solvability of \eqref{frac-Stokes} in homogeneous Sobolev spaces of arbitrary orders.

\begin{thm}\label{thm:fractional-Stokes0}
Let $\alpha , s\in \mathbb{R}$. Then for each $\boldF \in \thSob{s}$, there exists a unique solution  $(\boldu,{p_*})$ in $\thSob{2\alpha+s-1}\times \thSob{s}$   of  \eqref{frac-Stokes}.
Moreover, we have
\[    \norm{\boldu}{\thSob{2\alpha+s-1}}+\norm{{p_*}}{\thSob{s}}\apprle \norm{\boldF}{\thSob{s}}. \]
\end{thm}
\begin{proof} Suppose that $\boldF \in \thSob{s}$. Then since $\Div \Div \boldF \in \thSob{s-2}$ and $\Delta$ is an isometric  isomorphism from $\thSob{t+2}$ to $\thSob{t}$ for any $t \in \mathbb{R}$, there exists a unique solution ${p_*} \in \thSob{s}$  of \eqref{Lapl-equation}, which satisfies
\[	\norm{{p_*}}{\thSob{s}}\apprle \norm{\boldF}{\thSob{s}}.\]
Moreover, since $\Lambda^{2\alpha}$ is an isomorphism from $\thSob{2\alpha+s-1}$ to $\thSob{s-1}$,  there exists a unique $\boldu\in \thSob{2\alpha+s-1}$ satisfying \eqref{frac-Stokes-u only}. Note also that
\[	\norm{\boldu}{\thSob{2\alpha+s-1}}\apprle \norm{\boldF}{\thSob{s}}+\norm{{p_*}}{\thSob{s}}\apprle \norm{\boldF}{\thSob{s}}.\]
Finally, since $\Div \boldu \in \thSob{2\alpha+s-2}$   and $ \Lambda^{2\alpha} (\Div \boldu )=0$ in $\thSob{s-2}$, it follows that $\Div \boldu=0$ in $\mathbb{R}^d$. This proves the existence assertion of the theorem.

To show the uniqueness, suppose that $(\boldu,{p_*})\in \thSob{2\alpha+s-1}\times \thSob{s}$ satisfies \eqref{frac-Stokes} with $\boldF=\mathbf{0}$. Since $\Div \boldu=0$ in $\mathbb{R}^d$ and $\Lambda^{2\alpha}\Div =\Div \Lambda^{2\alpha}$ in $\mathscr{S}_0'$, it follows from (\ref{frac-Stokes}) that $\Delta {p_*} =0$ in $\mathbb{R}^d$. Since $\Delta$ is an isomorphism, we then deduce  that ${p_*}=0$  in $\mathbb{R}^d$. From  \eqref{frac-Stokes} again, we have  $\Lambda^{2\alpha}\boldu=\mathbf{0}$   in $\mathbb{R}^d$. Since $\Lambda^{2\alpha}$ is an isomorphism, we then conclude that $\boldu=\mathbf{0}$   in $\mathbb{R}^d$.
 This completes the proof of Theorem \ref{thm:fractional-Stokes0}.
\end{proof}

\begin{thm}\label{thm:fractional-Stokes}
Let $\alpha $ be any real number. Then for each $\boldF \in \thSob{1-\alpha}$, there exists a unique solution $(\boldu,{p_*})$ in $\thSob{\alpha}\times \thSob{1-\alpha}$   of  \eqref{frac-Stokes}.
Moreover, we have
\[    \norm{\boldu}{\thSob{\alpha}}+\norm{{p_*}}{\thSob{1-\alpha}}\apprle \norm{\boldF}{\thSob{1-\alpha}}. \]
In addition,   $\boldu$   satisfies the energy identity
\begin{equation}\label{eq:divF-identity}
   \int_{\mathbb{R}^d} |\Lambda^{\alpha} \boldu|^2 \myd{x}=   \action{ \Div \boldF,  \boldu} .
\end{equation}
\end{thm}

\begin{proof}
By Theorem \ref{thm:fractional-Stokes0}, it remains to show that the solution   $(\boldu,{p_*})$ in $\thSob{\alpha}\times \thSob{1-\alpha}$  of \eqref{frac-Stokes} satisfies the energy identity \eqref{eq:divF-identity}.
Note first that
\[    \int_{\mathbb{R}^d} \Lambda^\alpha \boldu \cdot \Lambda^\alpha \Phi \myd{x}-\action{{p_*},\Div \Phi}=\action{\Lambda^{2\alpha}\boldu, \Phi}+\action{\nabla {p_*},\Phi}=   \action{ \Div \boldF,   \Phi}   \]
for all $\Phi \in \mathscr{S}_0$.
By Lemma \ref{lem:density-approx-S0}, there exists  a sequence $\{\boldu_k\}$ in $\mathscr{S}_0$ such that $\norm{\boldu_k-\boldu}{\thSob{\alpha}}\rightarrow 0$. For each $k$, we have
\[    \int_{\mathbb{R}^d} \Lambda^\alpha \boldu \cdot \Lambda^\alpha  \boldu_k \myd{x}-\action{{p_*},\Div  \boldu_k}=   \action{ \Div \boldF,  \boldu_k} .  \]
Hence by the duality inequality (\ref{dual-Hspace}), $\boldu$ satisfies
\[    \int_{\mathbb{R}^d} |\Lambda^\alpha \boldu|^2 \myd{x}-\action{{p_*},\Div \boldu}=  \action{ \Div \boldF,  \boldu}. \]
Since $\Div \boldu=0$, we get the energy identity \eqref{eq:divF-identity}.
 This completes the proof of Theorem \ref{thm:fractional-Stokes}.
\end{proof}

\subsection{A solution formula}\label{subsec:sol-formula}

Given  $\boldF = [F^{jk}]_{1 \le j, k \le d} \in \mathscr{S}_0$, let $(\boldu ,{p_*}) \in \mathscr{S}_0 \times \mathscr{S}_0$ be the solution of (\ref{frac-Stokes}). First, taking the Fourier transform in (\ref{Lapl-equation}), we have
\[
 -(2\pi   |\xi|)^{2  } \widehat{{p_*}}(\xi ) = (2 \pi i \xi  ) \otimes (2\pi i \xi ) : \widehat{\boldF}(\xi ) \quad\mbox{for all } \xi \in {\mathbb R}^d
\]
or equivalently
\[
  \widehat{{p_*}}(\xi ) =  \frac{ \xi  \otimes  \xi}{ |\xi|^{2}}  : \widehat{\boldF}(\xi ) \quad\mbox{for all } \xi \neq 0 .
\]
Next, taking the Fourier transform in (\ref{frac-Stokes-u only}), we have
\begin{align*}
(2\pi |\xi|)^{2 \alpha } \widehat{u}^j(\xi ) &= \sum_{l=1}^d (2 \pi i \xi_l )  \widehat{F}^{jl}(\xi)  - (2 \pi i \xi_j )   \widehat{{p_*}}(\xi )  \\
&= \sum_{k,l=1}^d (2 \pi i \xi_l ) \left( \delta_{jk} - \frac{\xi_j  \xi_k}{|\xi|^2} \right)  \widehat{F}^{kl}(\xi)
\end{align*}
and so
\[
  \widehat{u}^j (\xi ) =  \sum_{k,l=1}^d \frac{(2 \pi i \xi_l )}{(2\pi |\xi|)^{2 \alpha }} \left( \delta_{jk} - \frac{\xi_j  \xi_k}{|\xi|^2} \right)  \widehat{F}^{kl}(\xi)
\]
for $j =1 , \ldots , d$.

Assuming  that $   2\alpha <  d+1 $, let $U_{\alpha,j}^{kl} \in \mathscr{S}'$ be the inverse Fourier transform of the locally integrable function
\[
\frac{(2 \pi i \xi_l )}{(2\pi |\xi|)^{2 \alpha }} \left( \delta_{jk} - \frac{\xi_j  \xi_k}{|\xi|^2} \right)
\]
in $\mathscr{S}'$.  The following lemma will be proved later in Appendix \ref{sec:Gaussian}.

\begin{lem}\label{prop:Fourier-transform-Gaussian-derivatives}
Suppose that  $j, k, l   \in \{1 , ..., d \}$, $\lambda \in \mathbb C$, and $\psi \in \mathscr{S}$. Let $\Gamma$ be the Gamma function.

\begin{enumerate}[label=\textnormal{(\arabic*)}]
\item  If $1<\Re \lambda<d+1$, then
\[
\int_{\mathbb{R}^d} \frac{-2 \pi i \xi_j}{|\xi|^{\lambda}}  {\psi}^{\vee}(\xi) \, d\xi
   =\frac{2\Gamma((d+2-\lambda)/2)  }{  \Gamma( \lambda/2)\pi^{d/2- \lambda }  } \int_{\mathbb{R}^d}  \frac{x_j}{|x|^{d+2-\lambda }} \psi(x) \, dx.
\]
\item  If $2<\Re \lambda<d+2$, then
\begin{align*}
& \int_{\mathbb{R}^d}   \frac{(-2 \pi i \xi_j )(-2 \pi i \xi_k )}{|\xi|^\lambda}  {\psi}^{\vee}(\xi) \, d\xi \\
&\quad  = \frac{2\Gamma((d+2-\lambda)/2)  }{  \Gamma( \lambda/2) \pi^{d/2-\lambda}}   \int_{\mathbb{R}^d}   \frac{(d+2-\lambda )x_j x_k -\delta_{jk} |x|^2}{|x|^{d+4-\lambda }}   {\psi}(x) \, dx  .
\end{align*}
\item  If $3<\Re \lambda<d+3$, then
\begin{align*}
 & \int_{\mathbb{R}^d}   \frac{(-2\pi i \xi_j)(-2\pi i \xi_k)(-2\pi i \xi_l)}{|\xi|^\lambda} {\psi}^{\vee}(\xi) \, d\xi \\
  & \quad  = \frac{4\Gamma((d+4-\lambda)/2)  }{ \Gamma( \lambda/2) \pi^{d/2-\lambda}}   \int_{\mathbb{R}^d} \left[- D_l \left( \frac{   x_j x_k}{|x|^{d+4-\lambda }} \right) - \frac{\delta_{jk}   x_l }{|x|^{d+4-\lambda }} \right]   {\psi}(x)\myd{x} .
\end{align*}
\end{enumerate}
\end{lem}

Assume that $1< 2\alpha <  d+1  $.
Then using Lemma \ref{prop:Fourier-transform-Gaussian-derivatives}, we can find  $U_{\alpha,j}^{kl}$ explicitly. First of all, it follows from  Lemma \ref{prop:Fourier-transform-Gaussian-derivatives} (1) that the inverse Fourier transform of $(2 \pi i \xi_l )(2\pi |\xi|)^{-2 \alpha }$ is
\[
   - \frac{2 \Gamma(1+d/2-\alpha) }{2^{2 \alpha} \pi^{d/2}   \Gamma( \alpha) } \frac{x_l}{|x|^{d+2-2\alpha}} .
\]
Next, by Lemma \ref{prop:Fourier-transform-Gaussian-derivatives} (3) with $\lambda =2\alpha +2$, we have
\begin{align*}
 & \int_{\mathbb{R}^d}   \frac{2\pi i \xi_j  \xi_k  \xi_l }{|\xi|^{2\alpha+2} } {\psi}^{\vee}(\xi) \, d\xi \\
  & \quad  = \frac{ \Gamma(1+d/2 -\alpha)  }{  \Gamma(1+\alpha) \pi^{d/2-2\alpha}}   \int_{\mathbb{R}^d} \left[- D_l \left( \frac{   x_j x_k}{|x|^{d+2-2\alpha }} \right) - \frac{\delta_{jk}   x_l }{|x|^{d+2-2\alpha }} \right]   {\psi}(x)\myd{x}
\end{align*}
for all $\psi \in \mathscr{S}$. Hence the inverse Fourier transform of $(2 \pi i \xi_l \xi_j \xi_k )(2\pi |\xi|)^{-2 \alpha }|\xi|^{-2}$ is
\[
  \frac{ \Gamma(1+d/2-\alpha) }{2^{2 \alpha} \pi^{d/2}  \Gamma (1+ \alpha)} \left[ - D_l \left( \frac{   x_j x_k}{|x|^{d+2-2\alpha }} \right) -  \frac{\delta_{jk}   x_l }{|x|^{d+2-2\alpha }} \right].
\]
Therefore,
\begin{equation}\label{eq:Stokes-Green}
U_{\alpha,j}^{kl}(x) = \frac{  \Gamma(1+d/2-\alpha) }{2^{2 \alpha} \pi^{d/2}\Gamma(1+ \alpha) }\left[   D_l \left( \frac{   x_j x_k}{|x|^{d+2-2\alpha }} \right) - \frac{\delta_{jk}(   2\alpha -1 )  x_l }{|x|^{d+2-2\alpha }} \right] .
\end{equation}
Note that
$$
U_{\alpha ,j}^{kl} \in \Lor{d/(d+1-2\alpha)}{\infty} .
$$
Now, setting   $\mathbf{U}_{\alpha} = [U_{\alpha, j}^{kl}]_{1 \le j,k,l \le d}$, we have
\[
\boldu = \mathbf{U}_{\alpha} * \boldF
\]
or   componentwise
\[
u^j  = \sum_{k,l=1}^d U_{\alpha , j}^{kl} * F^{kl} \quad (j=1,\ldots , d) .
\]

\subsection{Very weak solutions}\label{subsec:very-weak}

The notion of  very weak solutions is introduced to solve  the fractional Stokes system \eqref{frac-Stokes} for $\boldF \in \Leb{1}$. Recall that if $s \ge 0$, then  $\Lambda^s$ is bounded from $\mathscr{S}$ to $\Leb{1}\cap \Leb{\infty}$.

\begin{defn} Let $\alpha \ge 0$. Suppose that  $\boldF$ is a distribution. Then we say that $\boldu \in \Leb{1}+\Leb{\infty}$ is a \emph{very weak solution} of \eqref{frac-Stokes} if
\begin{equation}\label{eq:very-weaksol}
   \int_{\mathbb{R}^d} \boldu \cdot \left(  \Lambda^{2\alpha} \Phi +\nabla \psi \right)  dx =- \action{ \boldF ,\nabla \Phi}
\end{equation}
for all $\Phi \in C_{c,\sigma}^\infty$ and $\psi\in C_c^\infty$.
\end{defn}
Indeed, taking $\Phi=0$ in (\ref{eq:very-weaksol}), we deduce that if $\boldu$ is a very weak solution of \eqref{frac-Stokes}, then it is divergence-free:
\[
\int_{\mathbb{R}^d} \boldu \cdot  \nabla \psi   \myd{x}=0 \quad\mbox{for all}\,\, \psi\in C_c^\infty .
\]

Existence of a pressure associated with a very weak solution of \eqref{frac-Stokes} is an easy  consequence of the following more or less standard result, which will be proved later in Appendix \ref{sec:deRham}.

\begin{lem}\label{lem:de-Rham}
If $\boldu \in \mathscr{S}'$   satisfies  $\action{\boldu,\Phi}=0$ for all $\Phi \in C_{c,\sigma}^\infty$, then there exists    $p \in \mathscr{S}'$ such that $\boldu=\nabla p$.
\end{lem}

\begin{lem}\label{lem:existence of p}
Let $\alpha \geq 0$. If $\boldu$ is a very weak solution of \eqref{frac-Stokes} with $\boldF \in \mathscr{S}'$, then
there exists ${p_*}\in \mathscr{S}'$ such that
\[
  \int_{\mathbb{R}^d} \boldu \cdot \left(  \Lambda^{2\alpha} \Phi +\nabla \psi \right) dx  -\action{{p_*},\Div \Phi}=- \action{ \boldF ,\nabla \Phi}
\]
for all $( \Phi , \psi ) \in \mathscr{S}$.
\end{lem}
\begin{proof}
Since $C_c^\infty$ is dense in $\mathscr{S}$, we easily have
\[
\int_{\mathbb{R}^d} \boldu \cdot  \nabla \psi   \myd{x}=0 \quad\mbox{for all}\,\, \psi  \in \mathscr{S}.
\]
Note that $\Lambda^{2\alpha}\boldu -\Div \boldF \in \mathscr{S}'$. Moreover, for all $\Phi \in C_{c,\sigma}^\infty$, we have
$$
\action{\Lambda^{2\alpha}\boldu -\Div \boldF,\Phi}=\int_{\mathbb{R}^d} \boldu \cdot   \Lambda^{2\alpha} \Phi   \myd{x} +  \action{ \boldF ,\nabla \Phi} = 0 .
$$
 Hence by Lemma \ref{lem:de-Rham}, there exists ${p_*}\in \mathscr{S}'$ such that
$$
\int_{\mathbb{R}^d} \boldu \cdot   \Lambda^{2\alpha} \Phi   \myd{x} +  \action{ \boldF ,\nabla \Phi} = \action{\Lambda^{2\alpha}\boldu -\Div \boldF ,\Phi}= \action{{p_*},\Div \Phi}
$$
for all $\Phi \in \mathscr{S}$. This completes the proof of Lemma \ref{lem:existence of p}.
\end{proof}

We are now able to prove uniqueness of very weak solutions of \eqref{eq:very-weaksol}.

\begin{lem}\label{lem:Liouville-type-theorem}
Let $\alpha \geq 0$. If $\boldu\in \Leb{1}+\Leb{\infty}$ is a very weak solution of \eqref{frac-Stokes}   with $\boldF=\mathbf{0}$, then $\boldu$ is constant. In addition, if  $\boldu \in \Lor{q}{\infty}$  for some $1<q<\infty$, then  $\boldu$ is identically zero.
\end{lem}
\begin{proof}
By Lemma \ref{lem:existence of p}, there exists ${p_*}\in \mathscr{S}'$ such that
$$
\int_{\mathbb{R}^d} \boldu \cdot   \Lambda^{2\alpha} \Phi   \myd{x}-\action{{p_*},\Div \Phi}  =0 \quad\mbox{and}\quad \int_{\mathbb{R}^d} \boldu \cdot  \nabla \psi   \myd{x}=0
$$
for all $( \Phi , \psi ) \in \mathscr{S}$.

Let $\psi\in \mathscr{S}_0$ be given. Then since $\nabla \psi \in \mathscr{S}_0$, $\Lambda^{2\alpha}\psi \in \mathscr{S}_0$,  and $\Lambda^{2\alpha}\nabla \psi = \nabla \Lambda^{2\alpha}\psi$, we have
\[    \action{{p_*} ,\Delta \psi}= \int_{\mathbb{R}^d} \boldu \cdot   \Lambda^{2\alpha} \nabla \psi   \myd{x}=   \int_{\mathbb{R}^d} \boldu \cdot \nabla (\Lambda^{2\alpha}\psi) \myd{x}=0.\]
Since $\Delta = -\Lambda^2  : \mathscr{S}_0 \rightarrow \mathscr{S}_0$ is an isomorphism, it follows from Lemma \ref{lem:polynomial-isomorphism} that  ${p_*}$ is a polynomial. Hence, for all $\Phi \in \mathscr{S}_0$, we have
\[  \int_{\mathbb{R}^d} \boldu \cdot   \Lambda^{2\alpha} \Phi   \myd{x}= \action{{p_*},\Div \Phi}=0 .\]
Therefore, by Lemma \ref{lem:polynomial-isomorphism} again,  $\boldu$ is a polynomial. It is easy to check that  polynomials in $\Leb{1}+\Leb{\infty}$ are constant and the only constant function   belonging to $\Lor{q}{\infty}$ is the zero function. This  completes  the proof of Lemma \ref{lem:Liouville-type-theorem}.
\end{proof}

The following theorem shows   existence and uniqueness of very weak solutions in $\Leb{q}$ of \eqref{frac-Stokes} when $\boldF \in \Leb{p}$, where $1<p<d/(2\alpha-1)$ and $1/q=1/p-(2\alpha-1)/d$.

\begin{thm}\label{theorem:Frational Stokes-Lpdata}
Let $1/2<\alpha<(d+1)/2$, $1<  p<d/(2\alpha-1)$, and $1/q=1/p-(2\alpha-1)/d$.
Then for each $\boldF\in \Leb{p}$, there exists a unique very weak solution $\boldu$ in   $\Leb{q}$ of \eqref{frac-Stokes}. Moreover, the solution $\boldu$ is given by $\boldu=\mathbf{U}_\alpha * \boldF$  and satisfies
$$
\norm{\boldu }{\NLeb{q}}\apprle \norm{\boldF}{\NLeb{p}}.
$$
In addition, if $\boldF\in   \thSob{1-\alpha}$, then $\boldu $ belongs to $ \thSob{\alpha}$ and satisfies the energy identity \eqref{eq:divF-identity}.
\end{thm}
\begin{proof} Define $\boldu=\mathbf{U}_\alpha * \boldF$. Then since $\mathbf{U}_\alpha \in \Lor{d/(d+1-2\alpha)}{\infty}$, it follows from   Lemma \ref{lem:Young}  that $\boldu \in \Leb{q}$ and $\norm{\boldu }{\NLeb{q}}\apprle \norm{\boldF}{\NLeb{p}}$.

We now show that $\boldu$ is a   very weak solution   of \eqref{frac-Stokes}.
To show this, choose a sequence $\{\boldF_l\}$ in $\mathscr{S}_0$ such that $\boldF_l \rightarrow \boldF$ in $\Leb{p}$.
By Lemma \ref{lemma:frac-Stokes-S and S'}, there exists a unique solution  $( \boldu_l , p_{*, l} ) \in \mathscr{S}_0 \times \mathscr{S}_0$ of  \eqref{frac-Stokes} with $\boldF = \boldF_l$. Moreover, it was shown in Subsection 3.2 that
\[     \boldu_l=\mathbf{U}_\alpha *\boldF_l.\]
For all $(\Phi , \psi ) \in \mathscr{S}_0$, we have
\[
  \int_{\mathbb{R}^d} \boldu_l \cdot \Lambda^{2\alpha} \Phi \myd{x} =\int_{\mathbb{R}^d} (\Lambda^{2\alpha} \boldu_l)\cdot \Phi \myd{x}=-\int_{\mathbb{R}^d} (\boldF_l : \nabla \Phi +  \nabla p_{*,l} \cdot \Phi) \myd{x}
\]
and
\[
\int_{\mathbb{R}^d}  \boldu_l \cdot \nabla \psi  \myd{x} =0 .
\]
Since $\boldu_l, {p_*}_l, \boldF_l\in \mathscr{S}_0$, it follows from Lemma \ref{lem:density-approx-S0} that
\begin{align*}
  \int_{\mathbb{R}^d} \boldu_l \cdot \left( \Lambda^{2\alpha} \Phi + \nabla \psi \right) dx &=-\int_{\mathbb{R}^d} (\boldF_l : \nabla \Phi +  \nabla p_{*,l} \cdot \Phi  )  \myd{x} \\
\intertext{for any $(\Phi ,\psi ) \in \mathscr{S}$. Hence for all $\Phi \in C_{c,\sigma}^\infty$ and $\psi \in C_c^\infty$,  we get}
  \int_{\mathbb{R}^d} \boldu_l \cdot  \left(\Lambda^{2\alpha} \Phi +\nabla \psi\right) dx &=-\int_{\mathbb{R}^d} \boldF_l :\nabla \Phi \myd{x}.
\end{align*}
Since $\boldF_l\rightarrow \boldF$ in $\Leb{p}$, it follows from   Lemma \ref{lem:Young} that $\boldu_l\rightarrow \boldu$ in $\Leb{q}$. Therefore, for all $\Phi \in C_{c,\sigma}^\infty$  and $\psi \in C_c^\infty$, we have
\[  \int_{\mathbb{R}^d} \boldu  \cdot \left( \Lambda^{2\alpha} \Phi +\nabla \psi \right)  dx=-\int_{\mathbb{R}^d} \boldF  :\nabla \Phi \myd{x}. \]
Hence  $\boldu$ is a   very weak solution of \eqref{frac-Stokes}, which is unique due to  Lemma \ref{lem:Liouville-type-theorem}.

To complete the proof, suppose in addition that  $\boldF\in   \thSob{1-\alpha}$. Then by Lemma \ref{lem:density-approx-S0}, there is a sequence $\{\boldG_l\}$ in $\mathscr{S}_0$ such that $\boldG_l \rightarrow \boldF$    in both $\Leb{p}$ and  $\thSob{1-\alpha}$. By Lemma \ref{lemma:frac-Stokes-S and S'}, there exists a unique solution  $( \boldv_l , \tilde{p}_{*, l} ) \in \mathscr{S}_0 \times \mathscr{S}_0$ of  \eqref{frac-Stokes} with $\boldF = \boldG_l$. Moreover, by Theorem \ref{thm:fractional-Stokes}, $\{ \boldv_l \}$ converges in $\thSob{\alpha}$ to some  $\boldv \in \thSob{\alpha}$ satisfying
\[
 \int_{\mathbb{R}^d} |\Lambda^{\alpha} \boldv|^2 \myd{x}=  \action{ \Div \boldF, \boldv} .
\]
On the other hand, it was already shown that
$\boldv_l =\mathbf{U}_\alpha *\boldG_l$ and $\boldv_l \to \boldu$ in $\Leb{q}$.
Hence,  for all $\Phi \in \mathscr{S}$, we have
\begin{align*}
\int_{\mathbb{R}^d}\boldu \cdot \Lambda^{\alpha} \Phi \,dx & = \lim_{l\to\infty} \int_{\mathbb{R}^d}\boldv_l \cdot \Lambda^{\alpha} \Phi \,dx \\
&= \lim_{l\to\infty} \int_{\mathbb{R}^d}\Lambda^{\alpha} \boldv_l \cdot \Phi \,dx
= \int_{\mathbb{R}^d}\Lambda^{\alpha} \boldv  \cdot \Phi \,dx ,
\end{align*}
which implies that $\Lambda^{\alpha} \boldu = \Lambda^{\alpha} \boldv$ in $\mathscr{S}'$. Since $\Lambda^\alpha   : \mathscr{S}_0 \rightarrow \mathscr{S}_0$ is an isomorphism, we also have
\[
\int_{\mathbb{R}^d}  \boldu \cdot  \Phi \, dx  = \int_{\mathbb{R}^d} \boldv \cdot \Phi \, dx  \quad\mbox{for all}\,\, \Phi \in \mathscr{S}_0 .
\]
Therefore,
\[
\action{ {\rm div}\, \boldF , \boldu  } = \lim_{l\to\infty} \int_{\mathbb{R}^d} \boldu \cdot  {\rm div}\, \boldG_l \, dx = \lim_{l\to\infty} \int_{\mathbb{R}^d} \boldv \cdot {\rm div}\, \boldG_l \, dx  = \action{   {\rm div}\, \boldF , \boldv }.
\]
This shows that $\boldu$ satisfies the energy identity \eqref{eq:divF-identity}. The proof of Theorem \ref{theorem:Frational Stokes-Lpdata} is completed.
\end{proof}

The following theorem shows   existence and uniqueness of very weak solutions in $\Lor{d/(d+1-2\alpha)}{\infty}$ of \eqref{frac-Stokes} when $\boldF \in \Leb{1}$.

\begin{thm}\label{theorem:Frational Stokes-Lpdata-intersection}
Let $1/2<\alpha<(d+1)/2$. Then for each  $\boldF\in \Leb{1} $,
there exists a unique very weak solution  $\boldu $   of \eqref{frac-Stokes}  in $\Lor{d/(d+1-2\alpha)}{\infty}$. Moreover, the solution $\boldu$ is  given by
$\boldu=\mathbf{U}_\alpha * \boldF$ and satisfies
$$
\norm{  \boldu}{\NLor{d/(d+1-2\alpha)}{\infty}}\apprle \norm{{\boldF}}{\NLeb{1}}.
$$
\end{thm}

\begin{proof} Given  $\boldF\in \Leb{1} $, we define $\boldu=\mathbf{U}_\alpha * \boldF$. Then it follows from   Lemma   \ref{lem:Young} that  $\boldu \in \Lor{d/(d+1-2\alpha)}{\infty}$ and
$\norm{  \boldu}{\NLor{d/(d+1-2\alpha)}{\infty}}\apprle \norm{{\boldF}}{\NLeb{1}}$. Hence  by   Lemma \ref{lem:Liouville-type-theorem}, it remains to show  that $\boldu$ is a   very weak solution of \eqref{frac-Stokes}.

 Fixing any $r$ with $1<r<d/(2\alpha-1)$, we choose a sequence $\{\boldG_l\}$ in $\Leb{1}\cap \Leb{r}$ such that $\boldG_l\rightarrow \boldF$ in $\Leb{1}$. Then by Theorem \ref{theorem:Frational Stokes-Lpdata}, $\boldv_l = \mathbf{U}_\alpha * \boldG_l$ satisfies
\[     \int_{\mathbb{R}^d} \boldv_l \cdot ( \Lambda^{2\alpha} \Phi + \nabla \psi ) \myd{x}=-\int_{\mathbb{R}^d} \boldG_l :\nabla \Phi \myd{x}   \]
for all $\Phi \in C_{c,\sigma}^\infty$ and $\psi \in C_c^\infty$. By Lemma   \ref{lem:Young} again,  $\boldv_l \rightarrow \boldu$ in $\Lor{d/(d+1-2\alpha)}{\infty}$. Hence for all $\Phi \in C_{c,\sigma}^\infty$ and $\psi \in C_c^\infty$, we obtain
\[     \int_{\mathbb{R}^d} \boldu \cdot ( \Lambda^{2\alpha} \Phi  +\nabla \psi ) \myd{x}=-\int_{\mathbb{R}^d} \boldF:\nabla \Phi \myd{x}.  \]
This completes the proof of Theorem  \ref{theorem:Frational Stokes-Lpdata-intersection}. 
\end{proof}

\section{Perturbed fractional heat equations}\label{sec:perturbed-heat}

For $0<T<\infty$, we consider the following initial value  problem for perturbed fractional heat equations in $\mathbb{R}^d , d \ge 2$:
\begin{equation}\label{eq:PFHE}
\left\{
\begin{alignedat}{2}
	\partial_t\boldb +\Lambda^{2\beta}\boldb+\Div\dsp(\boldu\otimes \boldb) &=\Div\boldF &&\quad \text{in } \mathbb{R}^d\times (0,T), \\
	\boldb(\cdot,0) &=\boldb_0 &&\quad \text{on } \mathbb{R}^d ,
\end{alignedat}
\right.
\end{equation}
where  $\beta$ is a positive number and  $\boldu:\mathbb{R}^d\times (0,T)\rightarrow \mathbb{R}^d$ is a given   vector field with $\Div\boldu=0$.

By a \emph{weak solution} of \eqref{eq:PFHE}, we mean a  vector field $\boldb\in \Leb{\infty}(0,T;\Leb{2})\cap \Leb{2}(0,T;\tSob{\beta})$ satisfying
\begin{align*}
 \int_0^T \int_{\mathbb{R}^d} \boldb \cdot(-\partial_t \Phi  +  \Lambda^{2\beta}\Phi) \myd{x}dt\, & -\int_0^T\int_{\mathbb{R}^d} (\boldu\otimes \boldb):\nabla \Phi \myd{x}dt  \\
  = & -\int_0^T\int_{\mathbb{R}^d} \boldF:\nabla \Phi \myd{x}dt+ \int_{\mathbb{R}^d} \boldb_0 \cdot \Phi( 0) \,dx
 \end{align*}
for all $\Phi \in C_c^\infty (\mathbb{R}^d\times [0,T))$, provided that the integrals are all well-defined. Note that if $\beta \geq 1$, then by Proposition \ref{prop:Gagliardo},
\begin{equation}\label{eq:our-interpolation}
	\norm{\nabla\boldb}{\NLeb{2}} \apprle  \norm{\boldb}{\NLeb{2}}^{1-1/\beta}\norm{\Lambda^\beta \boldb}{\NLeb{2}}^{1/\beta}.
\end{equation}
Hence if $\beta\geq 1$, then every weak solution   of \eqref{eq:PFHE} belongs to $\Leb{2\beta}(0,T;\thSob{1})$.

This section consists of two parts. We first show that if  $\boldF \in \Leb{2\beta/(2\beta -1)}(0,T;\Leb{2})$, $\beta\geq 1$, and $\boldb_0 \in \Leb{2}$,  then there exists a unique weak solution  of \eqref{eq:PFHE} with $\boldu=0$  satisfying the energy identity. Next, extending this result to nonzero $\boldu$, we  prove that if $\beta \ge 1$ and  $\boldu \in \Leb{\infty}(0,T;\Lor{d/(d+1-2\alpha)}{\infty})\cap \Leb{2}(0,T;\thSob{\alpha})$, then for every $\boldF \in \Leb{2\beta/(2\beta -1)}(0,T;\Leb{2})$ and $\boldb_0 \in \Leb{2}$,  there exists a unique weak solution $\boldb$ of  \eqref{eq:PFHE}. Moreover, the solution $\boldb$    satisfies the energy identity, which  will be used to prove   uniqueness of weak solutions of the original nonlinear problem \eqref{eq:MHD-relaxed}.

To prove   existence and uniqueness of weak solutions  of \eqref{eq:PFHE} when $\boldu=0$, we introduce the Fourier truncation operators $\mathcal{S}_R$ ($R>0$) defined by
\[   \widehat{\mathcal{S}_R f}(\xi) = \chi_{B_R}(\xi)\hat{f}(\xi) \quad (f \in L_2 ),\]
where $\chi_{B_R}$ is the characteristic function of   $B_R$.   The  following  properties of $\mathcal{S}_R$ are quite  elementary.
\begin{lem}\label{lem:truncation-operator} Let   $0 \le s  <  t$.
\begin{enumerate}[label=\textnormal{(\arabic*)}]
\item  If $f\in \Leb{2}$, then
$\mathcal{S}_R f \in H^s$ and $\norm{\mathcal{S}_R f}{\tSob{s}} \apprle (1+R)^{s} \norm{f}{\NLeb{2}}$.
\item  If $f\in \tSob{s}$, then $\mathcal{S}_R f\rightarrow f$ in $\tSob{s}$ as $R\rightarrow \infty$
\item  If $f\in H^t$, then
$\norm{\mathcal{S}_R f -f}{\tSob{s}}\apprle (1+R)^{-(t-s )} \norm{f}{\tSob{t}}$.
\end{enumerate}
\end{lem}
For $R>0$, we denote by $V_{R}$ the space of all vector fields $\boldb \in \Leb{2}$ such that  $\supp \hat{\boldb}\subset \overline{B_R}$.  Note that $V_{R}$ is a closed subspace of $\Leb{2}$. Moreover, since $\mathcal{S}_R\boldb=\boldb$ for any $\boldb \in V_{R}$, it follows from Lemma  \ref{lem:truncation-operator} that $V_{R}$ is continuously embedded into $\tSob{s}$ for any $s\geq 0$.

\begin{lem}\label{lem:linear-heat}
Suppose that $\beta\geq 1$ and $\gamma=2\beta/(2\beta-1)$. Then for every $\boldF\in \Leb{\gamma}(0,T;\Leb{2})$ and $\boldb_0 \in \Leb{2}$, there exists a unique weak solution $\boldb$ of \eqref{eq:PFHE} with $\boldu=0$. Moreover, $\boldb$ belongs to $C([0,T];\Leb{2})$ and satisfies the energy identity:
\[
\norm{\boldb(t)}{\NLeb{2}}^2+2\int_0^t\norm{\Lambda^\beta \boldb(s)}{\NLeb{2}}^2 \, ds=-2\int_0^t \int_{\mathbb{R}^d} \boldF : \nabla \boldb \myd{x}ds + \norm{\boldb_0}{2}^2
\]
for   $0\leq t\leq T$.
\end{lem}

\begin{proof}
Choose a sequence $\{\boldF_k\}$ in  $C ([0,T];\Leb{2})$ such that $\boldF_k \rightarrow \boldF$ in $\Leb{\gamma}(0,T;\Leb{2})$ and $\norm{\boldF_k(t)}{\NLeb{2}}\leq \norm{\boldF(t)}{\NLeb{2}}$ for all $t$ and $k$.  For each $k\in \mathbb{N}$, we consider the following truncated problem  for $\boldb_k \in C^1([0,T);V_k )$:
\begin{equation}\label{eq:PFHE-truncated}
\left\{
\begin{alignedat}{2}
	\partial_t \boldb_{k} +\Lambda^{2\beta}\boldb_{k} &=\Div  \mathcal{S}_k\boldF_k  &&\quad \text{in } \mathbb{R}^d\times (0,T) \\
	\boldb_{k}(\cdot,0) &=  \mathcal{S}_k \boldb_0 &&\quad \text{on } \mathbb{R}^d .
\end{alignedat}
\right.
\end{equation}
Since $\Lambda^{2\beta}:V_k \rightarrow V_k$ is a bounded linear operator, it follows from a standard theory of linear differential equations (see e.g. Cartan \cite[Theorem 1.9.1]{C71}) in Banach spaces that the   problem \eqref{eq:PFHE-truncated} has a unique solution $\boldb_{k} \in C^1([0,T];V_k )$.

Given $k,m\in \mathbb{N}$, write $\tilde{\boldb}=\boldb_k-\boldb_m$, $\tilde{\boldF}=\mathcal{S}_k \boldF_k -\mathcal{S}_m \boldF_m$, and $\tilde{\boldb}_0= \mathcal{S}_k \boldb_0- \mathcal{S}_m \boldb_0$. Then $\tilde{\boldb}$ satisfies
\begin{equation}\label{eq:PFHE-truncated-diff}
\left\{
\begin{alignedat}{2}
	\partial_t \tilde{\boldb} +\Lambda^{2\beta}\tilde{\boldb} &=\Div  \tilde{\boldF}  &&\quad \text{in } \mathbb{R}^d\times (0,T) \\
	\tilde{\boldb}(\cdot,0) &=\tilde{\boldb}_0 &&\quad \text{on } \mathbb{R}^d .
\end{alignedat}
\right.
\end{equation}
Multiplying \eqref{eq:PFHE-truncated-diff} by $\tilde{\boldb}$ and then integrating over $\mathbb{R}^d\times [0,t]$, we have
\[
	\frac{1}{2}\norm{\tilde{\boldb}(t)}{\NLeb{2}}^2 +\int_0^t \norm{\Lambda^\beta \tilde{\boldb}(s)}{\NLeb{2}}^2 \, \myd{s} =-\int_0^t\int_{\mathbb{R}^d} \tilde{\boldF} : \nabla \tilde{\boldb} \,\myd{x}ds + \frac{1}{2}\norm{\tilde{\boldb}_0}{\NLeb{2}}^2
\]
for all $0\leq t\leq T$. By Lemma \ref{lem:truncation-operator}, \eqref{eq:our-interpolation}, and Young's inequality, we have
\begin{align*}
- \int_{\mathbb{R}^d} \tilde{\boldF} : \nabla \tilde{\boldb}\myd{x} 
    &\leq C\norm{\tilde{\boldF}}{\NLeb{2}}
    \norm{\tilde{\boldb}}{\NLeb{2}}^{1-1/\beta}\norm{\Lambda^\beta \tilde{\boldb}}{\NLeb{2}}^{1/\beta}\\
    &\leq C\norm{\tilde{\boldF}}{\NLeb{2}}^{\gamma}
    \norm{\tilde{\boldb}}{\NLeb{2}}^{2(\beta-1)/(2\beta-1)}+\frac{1}{2}\norm{\Lambda^\beta \tilde{\boldb}}{\NLeb{2}}^2 .
\end{align*}
Since $0\leq {(\beta-1)}/{(2\beta-1)}<1$, it follows that
\begin{align*}
\norm{\tilde{\boldb}(t)}{\NLeb{2}}^2+\int_0^t \norm{\Lambda^\beta \tilde{\boldb}(s)}{\NLeb{2}}^2\myd{s}&\leq \norm{\tilde{\boldb}_0}{\NLeb{2}}^2+ C \int_0^t \norm{\tilde{\boldF}(s)}{\NLeb{2}}^\gamma \norm{\tilde{\boldb}(s)}{\NLeb{2}}^{2(\beta-1)/(2\beta-1)}\myd{s}\\
&\leq \norm{\tilde{\boldb}_0}{\NLeb{2}}^2 + C \int_0^t \norm{\tilde{\boldF}(s)}{\NLeb{2}}^\gamma \left(1+\norm{\tilde{\boldb}(s)}{\NLeb{2}}^{2}\right) ds
\end{align*}
for all $0\leq t\leq T$.  Hence by Gronwall's inequality, we get
\begin{align*}	
 \sup_{0\leq t\leq T}\norm{\boldb_k(t)-\boldb_m(t)}{\NLeb{2}}^2 &+\int_0^T \norm{\Lambda^\beta (\boldb_k-\boldb_m)(t)}{\NLeb{2}}^2 \myd{t} \\
&  \leq \left(\|\mathcal{S}_k \boldb_0- \mathcal{S}_m \boldb_0 \|_2^2 +  C\int_0^T \norm{\mathcal{S}_k \boldF_k(t) -\mathcal{S}_m\boldF_m(t)}{\NLeb{2}}^\gamma \myd{t}\right) \\
& \quad \times \exp\left(C\int_0^T \norm{\mathcal{S}_k \boldF_k(t) -\mathcal{S}_m\boldF_m(t)}{\NLeb{2}}^\gamma \myd{t}\right)
\end{align*}
for some constant $C$ depending only on $d$, $\beta$, and $T$.  
By the dominated convergence theorem, we easily show that $\mathcal{S}_k \boldF_k\rightarrow \boldF$ in $\Leb{\gamma}(0,T;\Leb{2})$.   Therefore, the sequence $\{\boldb_k\}$ converges in $C([0,T];\Leb{2})\cap \Leb{2}(0,T;\tSob{\beta})$ to a function $\boldb$. It is easy  to check that $\boldb$ is a weak solution of \eqref{eq:PFHE} with $\boldu=0$. To show that $\boldb$ satisfies the energy identity, we multiply \eqref{eq:PFHE-truncated} by $\boldb_k$ to obtain
\begin{equation}\label{eq:approximation-B}
	\frac{1}{2}\norm{{\boldb}_k(t)}{\NLeb{2}}^2 +\int_0^t \norm{\Lambda^\beta  {\boldb}_k(s)}{\NLeb{2}}^2 \, \myd{s} =-\int_0^t\int_{\mathbb{R}^d} \mathcal{S}_k\boldF_k : \nabla {\boldb}_k \myd{x}ds + \frac{1}{2}\norm{\mathcal{S}_k\boldb_0}{\NLeb{2}}^2 .
\end{equation}
By the interpolation inequality \eqref{eq:our-interpolation}, we get
\begin{align*}
\int_0^T \norm{\nabla (\boldb_k-\boldb)}{\NLeb{2}}^{2\beta}\myd{t}\leq C\int_0^T \norm{\boldb_k-\boldb}{\NLeb{2}}^{2(\beta-1)}\norm{\Lambda^\beta (\boldb_k-\boldb)}{\NLeb{2}}^2 \myd{t}\rightarrow 0.
\end{align*}
Hence the energy identity follows from \eqref{eq:approximation-B} by letting $k\rightarrow \infty$.

Finally, to show the uniqueness, let $\boldb$ be a weak solution of \eqref{eq:PFHE} with $(\boldu,\boldF)=(\mathbf{0},\mathbf{0})$. Then the truncations $\boldb_k=\mathcal{S}_k\boldb$ satisfy
\[ 
\left\{
\begin{alignedat}{2}
	\partial_t \boldb_{k} +\Lambda^{2\beta}\boldb_{k} &=\mathbf{0}  &&\quad \text{in } \mathbb{R}^d\times (0,T), \\
	\boldb_{k}(\cdot,0) &=\mathbf{0}&&\quad \text{on } \mathbb{R}^d .
\end{alignedat}
\right.
\] 
It is easy to see that $\boldb_k$ is identically zero. Since $\boldb_k \to \boldb$ in $L_2 (0,T; L_2 )$, we conclude that  $\boldb$ is identically zero. This completes the proof of Lemma \ref{lem:linear-heat}.
\end{proof}

\begin{lem}\label{lem:div-free}
Let $\alpha$ and $\beta$ satisfy
\[	\frac{1}{2}<\alpha<\frac{d+1}{2},\quad \beta \geq 1,\quad \text{and}\quad \alpha+\beta\geq \frac{d}{2}+1.\]
\begin{enumerate}[label=\textnormal{(\arabic*)}]
\item For all $\boldu \in \Lor{d/(d+1-2\alpha)}{\infty}\cap \thSob{\alpha}$ and $\boldv,\boldw\in \tSob{\beta}$, we have
\begin{equation}\label{eq:trilinear-boundedness}
 \int_{\mathbb{R}^d} | \boldu | \, |  \boldv | \, | \nabla \boldw  | \myd{x}  \apprle (\norm{\boldu}{\NLor{d/(d+1-2\alpha)}{\infty}} +\norm{\boldu}{\thSob{\alpha}})\norm{\boldv}{\tSob{\beta}}\norm{\boldw}{\tSob{\beta}}.
\end{equation}
Moreover, if $\Div \boldu=0$, then
\begin{equation}\label{eq:trilinear}
\int_{\mathbb{R}^d} \boldu \otimes \boldv : \nabla \boldw \myd{x}=-\int_{\mathbb{R}^d} (\boldu\cdot \nabla)\boldv \cdot \boldw \myd{x}.
\end{equation}
\item If $\boldu\in \Leb{\infty}(0,T;\Lor{d/(d+1-2\alpha)}{\infty})\cap \Leb{2}(0,T;\thSob{\alpha})$ and $\boldb \in \Leb{\infty}(0,T;\Leb{2})\cap \Leb{2}(0,T;\thSob{\beta})$, then
\[\boldb \in L_{2\beta} (0,T; \dot{H}^1 ) \quad\mbox{and}\quad 	\boldu\otimes \boldb\in \Leb{2\beta /(2\beta-1)}(0,T;\Leb{2}).\]
 Moreover, if $\Div \boldu=0$, then
\begin{equation}\label{eq:div-free-u}
\int_{\mathbb{R}^d} \boldu (t) \otimes \boldb (t)  : \nabla \boldb (t)  \myd{x}=0\quad \text{for almost all } t\in (0,T)
\end{equation}
\end{enumerate}
\end{lem}
\begin{proof}
(1) Let  $\gamma =2\beta /(2\beta-1)$. Then by Lemma \ref{lem:product-estimate},
 there exist $0<\theta_1,\theta_2<1$ with $\theta_1+\theta_2\leq 2-1/\beta=2/\gamma$ such that
  \begin{equation}\label{eq:prod-3}
   \norm{| \boldu| \, | \boldv| }{\NLeb{2}}\apprle \norm{\boldu}{\NLor{d/(d+1-2\alpha)}{\infty}}^{1-\theta_1}\norm{\boldu}{\thSob{\alpha}}^{\theta_1}\norm{\boldv}{\NLeb{2}}^{1-\theta_2}\norm{\boldv}{\thSob{\beta}}^{\theta_2}.
   \end{equation}
By Young's inequality, we have
\begin{align*}
\int_{\mathbb{R}^d} | \boldu | \, |  \boldv | \, | \nabla \boldw  | \myd{x} &\leq \norm{| \boldu| \, | \boldv|}{\NLeb{2}}\norm{\nabla \boldw}{\NLeb{2}}\\
&\apprle (\norm{\boldu}{\NLor{d/(d+1-2\alpha)}{\infty}} +\norm{\boldu}{\thSob{\alpha}})\norm{\boldv}{\tSob{\beta}}\norm{\boldw}{\tSob{\beta}},
\end{align*}
which proves \eqref{eq:trilinear-boundedness}.

To show \eqref{eq:trilinear}, choose sequences $\{\boldv_k\}$ and $\{\boldw_k\}$ in $C_c^\infty$ satisfying $\boldv_k\rightarrow \boldv$ and $\boldw_k\rightarrow \boldw$ in $\tSob{\beta}$. Since $\Div \boldu=0$, it follows that
\[
0=\int_{\mathbb{R}^d} \boldu\cdot \nabla (\boldv_k \cdot \boldw_k)\myd{x}=\int_{\mathbb{R}^d} (\boldu\cdot \nabla)\boldv_k \cdot \boldw_k \myd{x}+\int_{\mathbb{R}^d} \boldu\otimes\boldv_k : \nabla \boldw_k \myd{x}
\]
for all $k$. Hence the identity \eqref{eq:trilinear} is derived by letting $k\rightarrow \infty$ and using \eqref{eq:trilinear-boundedness}.

(2) It immediately follows from (\ref{eq:our-interpolation}) that $\boldb \in L_{2\beta} (0,T; \dot{H}^1 )$. Let $1<r\leq \infty$ be defined by
  \[	\frac{1}{r} = 1-\frac{\gamma}{2}\left(\theta_1+\theta_2\right).\]
Then by \eqref{eq:prod-3} and H\"older's inequality,
\begin{align*}
\int_0^T\norm{\boldu\otimes \boldb}{\NLeb{2}}^\gamma \myd{t}&\apprle \int_0^T \norm{\boldu}{\NLor{d/(d+1-2\alpha)}{\infty}}^{\gamma(1-\theta_1)}\norm{\boldu}{\thSob{\alpha}}^{\gamma\theta_1}\norm{\boldb}{\NLeb{2}}^{\gamma(1-\theta_2)}\norm{\boldb}{\thSob{\beta}}^{\gamma\theta_2}\myd{t}\\
&\apprle \int_0^T \norm{\boldu}{\thSob{\alpha}}^{\gamma\theta_1}\norm{\boldb}{\thSob{\beta}}^{\gamma\theta_2}\myd{t}\\
&\apprle T^{1/r}\left(\int_0^T \norm{\boldu}{\thSob{\alpha}}^2 \myd{t}\right)^{\gamma \theta_1 /2} \left(\int_0^T \norm{\boldb}{\thSob{\beta}}^2 \myd{t} \right)^{\gamma \theta_2 /2}.
\end{align*}
Finally, the identity \eqref{eq:div-free-u} is obtained by taking $\boldv = \boldw=\boldb$ in \eqref{eq:trilinear}.  This completes the proof of Lemma \ref{lem:div-free}.
\end{proof}

\begin{thm}\label{thm:energy-equality}
Let $\alpha$ and $\beta$ satisfy
\[	\frac{1}{2}<\alpha<\frac{d+1}{2},\quad \beta \geq 1,\quad \text{and}\quad \alpha+\beta \geq \frac{d}{2}+1.\]
Suppose that
$\boldu \in \Leb{\infty}(0,T;\Lor{d/(d+1-2\alpha)}{\infty})\cap \Leb{2}(0,T;\thSob{\alpha})$ and $\Div \boldu=0$ on $\mathbb{R}^d \times (0,T)$. Then for every $\boldF\in \Leb{2\beta/(2\beta-1)}(0,T;\Leb{2})$ and $\boldb_0 \in \Leb{2}$,  there exists a unique weak solution $\boldb$ of  \eqref{eq:PFHE}. Moreover, the solution $\boldb$ belongs to $C([0,T];\Leb{2})$ and satisfies the energy identity:
\begin{equation}\label{eq:energy-identity-thm}
\norm{\boldb(t)}{\NLeb{2}}^2+2\int_0^t\norm{\Lambda^\beta \boldb(s)}{\NLeb{2}}^2 \, ds=-2\int_0^t \int_{\mathbb{R}^d} \boldF : \nabla \boldb   \myd{x}ds + \norm{\boldb_0}{\NLeb{2}}^2
\end{equation}
for   $0\leq t\leq T$.
\end{thm}

\begin{proof}  We first show that  every weak solution $\boldb$ of \eqref{eq:PFHE} belongs to $  C([0,T];\Leb{2})$ and satisfies the energy identity \eqref{eq:energy-identity-thm}. Let $\boldb$ be a weak solution of \eqref{eq:PFHE} with $\boldF\in \Leb{2\beta/(2\beta-1)}(0,T;\Leb{2})$. Then $\boldb$ satisfies
\[
\left\{
\begin{alignedat}{2}
\partial_t \boldb +\Lambda^{2\beta} \boldb &= \Div \tilde{\boldF}&&\quad \text{in } \mathbb{R}^d\times (0,T),\\
\boldb( \cdot , 0) &=\boldb_0 &&\quad \text{on } \mathbb{R}^d  ,
\end{alignedat}
\right.
\]
where  $\tilde{\boldF}=\boldF-\boldu\otimes \boldb$. By Lemma \ref{lem:div-free}, we have
\[	\tilde{\boldF}\in \Leb{2\beta/(2\beta-1)}(0,T;\Leb{2})\quad \text{and}\quad \int_{\mathbb{R}^d} \tilde{\boldF}: \nabla \boldb\myd{x}=\int_{\mathbb{R}^d} \boldF :\nabla \boldb \myd{x}. \]
Hence it follows from Lemma \ref{lem:linear-heat} that $\boldb$ belongs to $ C([0,T];\Leb{2})$ and satisfies the energy identity \eqref{eq:energy-identity-thm}. 

Uniqueness of weak solutions of \eqref{eq:PFHE} is an  immediate consequence of  the energy identity. Indeed, if   $\boldb$ is a weak solution of \eqref{eq:PFHE} with $(\boldF , \boldb_0 ) =(\mathbf{0}, \mathbf{0})$, then since it satisfies the energy identity, we have
\[	\norm{\boldb(t)}{\NLeb{2}}^2+2\int_0^t \norm{\Lambda^\beta \boldb(s)}{\NLeb{2}}^2 \, \myd{s}=0 \]
for all $0\leq t\leq T$, which implies that $\boldb$ is identically zero.

To complete the proof, it remains to   prove existence of a weak solution of \eqref{eq:PFHE}.  To prove this, we may consider truncated problems as in the proof of Lemma \ref{lem:linear-heat}. The resulting truncated problems are still linear but non-autonomous with coefficients being singular in time. To avoid this difficulty, we   apply the Galerkin approximation method.  Choose a sequence $\{\Phi_k\}$ in $\mathscr{S}$ such that the span of $\{\Phi_k\}$ is dense in $\tSob{\beta}$ and $\{\Phi_k\}$ is orthonormal in $\Leb{2}$.   Then for a fixed $m\in \mathbb{N}$, we seek a function $\boldb_m$ of the form
\[	\boldb_m(t)=\sum_{k=1}^m c_k^m(t) \Phi_k \]
that satisfies
\[	\boldb_m(0)=\sum_{k=1}^m \left(\boldb_0 , \Phi_j \right)  \]
and
\begin{equation}\label{eq:ODE-Galerkin}
\left(\partial_t\boldb_m (t) ,\Phi_j \right)+ \left(\Lambda^\beta \boldb_m (t) ,\Lambda^\beta \Phi_j \right)- \left( \boldu(t) \otimes \boldb_m (t),\nabla \Phi_j \right)=- \left( \boldF (t),\nabla \Phi_j \right)
\end{equation}
for $1\leq j\leq m$. Equivalently, since $\{\Phi_k\}$ is orthonormal in $\Leb{2}$,  we  look for a vector function   $\boldc^m =(c_1^m,\dots,c_m^m): [0,T] \to \mathbb{R}^m$  satisfying
\begin{equation}\label{eq:system}
	\frac{d}{dt} \bm{c}^m(t)+ A(t) \bm{c}^m (t) = \bm{f}(t)\quad\mbox{and}\quad \bm{c}^m (0)=\bm{c}_0  ,
\end{equation}
where $A (t) = [ A_{jk}(t)]_{1\le j,k\le m}$, $\bm{f}(t) = ( f_1 (t) , \ldots , f_m (t))$, and $\bm{c}_0  = ( c_{0,1} , \ldots , c_{0,m} )$ are defined by
\[	A_{jk}(t)= \left( \Lambda^\beta \Phi_k ,  \Lambda^\beta \Phi_j \right) - \left(  \boldu (t) \otimes\Phi_k , \nabla \Phi_j \right), \]
\[	f_j(t)=- \left(  \boldF (t) , \nabla \Phi_j \right)  , \quad\mbox{and}\quad c_{0,j}= \left(\boldb_0 , \Phi_j \right) \]
for $1 \le j,k \le m$.
By Lemma \ref{lem:div-free}, we have
\[	|A(t)|\apprle \left(1+\norm{\boldu(t)}{\NLor{d/(d+1-2\alpha)}{\infty}}+\norm{\boldu(t)}{\thSob{\alpha}}\right) \norm{  \Phi_k}{H^\beta}\norm{ \Phi_j}{H^\beta} , \]
which implies that  $A \in \Leb{2}(0,T)$. Moreover, since $\bm{f} \in \Leb{2\beta/(2\beta-1)}(0,T)$, it follows from  a standard theory of ordinary differential equations (ODEs) that there exists a unique absolutely continuous function    $\boldc^m : [0,T] \to \mathbb{R}^m$ satisfying   \eqref{eq:system}.

Multiplying \eqref{eq:ODE-Galerkin} by $c_j^m$ and taking summation over $1\leq j\leq m$, we get
\[
\frac{1}{2} \frac{d}{dt}\norm{\boldb_m}{\NLeb{2}}^2+\norm{\Lambda^\beta \boldb_m}{\NLeb{2}}^2-\left(\boldu\otimes \boldb_m,\nabla \boldb_m \right)=- \left( \boldF,\nabla \boldb_m \right).
\]
By Lemma \ref{lem:div-free}, \eqref{eq:our-interpolation}, and Young's inequality,
$$
\left( \boldu \otimes \boldb_m,\nabla \boldb_m \right)=0
$$
and
\begin{align*}
-\left(\boldF,\nabla \boldb_m \right)
&\leq C \norm{\boldF}{\NLeb{2}}\norm{\boldb_m}{\NLeb{2}}^{1-1/\beta} \norm{\Lambda^\beta \boldb_m}{\NLeb{2}}^{1/\beta}\\
&\leq C \norm{\boldF}{\NLeb{2}}^{\gamma} \norm{\boldb_m}{\NLeb{2}}^{2(\beta-1)/(2\beta-1)}+\frac{1}{2}\norm{\Lambda^\beta \boldb_m}{\NLeb{2}}^2.
\end{align*}
Since $0 \le (\beta-1)/(2\beta -1)< 1$, we have
\begin{align*}
 \frac{d}{dt} \norm{\boldb_m}{\NLeb{2}}^2+ \norm{\Lambda^\beta \boldb_m}{\NLeb{2}}^2&\leq 2C \norm{\boldF}{\NLeb{2}}^{\gamma} \left(1+ \norm{\boldb_m}{\NLeb{2}}^{2}\right).
\end{align*}
Hence, by Gronwall's inequality,
\begin{align*}
\sup_{0<t<T}\norm{\boldb_m(t)}{\NLeb{2}}^2+\int_0^T\norm{\Lambda^\beta \boldb_m(t)}{\NLeb{2}}^2 \, \myd{t}\leq \norm{\boldb_0}{\NLeb{2}}^2 \,  \exp \left(C\int_0^T \norm{\boldF (t) }{\NLeb{2}}^\gamma \, \myd{t}\right)
\end{align*}
for some constant $C$ depending  only on $d$, $\beta$, and $T$.  This shows that  the sequence $\{\boldb_m\}$ is uniformly bounded in $\Leb{\infty}(0,T;\Leb{2})\cap \Leb{2}(0,T;\tSob{\beta})$. Therefore, by weak-* compactness, there exist a   subsequence $\{\boldb_{m_k}\}$ and a function $\boldb$ such that
\[	\boldb_{m_k}\rightarrow \boldb \quad \text{weakly-* in } \Leb{\infty}(0,T;\Leb{2})\cap \Leb{2}(0,T;\tSob{\beta}). \]
It is quite standard to show that $\boldb$ is a weak solution of \eqref{eq:PFHE}.
The proof of Theorem \ref{thm:energy-equality} is complete. 
\end{proof}

\section{Proof of Theorem \ref{thm:GWP}}\label{sec:proofs}
This section is fully devoted to proving Theorem \ref{thm:GWP}.
Throughout the remaining sections, we assume that $\nu=\eta=1$ for simplicity. 

\subsection{Approximation of the problem by truncations}

Let $\Proj$ be the Leray projection defined by
\[ \widehat{\Proj (\boldu)}(\xi)  =  \left(I - \frac{\xi \otimes  \xi}{|\xi|^2} \right) \widehat{\boldu} (\xi)  .  \]
It is well-known that the Leray projection $\Proj$ has the following  properties:
\begin{enumerate}[label=\textnormal{(\arabic*)}]
\item $\Div   \Proj(\boldu) =0$ for all $\boldu \in \mathscr{S}$;
\item $\Proj(\nabla \phi)=0$ for all $\phi \in \mathscr{S}$;
\item $\Proj(\boldu)=\boldu$ for all $\boldu \in \mathscr{S}$ with $\Div \boldu=0$;
\item    $\Proj[\Proj(\boldu)]=\Proj(\boldu)$ for all $\boldu \in \mathscr{S}$;
\item  $\Proj$ is   bounded   from  $\Leb{q}$ to itself for $1<q<\infty$.
\end{enumerate}
Note also that $ \Proj \mathcal{S}_R (\boldu ) \in  V_{R,\sigma}$ for all  $\boldu \in L_2$, where $V_{R,\sigma}$ is the space of all vector fields $\boldv \in V_R$ such that $\Div \boldv=0$.

\medskip
The following lemma is necessary to construct   approximate solutions of \eqref{eq:MHD-relaxed}.
\begin{lem}\label{lem:b_R and u_R} Let $1/2 < \alpha < (d+1)/2$. Given    $\boldb  \in \Leb{2}$, we define
\[
\boldu  = \mathbf{U}_\alpha * (\boldb \otimes \boldb ),
\]
where  $\mathbf{U}_{\alpha} = [U_{\alpha, j}^{kl}]_{1 \le j,k,l \le d}$ is   defined by \eqref{eq:Stokes-Green}. Then
\begin{equation}\label{eq:sol-estimate-bob-R}
\boldu   \in  \Lor{d/(d+1-2\alpha)}{\infty}
\quad\mbox{and}\quad
\norm{\boldu }{\NLor{d/(d+1-2\alpha)}{\infty}}\apprle  \norm{\boldb  }{\NLeb{2}}^2   .
\end{equation}
Moreover, if $\boldb \in V_R$, then for any $\gamma \in \mathbb{N}_0^d$ and $d/(d+1-2\alpha) < q< \infty$,
\begin{equation}\label{eq:sol-estimate-R-s}
D^\gamma \boldu \in L_q , \quad \norm{D^\gamma \boldu }{q}\apprle_R  \norm{\boldb  }{\NLeb{2}}^2,
\end{equation}
and
\begin{equation}\label{eq:sol-estimate-u0b-R-s}
\norm{D^\gamma (\boldu \otimes \boldb ) }{\NLeb{2}}+ \norm{D^\gamma (\boldb \otimes \boldu )}{\NLeb{2}}\apprle_R  \norm{\boldb  }{\NLeb{2}}^3.
\end{equation}
\end{lem}
\begin{proof} Set  $\boldF  = \boldb \otimes \boldb $. Then since $\boldF  \in L_1$ and $\mathbf{U}_{\alpha}\in  \Lor{d/(d+1-2\alpha)}{\infty}$, the estimate (\ref{eq:sol-estimate-bob-R}) immediately follows from Lemma   \ref{lem:Young}. Suppose in addition that $\boldb \in V_R$. Since $V_R \hookrightarrow \tSob{s}$ and $\tSob{d+s}\hookrightarrow   L_\infty$  for any $s > 0$, it follows from H\"{o}lder's inequality that
\[
 \| D^\gamma \boldF  \|_{L_1 \cap L_\infty} \apprle_R  \norm{\boldb  }{\NLeb{2}}^2  \quad\mbox{for all}\,\, \gamma \in {\mathbb N}_0^d .
\]
Let $\gamma \in {\mathbb N}_0^d$ be fixed. Then since $D^\gamma \boldu  = \mathbf{U}_\alpha *  D^\gamma \boldF$,
it follows from Lemma   \ref{lem:Young} that
\[
 \| D^\gamma \boldu \|_{q} \apprle   \| D^\gamma \boldF  \|_{L_1 \cap L_\infty} \apprle_R \norm{\boldb  }{\NLeb{2}}^2
\]
for any $d/(d+1-2\alpha) < q< \infty$. Finally, since there exist $p$ and $q$ such that
\[
2< p< \infty , \quad \frac{d}{d+1-2\alpha} < q < \infty , \quad\mbox{and}\quad \frac{1}{2} = \frac{1}{p}+\frac{1}{q},
\]
the estimate  (\ref{eq:sol-estimate-u0b-R-s}) is easily deduced from (\ref{eq:sol-estimate-R-s}) by using H\"{o}lder's inequality.
This completes the proof of Lemma \ref{lem:b_R and u_R}.
\end{proof}

By virtue of Lemma \ref{lem:b_R and u_R}, we may consider the following
  truncated  problem  for $\boldb_R \in C^1([0,\infty);V_{R,\sigma})$:
\begin{equation}\label{eq:MHD-relaxed-approx}
\left\{
\begin{alignedat}{2}
\partial_t \boldb_R + \Lambda^{2\beta} \boldb_R &=\Proj\mathcal{S}_R[\Div\dsp (\boldb_R\otimes \boldu_R-\boldu_R\otimes \boldb_R)]  &&\quad \text{in }\mathbb{R}^d\times(0,\infty),\\
\boldu_R & = \mathbf{U}_\alpha * (\boldb_R\otimes \boldb_R)&&\quad \text{in }\mathbb{R}^d\times(0,\infty),\\
\boldb_R(\cdot,0)&=\Proj\mathcal{S}_R\boldb_0  &&\quad \text{on } \mathbb{R}^d ,
\end{alignedat}
\right.
\end{equation}
where $\boldb_0 \in \Leb{2}$ with ${\rm div}  \,\boldb_0=0$ is given. Note here  that if $\boldb_R \in C^1([0,T);V_{R,\sigma})$, $0<T \le \infty$,  and  $\boldu_R = \mathbf{U}_\alpha * (\boldb_R\otimes \boldb_R)$, then  by  Lemma \ref{lem:b_R and u_R},
$$
\boldu_R  \in C([0,T);\Lor{d/(d+1-2\alpha)}{\infty}  )
$$
and
\begin{equation}\label{eq:sol-estimate-R}
\norm{\boldu_R (t)}{\NLor{d/(d+1-2\alpha)}{\infty}}\apprle  \norm{\boldb_R(t)}{\NLeb{2}}^2 \quad\mbox{for}\,\, 0 \le t < T .
\end{equation}
Moreover, it follows from Theorem \ref{theorem:Frational Stokes-Lpdata-intersection} that for each $t \in (0,T)$, $\boldu_R (t) $ is a very weak solution  of the fractional Stokes system
\begin{equation}\label{eq:U-equation}
   \Lambda^{2\alpha} \boldu_R (t) +\nabla {p_*} (t) = \Div\dsp (\boldb_R (t)\otimes \boldb_R (t) ),\quad \Div \boldu_R (t) =0\quad\mbox{in}\,\,\mathbb{R}^d ;
\end{equation}
that is, $\boldu_R (t) $ satisfies
\begin{equation}\label{eq:weak-sol-R}
 \int_{\mathbb{R}^d}  \boldu_R(t) \cdot \left(\Lambda^{2\alpha} \Phi + \nabla \psi\right) dx   = - \int_{\mathbb{R}^d}\boldb_R (t)\otimes \boldb_R (t) : \nabla \Phi \myd{x}
\end{equation}
for all $( \Phi , \psi ) \in \mathscr{S}$ with $\Div \Phi=0$. We shall show that for each $R>0$, there exists a unique solution $\boldb_R \in C^1([0,\infty);V_{R,\sigma})$ of the $R$-truncated  problem (\ref{eq:MHD-relaxed-approx}) and then that some sequence $\{(\boldu_{R_k} , \boldb_{R_k} ) \}$ converges   to a weak solution of the original problem \eqref{eq:MHD-relaxed}.
 
\subsection{Existence of approximate solutions}

Let $0<R<\infty$ be fixed. For  $\boldb  \in V_{R,\sigma}$, we define
\[
G_R (\boldb  )  = -  \Lambda^{2\beta} \boldb  + \Proj\mathcal{S}_R[\Div\dsp(\boldb \otimes \boldu -\boldu \otimes \boldb)],
\]
where $\boldu  = \mathbf{U}_\alpha * (\boldb \otimes \boldb )$. By Lemma \ref{lem:b_R and u_R}, $G_R$ is a well-defined mapping from $V_{R,\sigma}$ into $V_{R,\sigma}$. Therefore, the $R$-truncated  problem (\ref{eq:MHD-relaxed-approx}) can be rewritten as the following Cauchy problem on an ODE in the Hilbert space  $V_{R,\sigma}$:
\begin{equation}\label{eq:MHD-relaxed-approx-ODE}
\left\{
\begin{alignedat}{2}
\partial_t \boldb_R (t)   &=G_R (\boldb_R (t) )    \quad (0< t< \infty),\\
\boldb_R( 0)&=\Proj\mathcal{S}_R\boldb_0 ,
\end{alignedat}
\right.
\end{equation}
where $\boldb_0 \in \Leb{2}$  with ${\rm div}  \,\boldb_0=0$ is given.

Suppose that   $\boldb^i \in V_{R,\sigma}$ and $\boldu^i = \mathbf{U}_\alpha * (\boldb^i \otimes \boldb^i ) $ for $i=1,2$. Then since  $V_{R,\sigma}\hookrightarrow \tSob{s}\hookrightarrow \thSob{s}$  for any $s \ge 0$, we have
\[
\| \Lambda^{2\beta} \boldb^1 -  \Lambda^{2\beta} \boldb^2 \|_2  \apprle_R \| \boldb^1 -   \boldb^2 \|_2 .
\]
Following the proof of Lemma \ref{lem:b_R and u_R}, we also obtain
\begin{align*}
&\norm{ \Div \dsp ( \boldb^1\otimes \boldu^1  -\boldb^2\otimes \boldu^2 )}{\NLeb{2}}\\
&\quad \le \norm{ \Div\dsp  ( (\boldb^1 -\boldb^2 )\otimes \boldu^1   )}{\NLeb{2}} + \norm{ \Div\dsp  ( \boldb^2\otimes (\boldu^1  - \boldu^2 ) )}{\NLeb{2}}\\
&\quad \apprle_R \|\boldb^1 -\boldb^2 \|_2 \|\boldb^1   \|_2^2 + \|\boldb^2   \|_2 \left(\|\boldb^1 \|_2 +\|\boldb^2 \|_2  \right)\|\boldb^1 -\boldb^2 \|_2
\end{align*}
and
\begin{align*}
&\norm{ \Div\dsp( \boldu^1\otimes \boldb^1  -\boldu^2\otimes \boldb^2 )}{\NLeb{2}}\\
&\quad \le \norm{ \Div\dsp  ( (\boldu^1 -\boldu^2 )\otimes \boldb^1   )}{\NLeb{2}} + \norm{ \Div \dsp ( \boldu^2\otimes (\boldb^1  - \boldb^2 ) )}{\NLeb{2}}\\
&\quad \apprle_R \left(\|\boldb^1 \|_2 +\|\boldb^2 \|_2  \right) \|\boldb^1 -\boldb^2 \|_2   \|\boldb^1   \|_2 + \|\boldb^2   \|_2^2 \|\boldb^1 -\boldb^2 \|_2 .
\end{align*}
Hence by the $L_2$-boundedness of $\Proj \mathcal{S}_R$, we conclude that
\begin{align*}
\|G_R (\boldb^1 ) -G_R (\boldb^2 ) \|_{2} & \apprle_R  \left(1+  \|\boldb^1 \|_2 +\|\boldb^2 \|_2 \right)^2 \| \boldb^1 -   \boldb^2 \|_2 ,
\end{align*}
which implies that $G_R$ is locally Lipschitz continuous on $V_{R,\sigma}$. Therefore   it follows from Picard's theorem for ODEs in infinite dimensional spaces that for every $\boldb_0 \in \Leb{2}$ with ${\rm div}  \,\boldb_0=0$, the $R$-truncated problem \eqref{eq:MHD-relaxed-approx} has a unique local classical solution $\boldb_R \in  C^1([0,T_R);V_{R,\sigma})$. Here $0<T_R \le \infty$ denotes the maximal existence time of the local solution $\boldb_R$. Indeed,   $\boldb_R$ exists globally, that is, $T_R =\infty$. To show this, let  $\boldu_R = \mathbf{U}_\alpha * (\boldb_R\otimes \boldb_R)$. Since  $V_{R,\sigma}\hookrightarrow \tSob{s}\hookrightarrow \thSob{s}$ and  $(\thSob{s})^*=\thSob{-s}$ for any $s \ge 0$, it  follows  that
$\boldb_R \otimes \boldb_R \in C([0,T_R);\thSob{1-\alpha} ).$  Hence by Theorem   \ref{theorem:Frational Stokes-Lpdata}, $\boldu_R$ belongs to $C([0,T_R);\thSob{\alpha} )$ and satisfies the energy identity
\[ 
   \int_{\mathbb{R}^d} |\Lambda^{\alpha} \boldu_R (t)|^2 \myd{x}= - \int_{\mathbb{R}^d} \boldb_R (t)\otimes \boldb_R (t) : \nabla \boldu_R (t)  \, dx
\] 
for  $0< t<  T_R$. Multiply the $\boldb_R$-equation in \eqref{eq:MHD-relaxed-approx} by $\boldb_R$ and using the divergence-free condition, we get
\[
\frac{d}{dt}\left(\frac{1}{2}\norm{\boldb_R (t)}{\NLeb{2}}^2 \right)+ \norm{\Lambda^\beta \boldb_R (t)}{\NLeb{2}}^2 =-\int_{\mathbb{R}^d}  \boldb_R (t)\otimes  \boldu_R (t) : \nabla \boldb_R (t) \, dx
\]
for  $0< t<  T_R$. Adding these two identities and using the divergence-free condition, we have
\[   \frac{d}{dt}\left(\frac{1}{2}\norm{\boldb_R (t) }{\NLeb{2}}^2 \right)+ \norm{\Lambda^\alpha \boldu_R (t)}{\NLeb{2}}^2+ \norm{\Lambda^\beta \boldb_R (t) }{\NLeb{2}}^2=0 ,\]
from which  we derive the global  energy estimate
\begin{equation}\label{eq:energy-estimate}
  \norm{\boldb_R(t)}{\NLeb{2}}^2+2 \int_0^t \norm{\Lambda^\alpha \boldu_R(\tau)}{\NLeb{2}}^2  \, d\tau +2  \int_0^t \norm{\Lambda^\beta \boldb_R(\tau)}{\NLeb{2}}^2 \, d\tau \leq \norm{\boldb_0}{\NLeb{2}}^2
\end{equation}
for  $0\leq t<T_R$.
Therefore,  by a standard ODE theory, we conclude that   $T_R= \infty$.

\medskip 
The following lemma will be used to  construct weak solutions of \eqref{eq:MHD-relaxed} by applying the Aubin-Lions compactness lemma.
\begin{lem}\label{lem:Aubin}
There exist $r>1$ and $s\geq 1$ such that  $\{\partial_t \boldb_R \}$ is bounded in $L_r(0,T;H^{-s})$ for any $0<T<\infty$.
\end{lem}
\begin{proof} To show this, let $s =\max \{\beta,1\}$.  Then since $-s+1 \le 0$ and $2\beta-s \le \beta$, it follows that
\begin{align*}
\action{\partial_t \boldb_R,\Phi}&=\action{-\Lambda^{2\beta}\boldb_R,\Phi} - \int_{\mathbb{R}^d} (\boldb_R\otimes \boldu_R - \boldu_R\otimes \boldb_R ):\nabla ( \Proj\mathcal{S}_R\Phi )\myd{x}\\
&\apprle \norm{\Lambda^{2\beta} \boldb_R}{\tSob{-s}}\norm{\Phi}{\tSob{s}}+\norm{|\boldu_R| | \boldb_R| }{\tSob{-s+1}}\norm{\nabla ( \Proj\mathcal{S}_R\Phi )}{\tSob{s-1}}\\
&\apprle \left( \norm{\boldb_R}{\tSob{2\beta-s}}+\norm{ |\boldu_R| | \boldb_R| }{\tSob{-s+1}} \right)\norm{\Phi}{\tSob{s}}\\
&\apprle \left( \norm{\boldb_R}{\tSob{\beta}}+\norm{|\boldu_R| |  \boldb_R| }{\NLeb{2}} \right)\norm{\Phi}{\tSob{s}}
\end{align*}
for all $\Phi \in \mathscr{S}$. Therefore,
\begin{equation}\label{eq:time-estimate}
    \norm{\partial_t \boldb_R(t)}{\tSob{-s}}\apprle   \norm{\boldb_R}{\tSob{\beta}}+\norm{|\boldu_R | | \boldb_R |}{\NLeb{2}}.
\end{equation}
By Lemma \ref{lem:product-estimate}, there exist  $0< \theta_1 , \theta_2 < 1$ such that
\[
\norm{|\boldu_R | | \boldb_R |  }{2}\apprle \norm{\boldu_R }{\NLor{d/(d+1-2\alpha)}{\infty}}^{1-\theta_1} \norm{\boldu_R }{\thSob{\alpha}}^{\theta_1}\norm{\boldb_R  }{\NLeb{2}}^{1-\theta_2}
\norm{\boldb_R  }{\thSob{\beta}}^{\theta_2} .
\]
 Setting   $r_0 = 2/(\theta_ 1 +\theta_2 ) >1$ and using Young's inequality, we have
\[
\norm{|\boldu_R | | \boldb_R |  }{2}  \apprle \left(\norm{\boldu_R }{\NLor{d/(d+1-2\alpha)}{\infty}}^{1-\theta_1} \norm{\boldu_R }{\thSob{\alpha}}^{\theta_1} \right)^{2/r_0 \theta_1}+ \left(\norm{\boldb_R  }{\NLeb{2}}^{1-\theta_2}
\norm{\boldb_R  }{\thSob{\beta}}^{\theta_2} \right)^{2/r_0 \theta_2}
\]
and so
\[
\norm{|\boldu_R | | \boldb_R | }{2}^{r_0}  \apprle  \norm{\boldu_R }{\NLor{d/(d+1-2\alpha)}{\infty}}^{2(1-\theta_1)/ \theta_1}\norm{\Lambda^\alpha \boldu_R}{\NLeb{2}}^{2} +  \norm{\boldb_R  }{\NLeb{2}}^{2(1-\theta_2)/\theta_2}
\norm{\Lambda^\beta\boldb_R}{\NLeb{2}}^{2}.
\]
By   \eqref{eq:sol-estimate-R} and   (\ref{eq:energy-estimate}),    we thus obtain
\[
\int_0^\infty \norm{|\boldu_R | | \boldb_R |  }{2}^{r_0} \, dt   \apprle  \norm{\boldb_0}{\NLeb{2}}^{4/\theta_1  -2}  +  \norm{\boldb_0}{\NLeb{2}}^{2/\theta_2 }  .
\]
Now, let $r= \min \{2,r_0 \}$. Then by     \eqref{eq:energy-estimate} and \eqref{eq:time-estimate}, we have
\begin{align*}
   \norm{\partial_t \boldb_R}{\Leb{r}(0,T;\tSob{-s})}&\apprle T^{1/r-1/2}\norm{\boldb_R}{\Leb{2}(0,T;\tSob{\beta})}
    +T^{1/r-1/r_0}\norm{\boldu_R   \otimes \boldb_R}{\Leb{r_0}(0,T;\Leb{2})} \\
   &\leq C(d,\alpha,\beta,T, \norm{\boldb_0}{\NLeb{2}}) ,
\end{align*}
which  proves that $\{\partial_t \boldb_R\}$ is   bounded in $\Leb{r}(0,T;\tSob{-s})$ for any $0<T<\infty$.
\end{proof}

\subsection{Completion of the proof of Theorem \ref{thm:GWP}}

We have shown that for each $0<R<\infty$, the $R$-truncated problem \eqref{eq:MHD-relaxed-approx} has a unique global    solution $\boldb_R \in  C^1([0, \infty );V_{R,\sigma})$. By  (\ref{eq:sol-estimate-R}) and (\ref{eq:energy-estimate}),  the families   $\{\boldb_R\}_{R>0}$ and  $\{\boldu_R\}_{R>0}$  are bounded in $\Leb{\infty}(0,\infty;\Leb{2})\cap \Leb{2}(0,\infty; \dot{H}^\beta)$ and $\Leb{\infty}(0,\infty;\Lor{d/(d+1-2\alpha)}{\infty})\cap \Leb{2}(0,\infty;\dot{H}^\alpha)$, respectively.

By the weak-* compactness result and (\ref{eq:energy-estimate}), there exists a sequence $\{R_k\}$ with $R_k\rightarrow \infty$ such that
\[ 
\boldu_{R_k}\rightarrow \boldu \quad\mbox{weakly-* in}\,\, \Leb{\infty}(0,\infty;\Lor{d/(d+1-2\alpha)}{\infty})\cap \Leb{2}(0,\infty;\thSob{\alpha})
\] 
and
\[ 
\boldb_{R_k}\rightarrow \boldb \quad\mbox{weakly-* in} \,\, \Leb{\infty}(0,\infty;\Leb{2})\cap \Leb{2}(0,\infty;\thSob{\beta})
\]
for some  $(\boldu , \boldb )$ satisfying
\begin{equation}\label{eq:energy-estimate-limit}
 \| \boldb \|_{\Leb{\infty}(0,\infty;\Leb{2})}^2+     \| \boldb\|_{\Leb{2}(0,\infty;\thSob{\beta})}^2+    \| \boldu\|_{\Leb{2}(0,\infty;\thSob{\alpha})}^2 \le 2 \norm{\boldb_0}{\NLeb{2}}^2 .
\end{equation}
By Lemma \ref{lem:Aubin} and the Aubin-Lions compactness lemma,   there exists a subsequence of $\{\boldb_{R_k}\}$,  denoted still  by $\{\boldb_{R_k}\}$, such that
\begin{equation}\label{strong conv-b}
\boldb_{R_k}\rightarrow \boldb \quad\mbox{in}\,\, \Leb{2}(K \times (0,T) )
\end{equation}
 for any compact $K \subset \mathbb{R}^d$ and finite  $T>0$.
 Moreover, by   \eqref{eq:energy-estimate}, we have
\begin{align*}
\norm{ \boldb_{R_k}\otimes \boldb_{R_k}-\boldb\otimes \boldb }{\NLeb{1};K}&\leq \norm{ \boldb_{R_k}\otimes (\boldb_{R_k}-\boldb) }{\NLeb{1};K}  +\norm{ (\boldb_{R_k} -\boldb )\otimes \boldb }{\NLeb{1};K}\\
&\leq 2 \norm{\boldb_0}{\NLeb{2}} \norm{\boldb_{R_k}  -\boldb }{\NLeb{2};K}
\end{align*}
and so
\begin{equation}\label{strong conv-bob}
\boldb_{R_k}\otimes \boldb_{R_k}\rightarrow \boldb\otimes \boldb \quad\mbox{in}\,\,  \Leb{2}(0,T;\Leb{1}(K))
\end{equation}
for  any compact $K \subset \mathbb{R}^d$ and finite  $T>0$.

Next, we show that
\begin{equation}\label{strong conv-u}
\boldu_{R_k}\rightarrow   \mathbf{U}_\alpha * (\boldb \otimes \boldb )
\quad\mbox{in}\,\, \Leb{2}(0,T;\Lor{d/(d+1-2\alpha)}{\infty}(K))
\end{equation}
 for  any finite  $T$ and  compact $K \subset \mathbb{R}^d$, which  automatically implies   that
\begin{equation}\label{u=Ubb}
 \boldu  = \mathbf{U}_\alpha * (\boldb \otimes \boldb );
\end{equation}
hence $\boldu$ satisfies (\ref{eq:velocity}).

Let $\boldv  = \mathbf{U}_\alpha * (\boldb \otimes \boldb )$. Fixing
 $0<r, T< \infty$, we write
\begin{align*}
  \boldu_{R_k} (x,t)- \boldv (x,t)  &=\int_{B_{M+r}} \mathbf{U}_\alpha (x-y): (\boldb_{R_k}\otimes \boldb_{R_k} -\boldb \otimes \boldb ) (y,t) \myd{y}\\
  &\relphantom{=}+\int_{\mathbb{R}^d \setminus B_{M+r}} \mathbf{U}_\alpha (x-y): (\boldb_{R_k}\otimes \boldb_{R_k} -\boldb \otimes \boldb ) (y,t)\myd{y}\\
  &=\bm{I}_1^k ( x,t)+\bm{I}_2^k (x,t )
\end{align*}
for $(x,t) \in B_r \times (0,T)$, where  $M>1$ is a large number to be determined later. Recall from \eqref{eq:Stokes-Green} that
$$
|\mathbf{U}_\alpha(x)|\apprle \frac{1}{|x|^{d+1-2\alpha}}.
$$
Hence by Lemma \ref{lem:Young} and (\ref{strong conv-bob}), we have
\begin{align*}
 \norm{\bm{I}_1^k}{L_2 (0,T;\Lor{d/(d+1-2\alpha)}{\infty})} &\apprle \norm{\mathbf{U}_\alpha}{\NLor{d/(d+1-2\alpha)}{\infty}}\norm{\boldb_{R_k}\otimes \boldb_{R_k}-\boldb \otimes \boldb }{L_2(0,T; \Leb{1}(B_{M+r}) )}\\
  & \rightarrow 0 \quad\mbox{as}\,\, k \rightarrow \infty .
\end{align*}
To estimate $|\bm{I}_2^k |$, we choose $0<\theta  \le  1/2$ and $2<p<{2d}/{(2\alpha-1)}$
such that
\[
\frac{1}{p}= \frac{1}{2}-\frac{\theta \beta}{d}.
\]
Then by  Proposition \ref{prop:Gagliardo},
\[
\|\boldb_{R_k}\|_{p} \apprle \|\boldb_{R_k}\|_{2}^{1-\theta}\norm{\Lambda^\beta\boldb_{R_k}}{\NLeb{2}}^\theta .
\]
Hence setting $q= 2/\theta $, we have
\[
 \norm{\boldb_{R_k}}{\Leb{q}(0,\infty;\Leb{p})} \apprle\norm{\boldb_{R_k}}{\Leb{\infty}(0,\infty;\Leb{2})}^{1-\theta}
 \norm{\boldb_{R_k}}{\Leb{2}(0,\infty;\thSob{\beta})}^{\theta} \apprle \|\boldb_{0}\|_{2}
\]
and similarly
\[
 \norm{\boldb}{\Leb{q}(0,\infty;\Leb{p})}   \apprle \|\boldb_{0}\|_{2}.
\]
On the other hand, since $d+1-2\alpha-{d(p-2)}/{p}>0$, it follows    that  if $|x|\leq r$, then
  \begin{align*}
  |\bm{I}_2^k (x,t)| &\leq \left(\int_{\mathbb{R}^d \setminus B_M} |\mathbf{U}_\alpha(z)|^{(p/2)'}dz \right)^{1/(p/2)'} \norm{ [\boldb_{R_k}\otimes \boldb_{R_k}-\boldb \otimes \boldb](t) }{\NLeb{p/2}}\\
  &\apprle M^{-[d+1-2\alpha-d(p-2)/p] } \left(\norm{\boldb_{R_k}(t)}{\NLeb{p}}^2 + \norm{\boldb (t)}{\NLeb{p}}^2 \right).
  \end{align*}
Therefore,  noting that $q \ge 4$, we deduce that
  \begin{align*}
      \norm{\boldu_{R_k}-\boldv}{\Leb{2}(0,T;\Lor{d/(d+1-2\alpha)}{\infty}(B_r))}
     &\le   \norm{\bm{I}_1^k}{\Leb{2}(0,T;\Lor{d/(d+1-2\alpha)}{\infty}(B_r))}\\
     & \quad + C(r,T)  M^{-[d+1-2\alpha-d(p-2)/ p]} \|\boldb_{0}\|_{2}^2
 \end{align*}
for some constant    $C(r,T)$ independent of $k$ and $M$.

Now, let  $\delta>0$ be given. Then  choosing  $M>1$ so large that
$$
C(r,T)  M^{-[d+1-2\alpha-d(p-2)/ p]} \|\boldb_{0}\|_{2}^2 <\delta ,
$$
and then letting  $k \to \infty$, we have
 \[  \limsup_{k\rightarrow \infty} \norm{\boldu_{R_k}-\boldv}{\Leb{2}(0,T;\Lor{d/(d+1-2\alpha)}{\infty}(B_r))}\leq \delta,
  \]
which  proves (\ref{strong conv-u}).

Finally, we  show that
\begin{equation}\label{strong cov-bu}
\begin{aligned}
&\boldu_{R_k}\otimes \boldb_{R_k} \rightarrow    \boldu\otimes \boldb
\quad\mbox{in}\,\, \Leb{1}(K \times (0,T)), \\
&\boldb_{R_k}\otimes \boldu_{R_k} \rightarrow    \boldb\otimes \boldu
\quad\mbox{in}\,\, \Leb{1}(K \times (0,T))
\end{aligned}
\end{equation}
 for  any compact $K \subset \mathbb{R}^d$ and finite  $T>0$. By symmetry, it suffices to show the first assertion. We write
\[
\boldu_{R_k}\otimes \boldb_{R_k}  -     \boldu  \otimes \boldb=  ( \boldu_{R_k}  -  \boldu)\otimes \boldb_{R_k} + \boldu \otimes( \boldb_{R_k}  -  \boldb).
\]
First, since $1/2<\alpha<(d+1)/2$, $\beta>0$,  and $2\alpha+\beta>d/2+1$, there exists $2<p_1 <\infty$ such that
\[
\max\left\{ \frac{1}{2}-\frac{\beta}{d}, 0 \right\} < \frac{1}{p_1}< \min \left\{ \frac{1}{2}, \frac{2\alpha-1}{d} \right\}.
\]
Then by  Proposition \ref{prop:upper-critical},
\[
\|\boldb_{R_k}\|_{\NLor{p_1}{1}} \apprle \|\boldb_{R_k}\|_{2}^{1-\theta_1}\norm{\Lambda^\beta \boldb_{R_k} }{\NLeb{2}}^{\theta_1} ,
\]
where $0<\theta_1 <1$ is defined by $1/p_1 = 1/2 -(\theta_1 \beta )/d$. Let   $K \subset \mathbb{R}^d$ be compact. Then since $1/p_1 > 1/2 -\beta/d$ and  $d/(2\alpha -1) < p_1$, it follows from   Lemma \ref{lem:Holder} and the Sobolev embedding theorem that
\begin{align*}
\|( \boldu_{R_k}  -  \boldu )\otimes \boldb_{R_k}\|_{1;K} & \apprle \| \boldu_{R_k}  -  \boldu   \|_{d/(d+1-2\alpha),\infty ;K}\|\boldb_{R_k}\|_{d/(2\alpha-1),1; K } \\
& \apprle_K \| \boldu_{R_k}  -  \boldu  \|_{d/(d+1-2\alpha),\infty ;K} \|\boldb_{R_k}\|_{H^{\beta}}.
\end{align*}
Hence  by  \eqref{eq:energy-estimate}, (\ref{strong conv-u}) and (\ref{u=Ubb}), we have
\begin{align*}
\|( \boldu_{R_k}  -  \boldu  )\otimes \boldb_{R_k}\|_{L_1 (0,T; L_1 (K)) } & \apprle_{K,T}  \norm{\boldu_{R_k}-\boldu}
{\Leb{2}(0,T;\Lor{d/(d+1-2\alpha)}{\infty}(K))}\|\boldb_{0}\|_{2} \\
& \rightarrow 0 \quad\mbox{as}\,\, k \to \infty
\end{align*}
for any finite $T$. To estimate $\boldu\otimes ( \boldb_{R_k}  - \boldb)$, we choose $2<p_2 < \infty$ such that
\[
\max\left\{ \frac{1}{2}-\frac{\alpha}{d}, 0 \right\} < \frac{1}{p_2}< \min \left\{ \frac{1}{2}, \frac{d+1-2\alpha}{d} \right\}.
\]
By  Proposition \ref{prop:upper-critical} again,
\[
\|\boldu \|_{2;K} \apprle_K  \|\boldu \|_{p_2} \apprle  \|      \boldu   \|_{d/(d+1-2\alpha),\infty}^{1-\theta_2} \|\Lambda^{\alpha} \boldu \|_{2}^{\theta_2}
\]
for some $0< \theta_2 < 1$. Hence by  (\ref{eq:energy-estimate-limit}), (\ref{strong conv-b}), and (\ref{u=Ubb}), we have
\begin{align*}
\|\boldu \otimes ( \boldb_{R_k} -     \boldb  )\|_{L_1 (0,T; L_1 (K)) } & \le  \|\boldu \|_{L_2 (0,T; L_2 (K)) }\|\boldb_{R_k}  -  \boldb  \|_{L_2 (0,T; L_2 (K)) }\\
& \apprle_{K,T} \left( \|\boldb_{0}\|_{2}  + \|\boldb_{0}\|_{2}^2 \right)  \norm{\boldb_{R_k}-\boldb}
{\Leb{2}(0,T;L_2 (K))} \\
& \rightarrow 0 \quad\mbox{as}\,\, k \to \infty ,
\end{align*}
which proves  (\ref{strong cov-bu}).

We have shown    strong convergence for  the nonlinear terms $\boldb_{R_k}\otimes \boldb_{R_k} $, $\boldu_{R_k}\otimes \boldb_{R_k} $, and $\boldb_{R_k}\otimes \boldu_{R_k}$.  Hence by a standard argument, we can show that $(\boldu,\boldb)$ satisfies \eqref{eq:magnetic}. It is obvious that both $\boldu$ and $\boldb$ are divergence-free.
This completes the proof of Theorem \ref{thm:GWP}. \hfill $\square$

\section{Proof of Theorem \ref{thm:uniqueness}}\label{sec:uniqueness}

This section is devoted to proving Theorem \ref{thm:uniqueness}.
By Theorem \ref{thm:GWP} and its proof, there exists a global weak solution $(\boldu,\boldb )$ of  \eqref{eq:MHD-relaxed}, which also satisfies  $\boldu  = \mathbf{U}_\alpha * (\boldb \otimes \boldb )$ in $\mathbb{R}^d\times [0, \infty )$.
Since $\boldb$ is divergence-free, $\boldb$ is a weak solution of \eqref{eq:PFHE} with $\boldF = \boldb \otimes \boldu$. It follows from Lemma \ref{lem:div-free} that $\boldb \otimes \boldu \in \Leb{2\beta /(2\beta-1)}(0,T;\Leb{2})$ for any $T<\infty$.
Hence by Theorem \ref{thm:energy-equality}, $\boldb$    belongs to $C([0, \infty );\Leb{2})$ and satisfies
\begin{equation}\label{energy identity-b-F2}
\frac{1}{2} \frac{d}{dt} \norm{\boldb (t)}{\NLeb{2}}^2+  \norm{\Lambda^\beta \boldb(t)}{\NLeb{2}}^2   =-   \int_{\mathbb{R}^d}  \boldb (t)  \otimes \boldu (t) : \nabla \boldb (t) \myd{x}
\end{equation}
for almost all $t  \in (0, \infty )$. Using Young's inequality (Lemma \ref{lem:Young}), we easily  deduce that $\boldu  \in C( [ 0,\infty );\Lor{{d}/{(d+1-2\alpha)}}{\infty})$.  Let  $1< p< d/(2\alpha -1)$ be chosen  so that $H^\beta \hookrightarrow L_{2p}$. Then it follows from H\"{o}lder's inequality and Lemma \ref{lem:product2} that
$$
\boldb (t) \otimes \boldb (t)  \in L_p \cap \dot{H}^{1-\alpha} \quad\mbox{and}\quad
\Div\dsp(\boldb (t) \otimes \boldb (t) ) = ( \boldb (t) \cdot \nabla ) \boldb (t) \in L_{2p/(p+1)}
$$
for almost all $t \in (0,\infty)$. Hence by Theorem \ref{theorem:Frational Stokes-Lpdata} and Lemma \ref{lem:div-free}, $\boldu$ satisfies
\begin{equation}\label{eq:energy-identity-u2}
   \int_{\mathbb{R}^d} |\Lambda^{\alpha} \boldu (t) |^2 \myd{x}=  \int_{\mathbb{R}^d} ( \boldb (t) \cdot \nabla ) \boldb (t) \cdot  \boldu (t) \, dx
\end{equation}
for almost all  $t  \in (0,\infty ) $. Adding (\ref{energy identity-b-F2}) and (\ref{eq:energy-identity-u2}) and then integraing over $(0,t)$, we get the energy identity (\ref{eq:energy-equality}) for any $t \ge 0$.

To complete the proof of Theorem \ref{thm:uniqueness} it remains to prove the uniqueness assertion.
Let $(\boldu^1,\boldb^1)$ and $(\boldu^2,\boldb^2)$ be two weak solutions of \eqref{eq:MHD-relaxed} for the same initial data $\boldb_0 \in \Leb{2}$ with ${\rm div} \, \boldb_0 =0$.
Write $\tilde{\boldu}=\boldu^1-\boldu^2$ and $\tilde{\boldb}=\boldb^1-\boldb^2$. Then  $(\tilde{\boldu},\tilde{\boldb})$ satisfies
\[ \partial_t \tilde{\boldb}+\Lambda^{2\beta}\tilde{\boldb}+(\boldu^1\cdot\nabla)\tilde{\boldb}=
\Div\dsp (\tilde{\boldb}\otimes \boldu^1+\boldb^2 \otimes \tilde{\boldu}-\tilde{\boldu}\otimes\boldb^2) \]
and
\[	\tilde{\boldu} = \mathbf{U}_\alpha * (\tilde{\boldb}\otimes \boldb^1+\boldb^2 \otimes \tilde{\boldb}).\]
Moreover, by Lemma \ref{lem:div-free},
\[
\boldu^i \otimes \boldb^j \in L_{2\beta /(2\beta -1)}(0,T; L_2 ) \quad\mbox{and}\quad  \boldb^i \in L_{2\beta}(0,T; \dot{H}^1 )
\]
for any $T< \infty$, where $i,j=1,2$.
Hence by   Theorem \ref{thm:energy-equality}, we have
 \begin{align*}
	\frac{1}{2}\frac{d}{dt} \norm{\tilde{\boldb}}{\NLeb{2}}^2+\norm{\Lambda^\beta \tilde{\boldb}}{\NLeb{2}}^2&=- \int_{\mathbb{R}^d} \tilde{\boldb}\otimes \boldu^1:\nabla \tilde{\boldb}\myd{x}+\int_{\mathbb{R}^d} \left(-\boldb^2 \otimes \tilde{\boldu} +\tilde{\boldu}\otimes \boldb^2 \right): \nabla \tilde{\boldb}\myd{x}\\
	&=I+II.
\end{align*}
Hence by Gronwall's inequality, it suffices to show that
\begin{equation}\label{eq:Gronwall-estimate}
	I+II\leq C\phi(t)\norm{\tilde{\boldb}}{\NLeb{2}}^2+\frac{1}{2}\norm{\Lambda^\beta \tilde{\boldb}}{\NLeb{2}}^2
\end{equation}
for some constant $C$, where $\phi$ is a locally integrable function defined by
\[	\phi(t)=1+\norm{\Lambda^\alpha \boldu^1(t)}{\NLeb{2}}^2 + \norm{\Lambda^\alpha \boldu^2(t)}{\NLeb{2}}^2+\norm{\Lambda^\beta \boldb^1(t)}{\NLeb{2}}^2 + \norm{\Lambda^\beta \boldb^2(t)}{\NLeb{2}}^2. \]
To do this, we make crucial use of Lemma \ref{lem:product-estimate}; that is, we  choose  $0<\theta_1,\theta_2<1$ with $\theta_1+\theta_2+1/\beta \leq 2$ such that
\begin{equation}\label{est-ui and bj}
\norm{\boldu^i \otimes \boldb^j}{\NLeb{2}}
 \apprle  \norm{\boldu^i}{\NLor{d/(d+1-2\alpha)}{\infty}}^{1-\theta_1}
 \norm{\Lambda^\alpha\boldu^i}{\NLeb{2}}^{\theta_1}
 \norm{\boldb^j}{\NLeb{2}}^{1-\theta_2}\norm{\Lambda^\beta \boldb^j}{\NLeb{2}}^{\theta_2}
\end{equation}
for $i,j=1,2$. Note also  that if $\alpha<1$ and $\mu={(1-\alpha)}/{\beta}$, then since $\alpha+(1-\mu)\beta \geq d/2+1$, the numbers $\theta_1$ and $\theta_2$ can be chosen  so that $(1+\mu)\theta_1+\theta_2+1/\beta\leq 2$.

Now,  by (\ref{est-ui and bj}) and Proposition \ref{prop:Gagliardo}, we  obtain
\begin{align*}
I&\leq \norm{\tilde{\boldb}\otimes \boldu^1}{\NLeb{2}}\norm{\nabla \tilde{\boldb}}{\NLeb{2}}\\
&\apprle  \norm{\boldu^1}{\NLor{d/(d+1-2\alpha)}{\infty}}^{1-\theta_1}\norm{\Lambda^\alpha\boldu^1}{\NLeb{2}}^{\theta_1}\norm{\tilde{\boldb}}{\NLeb{2}}^{1-\theta_2}\norm{\Lambda^\beta \tilde{\boldb}}{\NLeb{2}}^{\theta_2} \norm{\tilde{\boldb}}{\NLeb{2}}^{1-1/\beta}\norm{\Lambda^\beta\tilde{\boldb}}{\NLeb{2}}^{1/\beta}.
\end{align*}
and
\begin{equation}\label{eq:estimate-II}
II\apprle \norm{\tilde{\boldu}}{\NLor{d/(d+1-2\alpha)}{\infty}}^{1-\theta_1}\norm{\Lambda^\alpha\tilde{\boldu}}{\NLeb{2}}^{\theta_1}\norm{\boldb^2}{\NLeb{2}}^{1-\theta_2}\norm{\Lambda^\beta \boldb^2}{\NLeb{2}}^{\theta_2}\norm{\tilde{\boldb}}{\NLeb{2}}^{1-1/\beta}\norm{\Lambda^\beta \tilde{\boldb}}{\NLeb{2}}^{1/\beta}.
\end{equation}

It is easy to estimate $I$. Indeed, by Lemma \ref{lem:b_R and u_R},  the energy identity \eqref{eq:energy-equality}, and Young's inequality, we have
\begin{align*}
I&\apprle \norm{\boldb_0}{\NLeb{2}}^{2(1-\theta_1)} \left(\norm{\Lambda^\alpha {\boldu}^1}{\NLeb{2}}^{\theta_1}\norm{\tilde{\boldb}}{\NLeb{2}}^{2-\theta_2-1/\beta} \right)\norm{\Lambda^\beta \tilde{\boldb}}{\NLeb{2}}^{\theta_2 +1/\beta},\\
&\leq C \norm{\Lambda^{\alpha}\boldu^1}{\NLeb{2}}^{\frac{2\theta_1}{2-\theta_2-1/\beta}}\norm{\tilde{\boldb}}{\NLeb{2}}^2+\frac{1}{4}\norm{\Lambda^\beta \tilde{\boldb}}{\NLeb{2}}^2\\
&\leq C \phi(t)^{\frac{\theta_1}{2-\theta_2-1/\beta}}\norm{\tilde{\boldb}}{\NLeb{2}}^2+\frac{1}{4}\norm{\Lambda^\beta \tilde{\boldb}}{\NLeb{2}}^2
\end{align*}
for some $C=C(\alpha,\beta,d,\norm{\boldb_0}{\NLeb{2}})$. Since  $\theta_1 / (2-\theta_2-1/\beta)\leq 1$ and $\phi(t)\geq 1$, we get
\[	I\leq C \phi(t)\norm{\tilde{\boldb}}{\NLeb{2}}^2+\frac{1}{4}\norm{\Lambda^\beta \tilde{\boldb}}{\NLeb{2}}^2.\]

To estimate  $II$, we recall  that
\[	\tilde{\boldu} = \mathbf{U}_\alpha * (\tilde{\boldb}\otimes \boldb^1+\boldb^2 \otimes \tilde{\boldb}).\]
By Theorems \ref{theorem:Frational Stokes-Lpdata} and \ref{theorem:Frational Stokes-Lpdata-intersection}, we have
\begin{equation}\label{eq:weak-Lp-estimates}
\norm{\tilde{\boldu}}{\NLor{d/(d+1-2\alpha)}{\infty}}\apprle \left(\norm{\boldb^1}{\NLeb{2}}+\norm{\boldb^2}{\NLeb{2}} \right)\norm{\tilde{\boldb}}{\NLeb{2}} \apprle \norm{\tilde{\boldb}}{\NLeb{2}}
\end{equation}
and
\begin{equation}\label{eq:difference-u}
\norm{\Lambda^\alpha \tilde{\boldu}}{\NLeb{2}} \apprle \norm{\tilde{\boldb}\otimes \boldb^1+\boldb^2 \otimes \tilde{\boldb}}{\thSob{1-\alpha}}.
\end{equation}

Suppose first that $\alpha \geq 1$. Then
\[	\thSob{1-\alpha}=\left(\thSob{\alpha-1}\right)^*\quad \text{and}\quad \thSob{\alpha-1}\hookrightarrow \Leb{r},\]
where $1<r<\infty$ is defined by ${1}/{r}={1}/{2}-{(\alpha-1)}/{d}$.
Fix any $\Phi \in \thSob{\alpha-1}$. Then by H\"older's inequality and Proposition \ref{prop:Gagliardo}, we have
\begin{align*}
&\int_{\mathbb{R}^d} \left(\tilde{\boldb}\otimes \boldb^1+\boldb^2\otimes \tilde{\boldb} \right):\Phi \myd{x}\\
&\quad \leq \left(\norm{\boldb^1}{\NLeb{2r'}}+\norm{\boldb^2}{\NLeb{2r'}}\right)\norm{\tilde{\boldb}}{\NLeb{2r'}}\norm{\Phi}{\NLeb{r}}\\
&\quad \apprle \left(\norm{\boldb^1}{\NLeb{2r'}}+\norm{\boldb^2}{\NLeb{2r'}}\right)\norm{\tilde{\boldb}}{2r'}\norm{\Phi}{\thSob{\alpha-1}}\\
&\quad \apprle \left(\norm{\boldb^1}{2}^{1-\lambda}\norm{\boldb^1}{\thSob{\beta}}^{\lambda}+\norm{\boldb^2}{2}^{1-\lambda}\norm{\boldb^2}{\thSob{\beta}}^{\lambda} \right)\norm{\tilde{\boldb}}{2}^{1-\lambda}\norm{\tilde{\boldb}}{\thSob{\beta}}^{\lambda} \norm{\Phi}{\thSob{\alpha-1}},
\end{align*}
where
\[	0<\lambda =\frac{d}{2\beta r}= \frac{d+2-2\alpha}{4\beta}\leq \frac{1}{2}.\]
Hence by \eqref{eq:difference-u},
\begin{equation}\label{eq:phi-estimate-u}
\begin{aligned}
\norm{\Lambda^\alpha\tilde{\boldu}}{2}&\apprle\left(\norm{\boldb^1}{2}^{1-\lambda}\norm{\boldb^1}{\thSob{\beta}}^{\lambda}+\norm{\boldb^2}{2}^{1-\lambda}\norm{\boldb^2}{\thSob{\beta}}^{\lambda} \right) \norm{\tilde{\boldb}}{2}^{1-\lambda}\norm{\tilde{\boldb}}{\thSob{\beta}}^{\lambda}\\
&\apprle  \phi(t)^{\lambda/2}\norm{\tilde{\boldb}}{2}^{1-\lambda}\norm{\Lambda^\beta\tilde{\boldb}}{\NLeb{2}}^{\lambda}.
\end{aligned}
\end{equation}
Combining  \eqref{eq:estimate-II}, \eqref{eq:weak-Lp-estimates}, and \eqref{eq:phi-estimate-u}, we have
\begin{align*}
II&\leq C\norm{\tilde{\boldb}}{\NLeb{2}}^{1-\theta_1}\left(\phi(t)^{\lambda/2}\norm{\tilde{\boldb}}{\NLeb{2}}^{1-\lambda}\norm{\Lambda^\beta \tilde{\boldb}}{\NLeb{2}}^{\lambda} \right)^{\theta_1}\phi(t)^{\theta_2/2}\norm{\tilde{\boldb}}{\NLeb{2}}^{1-1/\beta}\norm{\Lambda^\beta \tilde{\boldb}}{\NLeb{2}}^{1/\beta} \\
&=C  \phi(t)^{(\lambda\theta_1 +\theta_2)/2}\norm{\tilde{\boldb}}{2}^{2-\lambda\theta_1-1/\beta}\norm{\Lambda^\beta \tilde{\boldb}}{\NLeb{2}}^{\lambda\theta_1 + 1/\beta}   \\
&\leq  C\phi(t)^{\frac{\lambda\theta_1 +\theta_2}{2-\lambda\theta_1-1/\beta}}\norm{\tilde{\boldb}}{\NLeb{2}}^2+\frac{1}{4}\norm{\Lambda^\beta \tilde{\boldb}}{\NLeb{2}}^2.
\end{align*}
Since  $\theta_1+\theta_2+1/\beta\leq 2$ and  $2\lambda \leq 1$, it follows that
\[	\frac{\lambda\theta_1+\theta_2}{2-\lambda\theta_1 - 1/\beta}\leq 1.\]
Therefore,
\[	II\leq  C\phi(t)\norm{\tilde{\boldb}}{\NLeb{2}}^2+\frac{1}{4}\norm{\Lambda^\beta \tilde{\boldb}}{\NLeb{2}}^2. \]

Suppose next that $\alpha <1$. Then since $\beta \geq \frac{d}{2}+1-\alpha>\frac{d}{2}$, it follows from \eqref{eq:difference-u} 
and Lemma \ref{lem:commutator} that
\begin{align*}
\norm{\Lambda^\alpha \tilde{\boldu}}{\NLeb{2}} \apprle \norm{\tilde{\boldb}\otimes \boldb^1+\boldb^2\otimes \tilde{\boldb}}{\thSob{1-\alpha}}\apprle \left(\norm{\boldb^1}{\tSob{1-\alpha}}+\norm{\boldb^2}{\tSob{1-\alpha}} \right)\norm{\tilde{\boldb}}{\tSob{\beta}}.
\end{align*}
Moreover,    by Proposition \ref{prop:Gagliardo},
\begin{align*}
\norm{\boldb^i}{\tSob{1-\alpha}}&\apprle \norm{\boldb^i}{\NLeb{2}}+\norm{\boldb^i}{\thSob{1-\alpha}}\\
&\apprle \norm{\boldb^i}{\NLeb{2}}+\norm{\boldb^i}{\NLeb{2}}^{1-\mu}\norm{\boldb^i}{\thSob{\beta}}^{\mu}\apprle \phi(t)^{\mu/2}
\end{align*}
for $i=1,2$, where $\mu=(1-\alpha)/\beta$.
Hence
\begin{align*}
II&\leq C\norm{\tilde{\boldb}}{\NLeb{2}}^{1-\theta_1}\left[\phi(t)^{\mu/2}\left(\norm{\tilde{\boldb}}{\NLeb{2}}+\norm{\tilde{\boldb}}{\thSob{\beta}} \right)\right]^{\theta_1}\phi(t)^{\theta_2/2}\norm{\tilde{\boldb}}{\NLeb{2}}^{1-1/\beta}\norm{\Lambda^\beta \tilde{\boldb}}{\NLeb{2}}^{1/\beta} \\
&\leq C \phi(t)^{(\mu\theta_1+\theta_2)/2}\norm{\tilde{\boldb}}{2}^{2-1/\beta}\norm{\Lambda^\beta \tilde{\boldb}}{\NLeb{2}}^{1/\beta} \\
&\relphantom{=}+C  \phi(t)^{(\mu\theta_1 +\theta_2)/2}\norm{\tilde{\boldb}}{2}^{2-\theta_1-1/\beta}\norm{\Lambda^\beta \tilde{\boldb}}{\NLeb{2}}^{\theta_1 + 1/\beta}   \\
&\leq  C\left[\phi(t)^{\frac{\mu\theta_1+\theta_2}{2-1/\beta}}+\phi(t)^{\frac{\mu \theta_1+\theta_2}{2-\theta_1-1/\beta}}\right]\norm{\tilde{\boldb}}{\NLeb{2}}^2+\frac{1}{4}\norm{\Lambda^\beta \tilde{\boldb}}{\NLeb{2}}^2.
\end{align*}
Since $(1+\mu)\theta_1+\theta_2+1/\beta\leq 2$, it follows that
\[	\frac{\mu\theta_1+\theta_2}{2-1/\beta}\leq \frac{\mu \theta_1+\theta_2}{2-\theta_1-1/\beta}\leq 1. \]
Therefore,
\[	II\leq  C\phi(t)\norm{\tilde{\boldb}}{\NLeb{2}}^2+\frac{1}{4}\norm{\Lambda^\beta \tilde{\boldb}}{\NLeb{2}}^2. \]
This completes the proof of \eqref{eq:Gronwall-estimate}. By Gronwall's inequality, we conclude that $\boldb^1=\boldb^2$   on $[0,\infty)$. Moreover, it follows from  \eqref{eq:weak-Lp-estimates} that $\boldu^1=\boldu^2$   on $[0,\infty)$. This completes the proof of Theorem \ref{thm:uniqueness}. \hfill \qed

\appendix

\section{Completion of the proof of Lemma \ref{lem:density-approx-S0}}\label{sec:app-b}

To complete the proof of Lemma \ref{lem:density-approx-S0}, it remains to prove the following result.

\begin{lem}\label{lem:density}  $\mathscr{S}_0$ is dense in $\Leb{p}\cap \thSob{s}$ for  $1<p<\infty$ and $s\in \mathbb{R}$.
\end{lem}

Lemma \ref{lem:density} will be proved by using   Littlewood-Paley operators as in the proofs of \cite[Proposition 2.27]{BCD11} and \cite[5.1.5/Theorem]{T83}. Let $\psi$ and $\varphi$ be radial $C^\infty$-functions on $\mathbb{R}^d$ such that
$$
{\rm supp}\, \psi\subset B_{4/3}, \quad {\rm supp}\, \varphi \subset \{ \xi \in \mathbb{R}^d : 3/4 \leq |\xi|\leq 8/3\},
$$
\[	\psi(\xi)+\sum_{j =0}^\infty \varphi(2^{-j}\xi)=1 \quad\mbox{for all}\,\, \xi \in \mathbb{R}^d , \]
and
\[	\sum_{j=-\infty}^\infty \varphi(2^{-j}\xi)=1 \quad\mbox{for all}\,\, \xi \neq 0.\]
For $j\in \mathbb{Z}$, let $\psi_j(\xi)=\psi(2^{-j}\xi)$ and $\varphi_j(\xi)=\varphi(2^{-j}\xi)$. Then the Littlewood-Paley operators $S_j$ and $\Delta_j$ are defined by
\[	S_j u = (\psi_j\hat{u})^{\vee}=\psi_j^{\vee}*u \quad \text{and}\quad \Delta_j u = (\varphi_j\hat{u})^{\vee}=\varphi_j^{\vee} * u	\]
for $u\in \mathscr{S}_0'$. It immediately follows from Young's convolution inequality that
\begin{equation}\label{eq:Lp-boundness-Fourier}
	\norm{S_j u}{\NLeb{p}} + \norm{\Delta_j u}{\NLeb{p}} \leq C\norm{u}{\NLeb{p}} \quad (1\leq p \le \infty)
\end{equation}
for some constant $C$ independent of $j$.

The following inequalities of Bernstein type are also easily proved by using Young's convolution inequality  (see e.g. \cite[Lemma 2.1]{BCD11} or \cite[Lemma 2.4]{WWY21}).
\begin{prop}\label{prop:Bernstein}
Let $1\leq p\leq q\leq \infty$, $k \in \mathbb{N}_0$, and $0<R<\infty$. Then for all $u\in \Leb{p}$ with $\supp \hat{u}\subset B_R$, we have
\[	\norm{D^k u}{\NLeb{q}}:=\max_{|\beta|=k} \norm{D^\beta u}{\NLeb{q}}\apprle R^{k+d(1/p-1/q)} \norm{u}{\NLeb{p}}.\]
Moreover, if $u\in \Leb{p}$ and $\supp \hat{u}\subset B_R\setminus B_{R/2}$, then
\[	\norm{D^k u}{\NLeb{p}}\approx R^k \norm{u}{\NLeb{p}}.\]
In particular, for all $u\in \Leb{p}$, we have
\begin{align}
\norm{\Delta_j u}{\NLeb{q}}&\apprle 2^{jd(1/p-1/q)}\norm{\Delta_j u}{\NLeb{p}}, \\
\norm{S_j u}{\NLeb{q}}&\apprle 2^{jd(1/p-1/q)}\norm{S_j u}{\NLeb{p}}, \\
\norm{D^k(\Delta_j u)}{\NLeb{p}}&\approx 2^{jk}\norm{\Delta_j u}{\NLeb{p}}.\label{eq:LP-3}
\end{align}
\end{prop}

The following is the classical Littlewood-Paley characterization of the homogeneous Sobolev space  $\thSob{s}_p$ (see e.g. \cite[Theorem 6.2.7]{G14-1}).
\begin{prop}\label{prop:Littlewood-Paley}
Let $s \in \mathbb{R}$ and $1<p<\infty$.  For all $u \in  \thSob{s}_p$, we
have
\[
\norm{u}{\thSob{s}_p} \approx  \left\| \left(\sum_{j\in \mathbb{Z}} 2^{2js} |{\Delta}_j u |^2 \right)^{1/2} \right\|_p .
\]
\end{prop} 

\begin{proof}[Proof of Lemma \ref{lem:density}] Suppose that $u \in \Leb{p}\cap \thSob{s}$.
Let $\varepsilon>0$ be given. Then by Proposition \ref{prop:Littlewood-Paley} with $s=0$ and $p=2$, respectively, 
there exists  $N\in \mathbb{N}$ such that
\[	\norm{u_N-u}{\NLeb{p}}+\norm{u_N-u}{\thSob{s}}<\frac{\varepsilon}{2},\]
where $u_N=\sum_{|j|\leq N} \Delta_j u$.
Note that $u_N$ belongs to $L_p \cap L_2$ and has  support in an annulus. Hence  it follows from Proposition \ref{prop:Bernstein} that  $u_N \in H_2^k \cap H_p^k$ for all $k \in \mathbb{N}$.

Choose $\theta \in C_c^\infty(B_2)$ with $\theta=1$ in $B_1$,  and let $\theta_R(x)=\theta(x/R)$ for $R>0$. Fixing an integer $M>N$ (say, $M=N+1$), we  define
\[	u_{N}^R=(I-S_{-M})(\theta_R u_N). \]
Note then that
\begin{equation}\label{eq:approximation-diff}
u_{N}^R \in \mathscr{S}_0 \quad\mbox{and}\quad 	u_{N}^R-u_N=(I-S_{-M})((\theta_R-1)u_N).
\end{equation}
By \eqref{eq:Lp-boundness-Fourier}, we have
\[	\norm{u_{N}^R-u_{N}}{\NLeb{p}}\apprle \norm{(\theta_R-1)u_N}{\NLeb{p}}.\]
Since $u_N \in \Leb{p}$ and $(\theta_R-1)u_N\rightarrow 0$ as $R\rightarrow \infty$, it follows from the dominated convergence theorem that
$$
 \norm{u_{N}^R-u_{N}}{\NLeb{p}}\rightarrow 0 \quad\mbox{as}\,\,R\rightarrow \infty.
$$

Let $k=\max \, \{0,\lfloor s \rfloor +2 \}$, where $\lfloor s \rfloor$ is  the greatest integer less than or equal to $s$.  If $j\geq 0$,  then  by  \eqref{eq:approximation-diff}, \eqref{eq:LP-3}, and  \eqref{eq:Lp-boundness-Fourier}, we have  
\begin{align*}
2^{js}\norm{\Delta_j (u_{N}^R-u_N)}{\NLeb{2}}&\leq 2^{j(k-1)}\norm{\Delta_j (I-S_{-M})((\theta_R-1)u_N)}{\NLeb{2}}\\
&\approx 2^{-j} \norm{D^k[\Delta_j [(I-S_{-M})((\theta_R-1)u_N)]]}{\NLeb{2}}\\
&\apprle 2^{-j} \norm{D^k((\theta_R-1)u_N)}{\NLeb{2}}.
\end{align*}
If $-M-1\leq j\leq -1$, then by \eqref{eq:Lp-boundness-Fourier} and \eqref{eq:approximation-diff},
\[	2^{js}\norm{\Delta_j(u_{N}^R-u_N)}{\NLeb{2}}\apprle 2^{js}\norm{(\theta_R-1)u_N}{\NLeb{2}}.	\]
Finally, noting  that
$$
\Delta_j(u_{N}^R-u_N)=0 \quad\mbox{for all}\,\, j\leq -M-2 ,
$$
we obtain
\begin{align*}
	\norm{u_{N}^R-u_N}{\thSob{s}}^2&\approx \int_{\mathbb{R}^d}\left(  \sum_{j\in \mathbb{Z}} 2^{2js}|\Delta_j (u_{N}^R-u)|^2 \right)  \myd{x}\\
	&\apprle \sum_{j=-M-1}^{-1} 2^{2js}\norm{(\theta_R-1)u_N}{\NLeb{2}}^2+\sum_{j=0}^\infty 2^{-j} \norm{D^k((\theta_R-1)u_N)}{\NLeb{2}}^2\\
&\apprle  \norm{(\theta_R-1)u_N}{\tSob{k}}^2 .
\end{align*}
 Hence by  the dominated convergence theorem again, we have
$$
\norm{u_{N}^R-u_N}{\thSob{s}}^2 \to 0 \quad\mbox{as}\,\, R\rightarrow \infty.
$$
 This completes the proof of Lemma \ref{lem:density}.
\end{proof}

\section{Proof of Lemma \ref{prop:Fourier-transform-Gaussian-derivatives}}\label{sec:Gaussian}

Our proof of Lemma \ref{prop:Fourier-transform-Gaussian-derivatives} is based on the following well-known result (see e.g. \cite[Chapter V, Lemma 2]{S70} and \cite[Theorem 2.4.6]{G14-2}).

\begin{lem}\label{prop:Fourier-transform-Gaussian}
If $0<\Re \lambda<d$, then
\begin{equation}\label{eq:Fourier-transform-identity}
   \int_{\mathbb{R}^d} \frac{1}{|\xi|^\lambda} \phi^{\vee}(\xi) \,d\xi= \frac{\Gamma((d-\lambda)/2) }{\Gamma(\lambda/2)\pi^{d/2- \lambda } }\int_{\mathbb{R}^d} \frac{1}{ |x|^{d-\lambda}}  \phi(x) \, dx
\end{equation}
 for all $\phi \in \mathscr{S}$.
\end{lem}

We will also use the method of analytic continuation. The following result is useful to show the analyticity of some functions defined by integrals.

\begin{lem}\label{lem:analyticity}
Fix $\phi \in \mathscr{S}$ and $\beta \in {\mathbb N}_0^d$. Then the function
\[   f (z)  = \int_{\mathbb{R}^d}  x^\beta |x|^z  \phi(x)\myd{x} \]
is holomorphic in the region $H=\{z \in \mathbb{C} : \Re z >-d -|\beta |\}$.
\end{lem}
\begin{proof}
Let $z\in H  $ be fixed. Choose $\varepsilon >0$ such that $\{w \in \mathbb{C} : |w-z| \le \varepsilon \}\subset H$.

Suppose that  $x\neq 0$. Then the function $  |x|^w =   \exp ( w \ln |x|)$ is entire and
\[
\frac{d}{dw} \left[ |x|^w \right] = ( \log |x| ) |x|^{w} .
\]
Hence for all $w \in \overline{B_\varepsilon(z)}$ with $w \neq z$,
\[    \left|\frac{|x|^{w}-|x|^z}{w -z} \right|\leq  | \log|x| |\max_{|\eta -z| \le \varepsilon  } ||x|^{\eta}|=   | \log |x||  |x|^{ \Re z + \varepsilon}.
\]
 Moreover, since $\Re z + \varepsilon > -d  -|\beta|$, we have
\[
\int_{\mathbb{R}^d} \left|   x^\beta \log |x|   |x|^{\Re z + \varepsilon} \phi(x)\right| \myd{x} = \int_{\mathbb{R}^d} | \log |x| |  |x|^{\Re z + \varepsilon  + |\beta|} | \phi(x) | \myd{x} < \infty.
\]
Hence it follows from the dominated convergence theorem that
\[    \lim_{w\rightarrow z} \frac{f(w)-f(z)}{w- z} = \int_{\mathbb{R}^d} x^\beta (\log |x| ) |x|^z  \phi(x) \, dx.  \]
 This completes the proof of Lemma \ref{lem:analyticity}.
\end{proof}

Recall that
\[   (D^\beta \phi)^{\vee} (\xi)=(-2\pi i \xi)^\beta {\phi}^{\vee} (\xi),
\quad  D^\beta \widehat{\phi}(\xi)= \left( (-2\pi i x)^\beta \phi \right)^{\wedge} (\xi)     \]
for   $ \phi \in \mathscr{S}$ and $\beta \in \mathbb{N}_0^d$.

\begin{proof}[Proof of Lemma \ref{prop:Fourier-transform-Gaussian-derivatives} (1)]  Assume  that $0< \Re \lambda < d$. Set $\phi =D_j \psi $.
Then since $ {\phi}^{\vee}(\xi) = (-2 \pi i \xi_j ){\psi}^{\vee}(\xi)$, it follows from Lemma \ref{prop:Fourier-transform-Gaussian} that
\[
  \int_{\mathbb{R}^d} \frac{-2 \pi i \xi_j}{|\xi|^{\lambda}}  {\psi}^{\vee}(\xi) \, d\xi = \frac{\Gamma((d-\lambda)/2) }{\Gamma(\lambda/2) \pi^{d/2- \lambda } }
  \int_{\mathbb{R}^d}  \frac{1}{|x|^{d-\lambda}} D_j \psi(x) \, dx.
\]
By the fundamental property of the Gamma function,
\[
 (d-\lambda)  \int_{\mathbb{R}^d} \frac{-2 \pi i \xi_j}{|\xi|^{\lambda}}  {\psi}^{\vee}(\xi) \, d\xi
   =\frac{2\Gamma((d+2-\lambda)/2)  }{  \Gamma( \lambda/2)\pi^{d/2- \lambda }  } \int_{\mathbb{R}^d}  \frac{1}{|x|^{d-\lambda}} D_j \psi(x) \, dx ,
\]
which holds for $0< \Re \lambda < d+1$, by analytic continuation  based on Lemma \ref{lem:analyticity}.

Assume next that $1< \Re \lambda < d+1$.
Then performing the integration by parts, we easily obtain
\[
\int_{\mathbb{R}^d}  \frac{1}{|x|^{d-\lambda}} D_j \psi(x) \, dx =  (d-\lambda ) \int_{\mathbb{R}^d}  \frac{x_j}{|x|^{d+2-\lambda }} \psi(x) \, dx .
\]
Hence, if $d< \Re \lambda < d+1$, then
\[
\int_{\mathbb{R}^d} \frac{-2 \pi i \xi_j}{|\xi|^{\lambda}}  {\psi}^{\vee}(\xi) \, d\xi
   =\frac{2\Gamma((d+2-\lambda)/2)  }{  \Gamma( \lambda/2)\pi^{d/2- \lambda }  } \int_{\mathbb{R}^d}  \frac{x_j}{|x|^{d+2-\lambda }} \psi(x) \, dx .
\]
This identity holds for $1< \Re \lambda < d+1$, by Lemma \ref{lem:analyticity} and analytic continuation again. This completes the proof of Lemma \ref{prop:Fourier-transform-Gaussian-derivatives} (1).
\end{proof}

\begin{proof}[Proof of Lemma \ref{prop:Fourier-transform-Gaussian-derivatives} (2)] Assume  that $1< \Re \lambda < d+1$.
Then by Lemma \ref{prop:Fourier-transform-Gaussian-derivatives} (1),
\[
\int_{\mathbb{R}^d} \frac{-2 \pi i \xi_j}{|\xi|^{\lambda}}  {\phi}^{\vee}(\xi) \, d\xi
   =\frac{2\Gamma((d+2-\lambda)/2)  }{  \Gamma( \lambda/2)\pi^{d/2- \lambda }  } \int_{\mathbb{R}^d}  \frac{x_j}{|x|^{d+2-\lambda }} \phi(x) \, dx.
\]
for all $\phi \in \mathscr{S}$.
Setting  $\phi =D_{k} \psi$, we have  
\[
 \int_{\mathbb{R}^d}   \frac{(-2 \pi i \xi_j )(-2 \pi i \xi_k )}{|\xi|^\lambda}  {\psi}^{\vee}(\xi) \, d\xi=\frac{2\Gamma((d+2-\lambda)/2)  }{  \Gamma( \lambda/2)\pi^{d/2- \lambda }  } \int_{\mathbb{R}^d}  \frac{x_j}{|x|^{d+2-\lambda }} D_k \psi(x) \, dx,
\]
which holds for $1< \Re \lambda < d+2$, by   analytic continuation.
Moreover, if   $d+1< \Re \lambda < d+2$, then
\[
\int_{\mathbb{R}^d}  \frac{x_j}{|x|^{d+2-\lambda }} D_{k} \psi(x) \, dx=  \int_{\mathbb{R}^d}  \frac{(d+2-\lambda )x_j x_k -\delta_{jk} |x|^2}{|x|^{d+4-\lambda }}   {\psi}(x)\myd{x} .
\]
Therefore, the proof of Lemma \ref{prop:Fourier-transform-Gaussian-derivatives} (2) is completed by analytic continuation.
\end{proof}

\begin{proof}[Proof of Lemma \ref{prop:Fourier-transform-Gaussian-derivatives} (3)]
Assume  that $2< \Re \lambda < d+2$.
Then by Lemma \ref{prop:Fourier-transform-Gaussian-derivatives} (2),
\begin{align*}
& \int_{\mathbb{R}^d}   \frac{(-2 \pi i \xi_j )(-2 \pi i \xi_k )}{|\xi|^\lambda}  {\phi}^{\vee}(\xi) \, d\xi \\
&\quad  = \frac{2\Gamma((d+2-\lambda)/2)  }{  \Gamma( \lambda/2) \pi^{d/2-\lambda}}   \int_{\mathbb{R}^d}   \frac{(d+2-\lambda )x_j x_k -\delta_{jk} |x|^2}{|x|^{d+4-\lambda }}   {\phi}(x) \, dx
\end{align*}
for all $\phi \in \mathscr{S}$. Setting $\phi =D_{l} \psi$, we have
\begin{align*}
 &\int_{\mathbb{R}^d}   \frac{(-2\pi i \xi_j)(-2\pi i \xi_k)(-2\pi i \xi_l)}{|\xi|^\lambda} {\psi}^{\vee}(\xi) \, d\xi \\
 &\quad =  \frac{2\Gamma((d+2-\lambda)/2)  }{  \Gamma( \lambda/2) \pi^{d/2-\lambda}}   \int_{\mathbb{R}^d}   \frac{(d+2-\lambda )x_j x_k -\delta_{jk} |x|^2}{|x|^{d+4-\lambda }}  D_l {\psi}(x) \, dx  \\
 &\quad = \frac{4\Gamma((d+4-\lambda)/2) }{(d+2-\lambda ) \Gamma( \lambda/2) \pi^{d/2-\lambda}}
   \int_{\mathbb{R}^d}   \frac{(d+2-\lambda )x_j x_k -\delta_{jk} |x|^2}{|x|^{d+4-\lambda }}  D_l {\psi}(x) \, dx ,
 \end{align*}
which holds  for $2< \Re \lambda < d+3$ with $\lambda \neq d+2$, by the analytic continuation.
Moreover, if $3< \Re \lambda < d+3$, then
\[
\int_{\mathbb{R}^d} \frac{1}{|x|^{d+2-\lambda}} D_{l} {\psi}(x) \myd{x}=  (d+2-\lambda ) \int_{\mathbb{R}^d}  \frac{x_l}{|x|^{d+4-\lambda }}  \psi(x) \, dx .
\]
Hence for $d+2< \Re \lambda < d+3$, we have
\begin{align*}
 & \int_{\mathbb{R}^d}   \frac{(-2\pi i \xi_j)(-2\pi i \xi_k)(-2\pi i \xi_l)}{|\xi|^\lambda} {\psi}^{\vee}(\xi) \, d\xi \\
  & \quad  = \frac{4\Gamma((d+4-\lambda)/2)  }{ \Gamma( \lambda/2) \pi^{d/2-\lambda}}   \int_{\mathbb{R}^d} \left[- D_l \left( \frac{   x_j x_k}{|x|^{d+4-\lambda }} \right) - \frac{\delta_{jk}   x_l }{|x|^{d+4-\lambda }} \right]   {\psi}(x)\myd{x} .
\end{align*}
The proof of Lemma \ref{prop:Fourier-transform-Gaussian-derivatives} (3) is completed by analytic continuation.
\end{proof}

\section{Proof of Lemma \ref{lem:de-Rham}} \label{sec:deRham}

This section is devoted to proving Lemma \ref{lem:de-Rham}.  It is a classical result due to de Rham \cite{dR73} that if a distribution $\boldu$ satisfies
\begin{equation}\label{eq:deRham}
\action{\boldu,\Phi}=0\quad \text{ for all } \Phi \in C_{c,\sigma}^\infty,
\end{equation}
then $\boldu=\nabla p$ for some distribution $p$. We prove that if $\boldu $ is a tempered distribution, then $p$ is also a tempered distribution by following the idea of Wang \cite{W93}. It will be first  shown  (see Lemma \ref{lem:representation} below) that  for each $g\in \mathscr{S}(\mathbb{R}^d)$ with $\int_{\mathbb{R}^d} g \myd{x}=0$, there exists  $\boldw_g \in \mathscr{S}(\mathbb{R}^d;\mathbb{R}^d)$ such that
\[	\Div \boldw_g = g  \quad \text{in }
\mathbb{R}^d .\]
The following lemma is necessary to show that $\boldw_g \in \mathscr{S}(\mathbb{R}^d;\mathbb{R}^d)$ and the mapping $g \mapsto \boldw_g$ is linear and continuous.

\begin{lem}\label{lem:schwartz-integral}
For a fixed $1\leq j\leq d$, let $\mathscr{S}_{\#}^{(j)}$ be the subspace of $\mathscr{S}$ that consists of all $u\in \mathscr{S}$ satisfying
\[	\int_{-\infty}^\infty  u(x_1,\dots,x_{j-1},t,x_{j+1},\dots,x_d)\myd{t}=0 \]
for all $(x_1, ..., x_{j-1}, x_{j+1}, ..., x_d)\in \mathbb{R}^{d-1}$. For each  $u\in \mathscr{S}_{\#}^{(j)}$, we define $T^{(j)}(u):\mathbb{R}^d\rightarrow \mathbb{R}$ by
\[	T^{(j)}(u)(x)=\int_{-\infty}^{x_j} u(x_1,\dots,x_{j-1},t,x_{j+1},\dots,x_d)\myd{t}\quad \text{for all } x\in \mathbb{R}^d. \]
Then $T^{(j)}$ is a linear operator from $\mathscr{S}_{\#}^{(j)}$ into $\mathscr{S}$.  Moreover, if $u_k \in \mathscr{S}_{\#}^{(j)}$ for each $k\in \mathbb{N}$  and $u_k\rightarrow 0$ in $\mathscr{S}$, then $T^{(j)}(u_k)\rightarrow 0$ in $\mathscr{S}$.
\end{lem}

\begin{proof} We prove the lemma only for the case when $d\geq 2$ and $j=1$.

Suppose that
$u\in \mathscr{S}_{\#}^{(1)}$ and $v=T^{(1)}(u)$.
Obviously, $v$ is smooth on $\mathbb{R}^d$. Moreover for all $\beta=(\beta_1,\beta')\in \mathbb{N}_0^d$, we have
\[	D^\beta v(x) =\begin{dcases}
\int_{-\infty}^{x_1} D^{\beta'}_{x'} u(t,x')\myd{t}&\quad\text{if } \beta_1=0,\\
D^{\beta_1-1}_1 D^{\beta'}_{x'} u(x_1,x') &\quad \text{if } \beta_1>0.
\end{dcases} \]
Hence to show that $v \in \mathscr{S}$, it remains  to show that
$$
\sup_{x \in \mathbb{R}^d} \left| x^\alpha D^{\beta'}_{x'} v(x) \right| < \infty
$$
for every $\beta'\in\mathbb{N}_0^{d-1}$ and $\alpha \in \mathbb{N}_0^d$.

Let $\beta' \in \mathbb{N}_0^{d-1}$ and $\alpha = (\alpha_1 , \alpha ') \in \mathbb{N}_0^d$ be fixed. Then since
\[	\int_{-\infty}^{\infty} D^{\beta'}_{x'} u(t,x')\myd{t}=D^{\beta'}_{x'} \int_{-\infty}^{\infty} u(t,x')\myd{t}=0,\]
 it follows that
\[	x^\alpha D^{\beta'}_{x'} v(x) =x_1^{\alpha_1} \int_{-\infty}^{x_1} (x')^{\alpha'} D^{\beta'}_{x'} u(t,x')\myd{t}=-x_1^{\alpha_1}\int_{x_1}^\infty (x')^{\alpha'} D^\beta_{x'} u(t,x') \myd{t}  \]
for all $x=(x_1 , x') \in \mathbb{R}^d$.
 On the other hand, since $u\in \mathscr{S}$, there exists a constant $C $ such that
\[		\sup_{(t,x')\in \mathbb{R}^d} |(1+|t|)^{\alpha_1+2} (x')^{\alpha'} D^{\beta'}_{x'} u(t,x')|\leq C .	\]
Hence for all $x=(x_1 , x') \in \mathbb{R}^d$ with $x_1 \geq 0$, we have
\begin{align*}	
\left|x^{\alpha}  D^{\beta'}_{x'} v(x)\right|&\leq x_1^{\alpha_1}\int_{x_1}^\infty \left|(x')^{\alpha'} D^{\beta'}_{x'} u(t,x') \right| \myd{t}\\
&\leq x_1^{\alpha_1} \int_{x_1}^\infty \frac{C}{(1+|t|)^{\alpha_1+2}}dt \leq C(\alpha , \beta' ).
\end{align*}
Similarly, if $x=(x_1 , x') \in \mathbb{R}^d$ and $x_1\leq 0$, then
\[	\left|x^{\alpha} D_{x'}^{\beta'} v(x) \right| \leq | x_1 |^{\alpha_1} \int_{-\infty}^{x_1} \left|(x')^{\alpha'} D_{x'}^{\beta'} u(t,x') \right| dt \leq C(\alpha , \beta' ).	\]
This proves that $v\in \mathscr{S}$.

We have shown that $T^{(1)}$ maps $\mathscr{S}_{\#}^{(1)}$ to $\mathscr{S}$. Linearity of $T^{(1)}$ is obvious from the very definition of $T^{(1)}$. Hence to complete the proof, it remains to prove the continuity property of $T^{(1)}$.

Suppose that  $u_k\in \mathscr{S}_{\#}^{(1)}$ and $u_k\rightarrow 0$ in $\mathscr{S}$. Let $\alpha,\beta \in \mathbb{N}^d_0$ be fixed. Then since $u_k\rightarrow 0$ in $\mathscr{S}$, it   follows that if $\beta_1>0$, then
\[	x^{\alpha} D^\beta T^{(1)}(u_k)(x) = x^{\alpha} D^{\beta_1-1}_1 D^{\beta'}_{x'} u_k(x) \rightarrow 0 \]
uniformly on $\mathbb{R}^d$. Moreover,    for any $\varepsilon>0$, there exists $N \in \mathbb{N}$ such that
\[ \sup_{x \in \mathbb{R}^d} \left|(1+|x_1|)^{\alpha_1+2}(x')^{\alpha'} D^{\beta'}_{x'} u_k(x) \right|<\varepsilon \]
for all $k\geq N$.  If $x=(x_1 , x') \in \mathbb{R}^d$ and $x_1\geq 0$, then for all $k\geq N$,
\[	\left| x^{\alpha} D^{\beta'}_{x'} T^{(1)}(u_k)(x ) \right|\leq \varepsilon x_1^{\alpha_1} \int_{x_1}^\infty \frac{1}{(1+|t|)^{\alpha_1+2}}dt<C\varepsilon.\]
Similar estimates hold for $x_1\leq 0$.  This implies that $T^{(1)}(u_k)\rightarrow 0$ in $\mathscr{S}$, which completes the proof of Lemma \ref{lem:schwartz-integral}.
\end{proof}

\begin{lem}\label{lem:representation} Fix   $\phi_i \in C_c^\infty(\mathbb{R})$ with $\int_{\mathbb{R}} \phi_i\myd{x}=1$ ($i=1,2,\dots,d$) and define
$$
\phi(x)=\phi_1(x_1)\cdots \phi_d(x_d) \quad \mbox{for all}\,\, x =(x_1 , \ldots , x_d ) \in \mathbb{R}^d .
$$
Then there exists a continuous linear operator $\mathcal{B}$ from $\mathscr{S}(\mathbb{R}^d)$ to $\mathscr{S}(\mathbb{R}^d;\mathbb{R}^d)$ such that
\begin{equation}\label{eq:rep-1}
   \Div \mathcal{B}(g) =g- \phi \int_{\mathbb{R}^d} g \myd{x}\quad \mbox{in}\,\,\mathbb{R}^d
\end{equation}
for all $g \in \mathscr{S}(\mathbb{R}^d)$.
\end{lem}
\begin{proof} For the sake of notational simplicity, we mainly focus on the case $d=3$. The general case can be proved in a similar way.

 For a fixed $g\in   \mathscr{S}(\mathbb{R}^3)$, we define
\begin{equation*}\label{eq:w-integral}
\overline{g}=g-\phi\int_{\mathbb{R}^3} g \myd{x}.
\end{equation*}
Also for $1 \le j \le 3$, let $S^{(j)}(g) : \mathbb{R}^3 \to \mathbb{R}$ be defined  by
\begin{align*}
S^{(1)}(g)(x)&=\overline{g}(x)-\phi_1(x_1)\int_{\mathbb{R}} \overline{g}(s_1,x_2,x_3)\myd{s_1},\\
S^{(2)}(g)(x)&=\phi_1(x_1)\left[\int_{\mathbb{R}} \overline{g}(s_1,x_2,x_3)\myd{s_1}-\phi_2(x_2)\int_{\mathbb{R}^2} \overline{g}(s_1,s_2,x_3)\myd{s_1}ds_2 \right],\\
S^{(3)}(g)(x)&=\phi_1(x_1)\phi_2(x_2)\left[\int_{\mathbb{R}^2}
\overline{g}(s_1,s_2,x_3)\myd{s_1}ds_2 \right].
\end{align*}
Then it is easy to check that $S^{(j)}$ is a linear operator from $\mathscr{S}(\mathbb{R}^3)$ into $\mathscr{S}_{\#}^{(j)}$ and that if $g_k\rightarrow 0$ in $\mathscr{S}$, then $S^{(j)}(g_k)\rightarrow 0$ in $\mathscr{S}$.

Now, for each $g\in \mathscr{S} $, we define
\[
\mathcal{B}(g)= \left( T^{(1)}S^{(1)}(g),T^{(2)}S^{(2)}(g),T^{(3)}S^{(3)}(g) \right).
 \]
Then by Lemma \ref{lem:schwartz-integral}, $\mathcal{B}$ is a continuous linear operator from $\mathscr{S}(\mathbb{R}^3)$ to $\mathscr{S}(\mathbb{R}^3;\mathbb{R}^3)$.  Moreover, by a direct computation,
\[	\Div\mathcal{B}(g)=S^{(1)}(g)+S^{(2)}(g)+S^{(3)}(g)=\overline{g}
\quad \mbox{in}\,\,\mathbb{R}^3 .\]
This completes the proof of Lemma \ref{lem:representation}.
\end{proof}

\begin{proof}[Proof of Lemma \ref{lem:de-Rham}]
Suppose that $\boldu \in \mathscr{S}'$   satisfies  $\action{\boldu,\Phi}=0$ for all $\Phi \in C_{c,\sigma}^\infty$. Then by a simple density argument, we have
\begin{equation}\label{u-div-cond}
\action{\boldu,\Phi}=0 \quad\mbox{for all } \Phi \in \mathscr{S} \,\,\mbox{with} \,\,\Div \Phi=0 .
\end{equation}

Let $\mathcal{B} : \mathscr{S}(\mathbb{R}^d) \to \mathscr{S}(\mathbb{R}^d;\mathbb{R}^d)$ be the   operator defined in  Lemma \ref{lem:representation}.
Then since $\mathcal{B}$ is linear and continuous, the mapping
\[	g \mapsto -\inner{\boldu,\mathcal{B}(g)} \]
defines a continuous linear functional $p$ on $ \mathscr{S}$.

Given $\Phi \in \mathscr{S}$, we define $g = \Div \Phi$ and $\Psi = \mathcal{B}(g) -\Phi$. Then $\Psi \in \mathscr{S}$. Moreover, since $\int_{\mathbb{R}^d} g \myd{x} =0$, we have
\[	\Div \Psi = \Div\mathcal{B}(g) - \Div \Phi =   0 \quad \mbox{in}\,\,\mathbb{R}^d .\]
Hence it follows from (\ref{u-div-cond}) that $\action{\boldu,\Psi} =0$.
Therefore,
\[
\action{\boldu,\Phi}=    \action{\boldu,\mathcal{B}(g)} = - \action{p,g} = - \action{p,\Div \Phi}.
\]
Since $\Phi \in \mathscr{S}$ is arbitrary, we deduce that $\boldu=\nabla p$.
  This completes the proof of Lemma \ref{lem:de-Rham}.
\end{proof}

\bibliographystyle{amsplain}
\providecommand{\bysame}{\leavevmode\hbox to3em{\hrulefill}\thinspace}
\providecommand{\MR}{\relax\ifhmode\unskip\space\fi MR }
\providecommand{\MRhref}[2]{%
  \href{http://www.ams.org/mathscinet-getitem?mr=#1}{#2}
}
\providecommand{\href}[2]{#2}

\end{document}